\documentclass[a4paper,12pt,reqno]{amsart}

\usepackage[normalem]{ulem}

\usepackage{fullpage}
\usepackage{amssymb}
\usepackage{amsmath}
\usepackage{mathtools}
\usepackage{enumitem}
\usepackage{mathrsfs}
\usepackage[all]{xy}
\usepackage{mathtools}
\usepackage{hyperref}
\usepackage{url}
\usepackage{bbm}
\usepackage{longtable}
\usepackage{dsfont}
\usepackage{upgreek}
\usepackage{lmodern}
\usepackage[OT2, T1]{fontenc}

\DeclareSymbolFont{cyrletters}{OT2}{wncyr}{m}{n}
\usepackage[british]{babel}

\usepackage[margin=0.75in]{geometry}
\numberwithin{equation}{section} \numberwithin{figure}{section}

\DeclareMathOperator*{\Flatsum}{\sum{}^{\flat}}

\DeclareSymbolFont{cyrletters}{OT2}{wncyr}{m}{n}
\DeclareMathSymbol{\Sha}{\mathalpha}{cyrletters}{"58}
\DeclareMathSymbol{\Be}{\mathalpha}{cyrletters}{"42}

\renewcommand\P{\mathbb{P}}
\newcommand\Z{\mathbb{Z}}
\newcommand\N{\mathbb{N}}
\newcommand\Q{\mathbb{Q}}
\newcommand\R{\mathbb{R}}
\newcommand\C{\mathbb{C}}

\renewcommand{\b}{\mathbf}

\renewcommand{\gcd}{\textrm{gcd}} 
\renewcommand{\leq}{\leqslant}
\renewcommand{\geq}{\geqslant}
\renewcommand{\#}{\sharp}

\newcommand{\p}{\mathfrak{p}}

\newtheorem{lemma}{Lemma}

\newtheorem{theorem}[lemma]{Theorem}
\newtheorem{proposition}[lemma]{Proposition}

\theoremstyle{definition}

\newtheorem{remark}[lemma]{Remark}

\usepackage[usenames,dvipsnames]{color}

\numberwithin{lemma}{section}

\title{Asymptotics for Local Solubility of Diagonal Quadrics Over a Split Quadric Surface}

\author{Cameron Wilson} 
\address{Department of Mathematics \\
University of Glasgow \\ G12~8QQ United Kingdom}
\email{c.wilson.6@research.gla.ac.uk}

\subjclass[2020] {
11A25, 
11D09, 
11G99, 
11L40, 
11N36, 
\bf{14G05}} 
\date{\today}

\begin{document}

\begin{abstract}
    We prove asymptotics for the density of everywhere locally soluble diagonal quadric surfaces parameterised by rational points on the split quadric surface $y_0y_1=y_2y_3$, which do not satisfy $-y_0y_2=\square$ nor $-y_0y_3=\square$.
\end{abstract}

\maketitle

\setcounter{tocdepth}{1}
\tableofcontents

\section{Introduction}
Let $V\subset\P^3\times\P^3$ be the variety cut out by the equations
\begin{equation}\label{DQPBSQ}
y_0x_0^2 + y_1x_1^2 + y_2x_2^2 + y_3x_3^2 = 0\;\text{and}\;y_0y_1=y_2y_3,
\end{equation}
let $U\subset\P^3$ be the projective variety cut out by the quadric
\[
y_0y_1=y_2y_3
\]
and let $\pi:V\rightarrow U$ be the dominant morphism sending $([x_0:x_1:x_2:x_3],[y_0:y_1:y_2:y_3])$ to the point $[y_0:y_1:y_2:y_3]$. Write $\widetilde{V}\rightarrow V$ for a desingularisation of $V$ and write $\widetilde{\pi}:\widetilde{V}\rightarrow U$ for its composition with $\pi$. Furthermore, we will define $H:\mathbb{P}^3(\mathbb{Q}) \rightarrow \mathbb{R}_{\geq 0}$ for the naive Weil height on $\mathbb{P}^3(\mathbb{Q})$ and restrict it to $U$. Finally we define the following counting problem:
\begin{equation}\label{THEcountingproblem}
N(B) \coloneqq \#\left\{y\in U(\Q) : \begin{array}{
    c}  -y_0y_2\neq\square,\;-y_0y_3\neq\square \\
         \widetilde{\pi}^{-1}(y)\;\text{has a $\Q$-point}\\
         H(y)\leq B
    \end{array}\right\}.
\end{equation}
The main result of this paper is the following.
\begin{theorem}\label{MAINTHEOREM}
    As $B\rightarrow\infty$,
    \[
     N(B) = \frac{cB^2\log\log B}{\log B} + O\left(\frac{B^2\sqrt{\log\log B}}{\log B}\right)
    \]
    where $c$ is given by
    \begin{align*}
    &\frac{935}{36\pi^2}\prod_{p\neq 2}\left(1+\frac{1}{p}\right)^{-2}\left(1+\frac{2}{p}+\frac{4}{p^2}+\frac{2}{p^3}+\frac{1}{p^4}\right)\\&+\frac{25}{36\pi^2}\prod_{p\neq 2}\left(1+\frac{1}{p}\right)^{-2}\left(1+\frac{2}{p}+\frac{2\left(1+\left(\frac{-1}{p}\right)\right)}{p^2}+\frac{2}{p^3}+\frac{1}{p^4}\right)
    \end{align*}
    which is $>0$.
\end{theorem}
This problem is inspired by the work Browning, Lyczak and Sarapin who discovered the first counterexample to the natural conjectures of local solubility of fibres (see \eqref{BLSMainTheorem}). They predicted that
\begin{equation}\label{BLSconjecture}
N(B) \sim \frac{cB^2}{\log B}
\end{equation}
for some $c>0$ \cite{BLS}. This result follows a vast pool of research in recent years which aims to count members of a family of algebraic varieties which have a rational point, yet the $\log\log B$ factor appearing in our main theorem is only a recent phenomenon. Problems of this type were considered by Serre in 1990 who proved upper bounds for the family of conics over $\P^5$ or $\P^2$ \cite{SerreConics}. He proved that
\begin{equation}\label{Serregeneralconics}
\#\left\{a\in\mathbb{P}(\Q)^5: \begin{array}{c} a_0x^2+a_1y^2+a_2z^2+a_3xy+a_4xz+a_5yz=0\\ \text{has a  solution over $\Q$},\\ H(a)\leq B\end{array}\right\} \ll \frac{B^6}{\sqrt{\log B}}
\end{equation}
and
\begin{equation}\label{Serrediagonalconics}
\#\left\{a\in\mathbb{P}(\Q)^2: \begin{array}{c} a_0x^2+a_1y^2+a_2z^2=0\\ \text{has a  solution over $\Q$},\\ H(a)\leq B\end{array}\right\} \ll \frac{B^3}{(\log B)^{3/2}}.
\end{equation}
This was followed by results of Hooley, who proved matching lower bounds for the family of diagonal planar conics \cite{Hooley1}, and Guo, who proved asymptotics for the problem of diagonal planar conics whose coefficients are restricted to be square-free and pairwise co-prime \cite{Guo}. Subsequently, matching lower bounds for Serre's problem were given by Hooley for the family of general planar conics \cite{Hooley2}.

In the case of families of conics, it is known that the Hasse Principle is satisfied; this means that an variety in the family has a rational solution if and only if it has a solution over every local field, $\mathbb{Q}_p$ for $p\in\{\text{primes}\}\cup\{\infty\}$. A benefit of studying families which satisfy this principle is that the otherwise difficult problem of asking whether an equation has rational points is equivalent to the more tractable problem of asking whether or not it has a solution in each of these local fields.

Not every family satisfies the Hasse Principle --- for example it is a famous example of Selmer that the plane cubic cut out by the equation $3x^3+4y^3+5z^3=0$ has points in every local field over $\Q$ but fails to have a rational point \cite{Selmer}. For such families we can study the proportion of equations which fail the Hasse Principle, as done in the case of ternary plane cubics by Bhargava \cite{Bhargava}. Another approach is to study the simpler problem of counting equations in families which are everywhere locally soluble. This was done in great generality by Loughran and Smeets \cite{LS}: suppose that $X$ is a smooth projective variety over $\Q$ and that $\phi : X\rightarrow \P^n$ is a dominant map with geometrically integral generic fibres. Then Loughran and Smeets proved the following:
\begin{equation}\label{Loughran--Smeets}
    \#\left\{y\in\P^n(\Q) : \begin{array}{
    c} \phi^{-1}(y)\;\text{is everywhere locally soluble}\\
         H(y)\leq B
    \end{array}\right\} \ll \frac{B^{n+1}}{(\log B)^{\Delta(\phi)}},
\end{equation}
where $\Delta(\phi)$ is an invariant which records the density of fibres which are poorly behaved under the action of their Galois group (see the definition given below). This upper bound generalised those given by Serre for the families of conics over $\P^5(\Q)$ or $\P^2(\Q)$ and gave a geometric interpretation to the growth rate. They further conjectured that this upper bound is optimal when at least one fibre of $\pi$ is everywhere locally soluble and that the fibre over every codimension $1$ point has an irreducible component of multiplicity $1$. They affirmed this prediction in the cases where $\Delta(\phi)=0$. Loughran, Rome and Sofos have provided a conjecture on the leading constant, as well as its geometric interpretation, for counting problems of this type \cite{LRS}. Recently Browning, Lyczak and Smeets investigated the paucity of rationally soluble fibrations where the map $\phi$ is allowed to have multiple fibres \cite{BLSmeets}.

Another generalisation of this problem involves the geometric base of the map $\phi$ --- what happens if we replace $\P^n$ by another projective variety $Y$? It is to this problem that the current paper contributes. We now consider dominant maps $\phi:X\rightarrow Y$ between smooth projective varieties $X$ and $Y$ over $\Q$ with geometrically integral fibres admitting multiple fibres. The definition of $\Delta(\phi)$ generalises for any base $Y$ \cite[equation $(1.3)$]{BL}. Set $Y^{(1)}$ to be the collection of codimension $1$ points of $Y$. Then recall that for any $D\in Y^{(1)}$, the absolute Galois group $\textrm{Gal}(\overline{\kappa(D)}/\kappa(D))$ of the residue field of $D$ acts on the irreducible components of the reduced fibres $\pi^{-1}(D)\otimes\overline{\kappa(D)}$. Choose some finite group $\Gamma_D(\phi)$ through which the action is factored and define $\Gamma_D^{\circ}(\phi)$ to be the collection of $\gamma\in\Gamma_D(\phi)$ which fix some multiplicity $1$ irreducible component of $\pi^{-1}(D)\otimes\overline{\kappa(D)}$. We then define $\delta_D(\phi) = \#\Gamma_D^{\circ}(\phi)/\#\Gamma_D(\phi)$ and
\[
\Delta(\phi) = \sum_{D\in Y^{(1)}}(1-\delta_D(\phi)).
\]
Note that the assumption that the generic fibre is geometrically integral ensures that this sum is finite. Following the conjecture of Loughran and Smeets we may predict that
\begin{equation}\label{generalbaseprediction}
    \#\left\{y\in Y(\Q) : \begin{array}{
    c} \phi^{-1}(y)\;\text{is everywhere locally soluble}\\
         H(y)\leq B.
    \end{array}\right\} \sim \frac{c'\#\{y\in Y(\Q):H(y)\leq B\}}{(\log B)^{\Delta(\phi)}}
\end{equation}
for some constant $c'>0$. Of particular interest are the cases where $Y$ is a Fano variety over $\Q$ such that $Y(\Q)$ is Zariski dense and $Y$ satisfies Manin's conjecture with no accumulating subsets \cite{ManinandCo.}. In particular,
\begin{equation}\label{Maninconjecture}
    \#\{y\in Y(\Q):H(y)\leq B^{1/2}\} \sim c_Y B(\log B)^{\rho_Y-1}
\end{equation}
for some constant $c_Y>0$ and where $\rho_Y$ is the Picard rank of $Y$. This problem has been investigated when $Y$ is a quadric of dimension $\geq 3$ by Browning and Loughran \cite{BL}, where the upper bound \eqref{generalbaseprediction} is proven. Furthermore, when $Y$ is a quadratic form of rank $\geq 5$ and all fibres over codimension $1$ points of $Y$ are split (ensuring that $\Delta(\phi)=0$), Browning and Heath-Brown proved that a positive proportion of fibres are everywhere locally soluble \cite{BH-B}. In these cases $\rho_Y=1$. Other examples include fibrations over algebraic groups, which have been studied by Loughran \cite{L}, and Loughran, Takloo-Bighash, and Tanimoto \cite{LTbT}, and fibrations over hypersurfaces by Sofos and Visse-Martindale \cite{SofosVisse}.

The prediction \eqref{generalbaseprediction} fails, however, as discovered in a recent paper of Browning, Lyczak and Sarapin \cite{BLS}. They showed that when the base is the quadric surface $U\subset\P^3$ cut out by the equation $y_0y_1=y_2y_3$ and the fibres under the map $\pi$ are diagonal quadric surfaces,
\begin{equation}\label{BLSMainTheorem}
B^2 \ll \#\left\{y\in\widetilde{\pi}(\widetilde{V})(\Q) :
         \widetilde{\pi}^{-1}(y)\;\text{has a $\Q$-point},\;
         H(y)\leq B \right\} \ll B^2,
\end{equation}
where $\widetilde{V}\subset\P^3\times\P^3$ is the desingularisation of $V$ given in \eqref{DQPBSQ} and $\widetilde{\pi}$ is the composition of this desingularisation and $\pi$, which gives a dominant morphism from $\widetilde{V}$ onto $U$. In this case $\rho_U=2$ and $\Delta(\pi)=2$ so that $\eqref{generalbaseprediction}$ predicts a growth rate of $\frac{B^2}{\log B}$. The anomaly in this particular case may be seen to arise from the presence of the thin set of points $\{y\in U(\Q):-y_0y_2=\square\;\text{or}\;-y_0y_3=\square\}$ whose fibres each have a rational point and whose contribution to the counting problem is $B^2$. Following Peyre's modern reformulation of the Manin conjecture \cite{Peyre} one is naturally led to conjecture \eqref{BLSconjecture} --- that prediction \eqref{generalbaseprediction} should hold in this case after the removal of this thin set.

It is therefore surprising that the asymptotic formula given in Theorem \ref{MAINTHEOREM} exhibits an extra growth factor of order $\log\log B$. Prior to this paper the only $\log\log B$ factor occurring in similar Manin conjecture problems was seen by using a product of two height functions on $\P^1(\Q)\times\P^1(\Q)$ \cite{LRS}: if $H':\P^1(\Q)\rightarrow \R_{\geq 0}$ is the naive Weil height on $\P^1(\Q)$ then it may be proven that
\[
\#\left\{(t_1,t_2)\in\Q^2:\begin{array}{l} H'(t_1)H'(t_2)\leq B,\\ \text{each $t_i$ is the sum of two squares} \end{array}\right\} \sim \frac{c'B^2\log\log B}{\log B}
\]
for some $c'>0$. A geometric interpretation of either occurrence of the $\log\log B$ factor is yet to be found.

\begin{remark}
    It should be noted that the methods of this paper may also be used to prove an asymptotic formula for the counting problem in \eqref{BLSMainTheorem}. In this situation though, some of the difficulty in the present work may be circumvented as we would not require certain error terms to be bounded as tightly as they need to be for our problem.
\end{remark}

\subsection{Character sums over hyperbolic regions}
The key idea of our proof is to use the character sum method developed by Friedlander and Iwaniec in \cite{F--I} for the problem of counting rationally soluble diagonal ternary conics with odd, square-free, and co-prime coefficients. There are two key difficulties with the application of this method in the proof of Theorem \ref{MAINTHEOREM}. First, the hyperbolic height condition obtained by parameterising the split quadric base $U$ results in regions which are oddly shaped and hyper-skewed; secondly, the implementation of the non-square conditions $-y_0y_2\neq\square$ and $-y_0y_3\neq\square$. Throughout, let $\|a_1,\ldots,a_n\|=\max\{\lvert a_1\rvert,\ldots,\lvert a_n\rvert\}$ for any $n\in\N$. After various reductions the proof will hinge on $8$ variables $k_0,l_0,\ldots,k_3,l_3$ summed over the height condition $h(\b{k},\b{l})=\|k_0l_0,k_1l_1\|\cdot\|k_2l_2,k_3l_3\|\leq B$. By decomposing $\N^{8}$ appropriately, this sum will split roughly into one of following three types:
\begin{align}
&\mathop{\sum\sum}_{\substack{h(\b{k},\b{l})\leq B \\ \|l_0,l_1,l_2,l_3\|\leq (\log B)^{1000}}}\left(\prod_{i=0}^{3}\frac{1}{\tau(k_i)\tau(l_i)}\right)\left(\frac{l_2l_3}{k_0k_1}\right)\left(\frac{l_0l_1}{k_2k_3}\right),\label{Introanalysissum1}\\ &\mathop{\sum\sum}_{\substack{h(\b{k},\b{l})\leq B \\ \|k_2,l_2,k_3,l_3\|\leq (\log B)^{1000}}}\left(\prod_{i=0}^{3}\frac{1}{\tau(k_i)\tau(l_i)}\right)\left(\frac{l_2l_3}{k_0k_1}\right)\left(\frac{l_0l_1}{k_2k_3}\right),\label{Introanalysissum2}\\
&\mathop{\sum\sum}_{\substack{h(\b{k},\b{l})\leq B \\ \|k_0,k_1\|,\|l_2,l_3\|>(\log B)^{1000}}}\left(\prod_{i=0}^{3}\frac{1}{\tau(k_i)\tau(l_i)}\right)\left(\frac{l_2l_3}{k_0k_1}\right)\left(\frac{l_0l_1}{k_2k_3}\right).\label{Introanalysissum3}
\end{align}
The last of these types is dealt with using the double oscillation of the Jacobi symbols over a hyperbolic height. This was the topic of previous work of the author \cite{Me}, and the necessary result from this paper is stated as Theorem \ref{Me} below. \\
Sums of the form \eqref{Introanalysissum1} will contribute to the main term of Theorem \ref{MAINTHEOREM} when all of the $l_i$ are $1$. Otherwise, such sums will exhibit oscillation in the Jacobi symbol, from which we would expect some cancellation, thus resulting in an error term.\\
For sums of the form \eqref{Introanalysissum2}, we should expect a contribution of order $B^2$ when we have $k_2=l_2=k_3=l_3=1$. However, it is for these sums that the $-y_0y_2\neq\square,-y_0y_3\neq\square$ conditions become vital: the ensure that no trivial summing of this type occurs. Thus we only need to consider sums of type \eqref{Introanalysissum2} when $\|k_2,l_2,k_3,l_3\|>1$, which (essentially) ensures that oscillation from the non-trivial Jacobi symbols results in the remaining contribution becoming an error term.\\
What makes this problem challenging, at least from an analytic standpoint, is that we require admissible error terms for the smaller main term $\frac{B^2\log\log B}{\log B}$ as opposed to the more natural main term $B^2$. This is particularly difficult in the hyper-skewed regions that occur due to the nature of the hyperbolic height $h(\b{k},\b{l})$, where the oscillation of Jacobi symbols has a diminished effect. The technology used to handle sums of the form \eqref{Introanalysissum2} and \eqref{Introanalysissum3} is a combination of the hyperbola method, Selberg--Delange methods, and the large sieve for quadratic characters as introduced by Heath-Brown \cite{Heath-Brown} (see also the work of Friedlander and Iwaniec \cite[Lemma $2$]{F--I}). The analysis of these sums becomes quite delicate and requires an improvement on the current large sieve results. For this reason, we have delegated the analysis of sums of the form \eqref{Introanalysissum1} and \eqref{Introanalysissum2} to a companion paper \cite{MeHooleyLargeSievePaper}, and will list the results we will require from this paper in section \S\ref{Listofhyperboliccharactersumsstuff}.

\subsection{Outline of proof of Theorem \ref{MAINTHEOREM}} The proof of Theorem \ref{MAINTHEOREM} is contained within \S\ref{GeoInput}-\S\ref{constant analysis}. In \S\ref{GeoInput} we apply the parameterisation of the quadric surface $U$ by $\P^1\times\P^1$ and then express $N(B)$ as a counting problem over the integers. As well as transforming our height condition this will change the form of the diagonal quadric fibres to the form
\[
t_0t_2x_0^2 + t_1t_3x_1^2 + t_1t_2x_2^2 + t_0t_3x_3^2 = 0.
\]
Local conditions for the solubility of general diagonal quadrics are considered throughout \S\ref{localsolutions}. Of particular note is that the real conditions result in $N(B)$ is being split into two similar but separate counting problems, $N_1(B)$ and $N_2(B)$ depending on the sign of the coefficients.

The next step is to reduce to odd, square-free and co-prime variables, apply the Hasse Principle to express indicator function of the diagonal quadrics having a rational point as a sum over Jacobi symbols in these new variables and sum over them. This is the content of \S\ref{reductions} and \S\ref{HasseprincipleApp}. With this we express $N(B)$ as a sum over Jacobi symbols involving $40$ variables. The purpose of \S\ref{deconstruction} is to decompose this expression into smaller pieces each of which contain inner sums of the form \eqref{Introanalysissum1},\eqref{Introanalysissum2} or \eqref{Introanalysissum3} and isolate the main terms and error terms.

Next we bound the error terms: in \S\ref{largeconductors} we use Theorem \ref{Me} to bound contributions involving sums of type \eqref{Introanalysissum3}; in \S\ref{smallconductors} we use Propositions \ref{symmetrictypeaverage1}, \ref{Asymmetrictypeaverage1}, \ref{symmetrictypeaverage2} and \ref{Asymmetrictypeaverage2} to bound contributions involving sums of type \eqref{Introanalysissum1} and \eqref{Introanalysissum2}.

The final error term, dealt with in \S\ref{vanishingmainterm}, is a contribution involving sums of a similar form to \eqref{Introanalysissum2}. These error terms arise from the subtleties of the non-square conditions. Through a delicate argument, these conditions ensure that, in the bad cases where $k_2=l_2=k_3=l_3=1$, the remaining variables are guaranteed to be summed over non-trivial Dirichlet characters of modulus $8$ (which emerge from the even parts of the $t_i$ and the $2$-adic solubility conditions). The oscillation from these characters ensure that such sums only contribute an error term.

The growth rate of the main term is computed in \S\ref{mainterm} using Proposition \ref{maintermproposition} and in \S\ref{constant analysis} we simplify the constant into an expression involving local densities. Finally the contributions are brought together to conclude the proof in \S\ref{Conclusion}.

\subsection{Acknowledgements} I would like to thank my Ph.D. supervisor, Efthymios Sofos, for his guidance throughout this project. Thanks are also due to Daniel Loughran, Julian Lyczak, and Nick Rome for useful discussions on the Loughran--Smeets conjecture and surrounding topics, and to Tim Browning for many helpful comments on the introduction. The code in \S\ref{APPENDIX} was written with the help of Dami\'an Gvirtz-Chen. Lastly, my gratitude is given to The Carnegie Trust for the Universities of Scotland for their sponsorship of my Ph.D.

\section{Technical Lemmas}\label{technicalstuff}
In this section, we record and prove some technical lemmas that will be heavily used moving forward.

\subsection{Large Conductor Lemmas}
We will have to deal with sums in the Jacobi symbol, $\left(\frac{n}{m}\right)$ over regions where $n$ and $m$ are large (see \S \ref{largeconductors}). This will require results on bilinear sums in the Jacobi symbol with a hyperbolic height condition, which were studied in previous work of the author \cite{Me}. The main result from this paper that is of interest to the present work, is the following:

\begin{lemma}[\cite{Me}, Theorem 1.1]\label{Me}
Let $X,z\geq 2$ and let $(a_n)$, $(b_m)$ be complex sequences supported on the odd square-free integers such that $\lvert a_n\rvert,\lvert b_m\rvert \leq 1$. Then if there exists an $\epsilon>0$ such that $z\geq X^{1/3-\epsilon}$ then
\[
\sum_{\substack{z < n,m \leq X\\ nm \leq X}} a_n b_m \Big(\frac{n}{m}\Big) \ll \frac{X^{1+\epsilon}}{z^{1/2}},
\]
where the implied constant depends at most on $\epsilon$. Furthermore, if there exists an $\epsilon>0$ such that $z\leq X^{1/3-\epsilon}$ then
\[
\sum_{\substack{z < n,m \leq X\\ nm \leq X}} a_n b_m \Big(\frac{n}{m}\Big) \ll_{\epsilon} \frac{X(\log X)^3}{z^{1/2}},
\]
where the implied constant depends at most on $\epsilon$.
\end{lemma}

\subsection{Hyperbolic character sums}\label{Listofhyperboliccharactersumsstuff}
The following results are the primary technical tools used to approximate the character sums of the forms \eqref{Introanalysissum1},\eqref{Introanalysissum2}, and \eqref{Introanalysissum3}. They are used throughout sections \S\ref{smallconductors}-\S\ref{mainterm}.

The first will be used to handle the main term.

\begin{proposition}\label{maintermproposition}
    Let $X\geq 3$, $C_1,C_2,C_3>0$ and take any $\b{q}\in(\Z/8\Z)^{*4}$. Then for any fixed odd integers $1\leq r_0,r_1,r_2,r_3\leq(\log X)^{C_1}$ and fixed integers $1\leq c_0,c_1,c_2,c_3\leq (\log X)^{C_2}$, $1\leq d_0,d_1,d_2,d_3\leq (\log X)^{C_3/2}$ we have
    \begin{align*} \mathop{\sum\sum\sum\sum}_{\substack{\|n_0c_0,n_1c_1\|\cdot\|n_2c_2,n_3c_3\|\leq X\\\|n_0d_0,n_1d_1\|,\|n_2d_2,n_3d_3\|>(\log X)^{C_3} \\ \mathrm{gcd}(n_i,r_i)=1\;\forall\;0\leq i\leq 3\\ n_i\equiv q_i\bmod{8}\;\forall\;0\leq i\leq 3}} \hspace{-10pt}\frac{1}{\tau(n_0)\tau(n_1)\tau(n_2)\tau(n_3)} &= \frac{\mathfrak{S}_2(\b{r})X^2\log\log X}{c_0c_1c_2c_3\log X} \\ &+ O_{C_1,C_2,C_3}\left(\frac{\tau(r_0)\tau(r_1)\tau(r_2)\tau(r_3)X^2\sqrt{\log\log X}}{c_0c_1c_2c_3\log X}\right)
    \end{align*}
    where the implied constant depends at most on $C_1,C_2,C_3$ and we define
    \[
    \mathfrak{S}_2(\b{r}) = \frac{4f_0^4}{\phi(8)^4\left(\prod_{p|2r_0}f_p\right)\left(\prod_{p|2r_1}f_p\right)\left(\prod_{p|2r_2}f_p\right)\left(\prod_{p|2r_3}f_p\right)}.
    \]
\end{proposition}

The next five will be used at various points to bound the error terms. For a positive integer $m$ we denote by $\psi_{m}$ the Jacobi symbol $\left(\frac{\cdot}{m}\right)$ or $\left(\frac{m}{\cdot}\right)$ generically. In the next result we will refer to the following conditions:

\begin{equation}\label{section3sumconds}
    \begin{cases}
        \|n_0d_0,n_1d_1\|,\|n_2d_2,n_3d_3\|>(\log X)^{D}\\
        \mathrm{gcd}(n_i,r_i)=1\;\forall\; 0\leq i\leq 3\\
        n_i\equiv q_i\bmod{8}\;\forall\; 0\leq i\leq 3
    \end{cases}
\end{equation}
where $D>0$ and the $r_i$ and $d_i$ are some integers and $q_i\in(\Z/8\Z)^{*}$ for each $i$.

\begin{proposition}\label{symmetrictypeaverage1}
    Let $X\geq 3$, $C_1,C_2,C_3>0$ and fix odd integers $Q_0,Q_2$ and some $\b{q}\in(\Z/8\Z)^{*4}$, $\Tilde{\b{q}}\in(\Z/8\Z)^{*4}$. Fixing some odd integers $1\leq r_0,r_1,r_2,r_3\leq(\log X)^{C_1}$ such that $\mathrm{gcd}(r_i,Q_i)=1$ for $i=0,2$ and any $1\leq c_0,c_1,c_2,c_3\leq (\log X)^{C_2}$, $1\leq d_0,d_1,d_2,d_3\leq (\log X)^{C_3/2}$ we define, for any $\b{m}\in\N^4$,
    \[
    H(X,\b{m}) = \mathop{\sum\sum\sum\sum}_{\substack{\b{n}\in\N^{4}, \|n_0m_0c_0,n_1m_1c_1\|\cdot\|n_2m_2c_2,n_3m_3c_3\|\leq X\\ \eqref{section3sumconds}}}\frac{\psi_{Q_0m_0m_1}(n_2n_3)\psi_{Q_2m_2m_3}(n_0n_1)}{\tau(n_0)\tau(n_1)\tau(n_2)\tau(n_3)},
    \]
    where we use \eqref{section3sumconds} with $D=C_3$. Then for any $C_4>0$:
    \begin{align*} \mathop{\sum\sum\sum\sum}_{\substack{\b{m}\in\N^{4},\|m_0,m_1\|,\|m_2,m_3\|\leq (\log X)^{C_3}\\ \mathrm{gcd}(m_0m_1,Q_0r_2r_3)=\mathrm{gcd}(m_2m_3,Q_1r_0r_1)=1\\ \b{m}\equiv \Tilde{\bar{q}}\bmod{8}\\ Q_0m_0m_1\;\text{and}\;Q_2m_2m_3\neq 1}}&\hspace{-15pt} \frac{\mu^2(2m_0m_1m_2m_3)\lvert H(X,\b{m})\rvert}{\tau(m_0)\tau(m_1)\tau(m_2)\tau(m_3)} \ll_{C_1,C_2,C_3,C_4}\frac{Q_0Q_2X^2}{c_0c_1c_2c_3(\log X)^{C_4}},
    \end{align*}
    where the implied constant depends at most on $C_1,C_2,C_3,C_4$.
\end{proposition}

\begin{proposition}\label{Asymmetrictypeaverage1}
    Let $X\geq 3$, $C_1,C_2>0$ be such that $(C_1\log\log X)^{C_2}>2$. Fix some odd square-free integers $Q_1,Q_2,Q_3\in\N$ such that $Q_1\leq (\log\log X)^{C_2}$, and some $\b{q}\in(\Z/8\Z)^{*4}$, $\Tilde{\b{q}}\in(\Z/8\Z)^{*2}$. Suppose $\chi_2$ and $\chi_3$ are characters modulo $Q_2$ and $Q_3$ respectively. Fixing any odd integers $1\leq r_0,r_1,r_2,r_3\leq (\log X)^{C_1}$ such that $\mathrm{gcd}(Q_1,r_0r_1r_2r_3)=\mathrm{gcd}(Q_2Q_3,r_2r_3)=1$ and fixing any $1\leq c_0,c_1,c_2,c_3\leq (\log X)^{C_2/32}$, $1\leq d_0,d_1,d_2,d_3\leq (\log X)^{C_2/4}$ we define, for any $\b{m}\in\N^2$
    \[
    H'(X,\b{m}) = \mathop{\sum\sum\sum\sum}_{\substack{\b{n}\in\N^4, \|n_0d_0,n_1d_1\|,\|n_2d_2,n_3d_3\|>(\log X)^{C_2}\\ \|n_0c_0,n_1c_1\|\cdot\|n_2m_2c_2,n_3m_3c_3\|\leq X\\ \mathrm{gcd}(n_i,2r_i)=1\;\forall\;0\leq i\leq 3\\ \b{n}\equiv \b{q}\bmod{8}}} \frac{\psi_{m_2m_3}(n_2n_3)}{\tau(n_0)\tau(n_1)\tau(n_2)\tau(n_3)}.
    \]
    Then,
    \begin{align*}
    \mathop{\sum\sum}_{\substack{\b{m}\in\N^2, \|m_2,m_3\|\leq (\log X)^{C_2}\\ \mathrm{gcd}(m_i,2Q_1Q_2Q_3r_i)=1\;\forall 2\leq i\leq 3\\ \b{m}\equiv \Tilde{\b{q}}\bmod{8}\\ Q_1m_2m_3\neq 1}}\frac{\mu^2(m_2m_3)\chi_2(m_2)\chi_3(m_3)}{\tau(m_2)\tau(m_3)}H'(X,\b{m}) \ll_{C_2} \frac{\tau(r_0)\tau(r_1)X^2}{c_0c_1c_2c_3(\log X)(\log \log X)^{C_3}}.
    \end{align*}
    where $C_3 = C_2/2-1$ and where the implied constant depends at most on $C_1$ and $C_2$.
\end{proposition}

\begin{lemma}\label{fixedconductorlemmaforvanishingmainterms}
    Let $X\geq 3$, $C_1,C_2>0$. Suppose $\chi_{0},\chi_{1},\chi_{2}$ and $\chi_{3}$ are Dirichlet characters modulo $8$ such that $\chi_i$ and $\chi_j$ are non-principal for some pair $(i,j)\in\{0,1\}\times\{2,3\}$. Then for any odd integers $1\leq r_0,r_1,r_2,r_3\leq (\log X)^{C_1}$ and any integers $1\leq c_{01},c_{23},M \leq (\log X)^{C_2}$ we have,
    \begin{equation*} \mathop{\sum\sum\sum\sum}_{\substack{\|n_0n_1c_{01},n_2n_3c_{23}\|\cdot M \leq X \\ \mathrm{gcd}(n_i,2r_i)=1\forall\;0\leq i\leq 3}}\frac{\chi_{0}(n_0)\chi_{2}(n_2)\chi_{1}(n_1)\chi_{3}(n_3)}{\tau(n_0)\tau(n_1)\tau(n_2)\tau(n_3)} \ll_{C_2} \frac{\tau(r_0)\tau(r_1)\tau(r_2)\tau(r_3)X^2}{c_{01}c_{23}M^2(\log X)}
    \end{equation*}
    where the implied constant depends at most on $C_2$.
\end{lemma}

\begin{proposition}\label{symmetrictypeaverage2}
    Let $X\geq 3$, $C_1,C_2,C_3>0$, let $Q_{02},Q_{13}$ be odd integers and take $\b{q}\in(\Z/8\Z)^{*4}$, $\Tilde{\b{q}}\in(\Z/8\Z)^{*2}$. Let $1\leq r_0,r_1,r_2,r_3\leq (\log X)^{C_1}$ be odd integers such that $\mathrm{gcd}(Q_{ij},2r_ir_j)=1$ for $(i,j)\in\{(0,2),(1,3)\}$ and any $1\leq c_0,c_1,c_2,c_3\leq (\log X)^{C_2}$. Define, for any $\b{m}\in\N^4$,
    \[
    H''(X,\b{m}) = \mathop{\sum\sum\sum\sum}_{\substack{\b{n}\in\N^{4}\\ \|n_0n_1c_0,n_2n_3c_1\|\cdot\|m_0m_1c_2,m_2m_3c_3\|\leq X\\ \mathrm{gcd}(n_i,r_i)=1\;\forall 0\leq i\leq 3\\ n_i\equiv q_i\bmod{8}\;\forall 0\leq i\leq 3}}\frac{\psi_{Q_{02}m_0m_2}(n_0n_2)\psi_{Q_{13}m_1m_2}(n_1n_3)}{\tau(n_0)\tau(n_1)\tau(n_2)\tau(n_3)}.
    \]
    Then
    \begin{align*} \mathop{\sum\sum\sum\sum}_{\substack{\b{m}\in\N^{4},\|m_0,m_1,m_2,m_3\|\leq (\log X)^{C_3}\\ \mathrm{gcd}(m_0m_2,2Q_{02}r_0r_2)=\mathrm{gcd}(m_1m_3,Q_{13}r_1r_3)=1\\ Q_{02}m_0m_2\neq 1\;\text{and}\;Q_{13}m_1m_3\neq 1\\ \b{m}\equiv \Tilde{\b{q}}\bmod{8}}}& \hspace{-10pt}\frac{\mu^2(2m_0m_1m_2m_3)\lvert H''(X,\b{m})\rvert}{\tau(m_0)\tau(m_1)\tau(m_2)\tau(m_3)}\ll_{C_1,C_2,C_3,C_4}\frac{Q_{02}Q_{13}X^2}{c_0c_1c_2c_3(\log X)^{C_4}}
    \end{align*}
    for any $C_4>0$ where the implied constant depends at most on the $C_i$.
\end{proposition}

\begin{proposition}\label{Asymmetrictypeaverage2}
    Let $X\geq 3$, $C_1,C_2>0$ and fix $q\in(\Z/8\Z)^{*4}$ and $\Tilde{q}\in(\Z/8\Z)^{*2},\Tilde{r}\in\N^2$ be vectors of odd integers. Fix odd integers $1\leq r_0,r_1,r_2,r_3,\Tilde{r}_0,\Tilde{r}_1\leq (\log X)^{C_1}$ and fix $1\leq c_{01},c_{23},\Tilde{c}_0,\Tilde{c}_1\leq (\log X)^{C_2}$. Then for any $\b{m}\in\N^{2}$ we define
    \[
    T(X,\b{m}) = \mathop{\sum\sum\sum\sum}_{\substack{\|n_0n_1c_{01},n_2n_3c_{23}\|\cdot \|m_0\Tilde{c}_0,m_1\Tilde{c}_1\| \leq X \\ \mathrm{gcd}(n_i,2r_i)=1\forall\;0\leq i\leq 3\\ n_i\equiv q_i\bmod{8}\forall\;0\leq i\leq 3}}\frac{\psi_{m_0m_1}(n_0n_2)}{\tau(n_0)\tau(n_1)\tau(n_2)\tau(n_3)}.
    \]
    Then for any $C_3>0$,
    \[
    \mathop{\sum\sum}_{\substack{\|m_0,m_1\|\leq (\log X)^{C_3}\\ m_i\equiv \Tilde{q}_i\bmod{8}\;\forall 0\leq i\leq 1\\ \mathrm{gcd}(m_i,\Tilde{r})=1\;\forall 0\leq i\leq 1\\ m_0m_1\neq 1}}\frac{\mu^2(m_0m_1)}{\tau(m_1)\tau(m_2)} |T(X,\b{m})| \ll_{C_1,C_2,C_3} \frac{\tau(r_0)\tau(r_2)X^2(\log\log X)^{1/2}}{c_{01}c_{23}\Tilde{c}_0\Tilde{c}_1(\log X)},
    \]
    where the implied constant depends at most on the $C_i$.
\end{proposition}

\subsection{Simplification Lemmas}
In this section we deal with some lemmas which will help us simplify our multivariable sums. Our first lemma will allow us to limit our count over our coefficents to one where the square parts and common factors are small (see \S \ref{simplification}).

\begin{lemma}\label{divisorlemma}
    Fix $c_0,c_1,c_2,c_3\in\mathbb{N}$. Then for all $B\geq 2$ we have,
    \[
    \#\{(t_0,t_1,t_2,t_3)\in\mathbb{N}^4:\|t_0,t_1\|\cdot\|t_2,t_3\|\leq B; c_i|t_i\;\forall \;0\leq i\leq 3\} \ll \frac{B^2(\log B)}{c_0c_1c_2c_3}
    \]
    where the implied constant is absolute.
\end{lemma}

\begin{proof}
Let $S_{\b{c}}(B)$ denote the expression on the left hand side. Then
\begin{align*}
S_{\b{c}}(B) &= \mathop{\sum\sum}_{nm\leq B}\#\{(t_0,t_1)\in\mathbb{N}^2:\|t_0,t_1\|=n,c_0|t_0,c_1|t_1\}\#\{(t_2,t_3)\in\mathbb{N}^2:\|t_2,t_3\|=m,c_2|t_2,c_3|t_3\}.
\end{align*}
Note that, if $\mathds{1}(c|k)$ denotes the indicator function for $c$ dividing $k$ then
\[
\#\{(t_0,t_1)\in\mathbb{N}^2:\|t_0,t_1\|=n,c_0|t_0,c_1|t_1\} \ll \frac{n}{c_1}\mathds{1}(c_0|n) + \frac{n}{c_0}\mathds{1}(c_1|n),
\]
and similarly
\[
\#\{(t_2,t_3)\in\mathbb{N}^2:\|t_2,t_3\|=m,c_2|t_2,c_3|t_3\} \ll \frac{m}{c_3}\mathds{1}(c_2|m) + \frac{m}{c_2}\mathds{1}(c_3|m).
\]
Therefore,
\begin{align*}
S_{\b{c}}(B) &\ll \mathop{\sum\sum}_{nm\leq B}nm\left(\frac{\mathds{1}(c_0|n)\mathds{1}(c_2|m)}{c_1c_3} + \frac{\mathds{1}(c_0|n)\mathds{1}(c_3|m)}{c_1c_2} + \frac{\mathds{1}(c_1|n)\mathds{1}(c_2|m)}{c_0c_3} + \frac{\mathds{1}(c_1|n)\mathds{1}(c_3|m)}{c_0c_2}\right).
\end{align*}
Let us look at the sum over the first term:
\begin{align*}
    \mathop{\sum\sum}_{nm\leq B} \frac{nm}{c_1c_3}\mathds{1}(c_0|n)\mathds{1}(c_2|m) &\ll \frac{B}{c_1c_3}\sum_{m\leq B}\sum_{n\leq B/m}\mathds{1}(c_0|n)\mathds{1}(c_2|m)\\
    &\ll \frac{B^2}{c_0c_1c_3}\sum_{m\leq B}\frac{1}{m}\mathds{1}(c_2|m).
\end{align*}
This is
\[
    \ll \frac{B^2}{c_0c_1c_3}\sum_{k\leq B/c_2}\frac{1}{c_2k}\ll \frac{B^2}{c_0c_1c_2c_3}\sum_{k\leq B}\frac{1}{k}\ll \frac{B^2(\log B)}{c_0c_1c_2c_3}.
\]
The sums over the other terms above are equivalent.
\end{proof}




The next lemma will allow us to get rid of terms regarding $\mu^2$, which will make analysis over four dimensions easier.

\begin{lemma}\label{removesquarefree}
Assume that $g_0,g_1,g_2,g_3:\N\rightarrow \C$ are multiplicative functions with $\lvert g_i(n)\rvert\leq 1$ for all $n\in\N$ and all $0\leq i\leq 3$. Then for all $X\geq 2$, $0\leq z,w_0,w_1\leq X^{1/4}$, and $\bf{q}\in\N^4$, $\bf{c}\in\N^4$
\begin{align*}
\mathop{\sum\sum\sum\sum}_{\substack{\b{n}\in\N^4,\|n_0,n_1\|,\|n_2,n_3\|>z\\ \|n_0c_0,n_1c_1\|\cdot\|n_2c_2,n_3c_3\|\leq X\\ \b{n}\equiv \b{q}\bmod{8}}}\mu^2(n_0n_1n_2n_3)\left(\prod_{i=0}^{3}g_i(n_i)\right) &= \sum_{r\leq w_0}\mu(r)\mathop{\sum\sum\sum\sum}_{\substack{\b{n}'\in\prod_{i=0}^{3}\left(\N\cap[1,w_1]\right)\\ p|n_0'n_1'n_2'n_3'\Rightarrow p|r\\ r^2|n_0'n_1'n_2'n_3'\\ \mathrm{gcd}(2,n_i)=1\;\forall i}}\left(\prod_{i=0}^{3}g_i(n_i') \right)G(X,\b{n}')\\
& + O\left(\frac{X^2(\log X)}{\min(w_0^{5/12},w_1^{1/3})c_0c_1c_2c_3}\right)
\end{align*}
where we have defined, for any $\b{a}\in\N^{4}$,
\[
G(X,\b{a}) = \mathop{\sum\sum\sum\sum}_{\substack{\b{n}''\in\N^4,\|n_0''a_0,n_1''a_1\|,\|n_2''a_2,n_3''a_3\|>z\\ \|a_0n_0''c_0,a_1n_1''c_1\|\cdot\|a_2n_2''c_2,a_3n_3''c_3\|\leq X\\ n_i''\equiv q_i/a_i\bmod{8}\;\forall i\\ \mathrm{gcd}(n_i'',r)=1}}\left(\prod_{i=0}^{3}g_i(n_i'')\right),
\]
and the implied constant is absolute.
\end{lemma}

\begin{proof}
    We begin by using the identity $\mu^2(n_0n_1n_2n_3) = \sum_{r^2|n_0n_1n_2n_3}\mu(r)$. Then
    \begin{align*}    \mathop{\sum\sum\sum\sum}_{\substack{\b{n}\in\N^4,\|n_0,n_1\|,\|n_2,n_3\|>z\\ \|n_0c_0,n_1c_1\|\cdot\|n_2c_2,n_3c_3\|\leq X\\ \b{n}\equiv \b{q}\bmod{8}}}\mu^2(n_0n_1n_2n_3)\left(\prod_{i=0}^{3}g_i(n_i)\right) =\sum_{r\leq X}\mu(r)\mathop{\sum\sum\sum\sum}_{\substack{\b{n}\in\N^4,\|n_0,n_1\|,\|n_2,n_3\|>z\\ \|n_0c_0,n_1c_1\|\cdot\|n_2c_2,n_3c_3\|\leq X\\ \b{n}\equiv \b{q}\bmod{8}\\ r^2|n_0n_1n_2n_3}}\left(\prod_{i=0}^{3}g_i(n_i)\right).
    \end{align*}
    Now write $n_i=n_i'n_i''$ for each $i$ where the $n_i''$ are co-prime to $r$ and all prime factors of each $n_i'$ divide $r$. This will yield
    \begin{align*}    \mathop{\sum\sum\sum\sum}_{\substack{\b{n}\in\N^4,\|n_0,n_1\|,\|n_2,n_3\|>z\\ \|n_0c_0,n_1c_1\|\cdot\|n_2c_2,n_3c_3\|\leq X\\ \b{n}\equiv \b{q}\bmod{8}}}\mu^2(n_0n_1n_2n_3)\left(\prod_{i=0}^{3}g_i(n_i)\right) = \sum_{r\leq X}\mu(r)\mathop{\sum\sum\sum\sum}_{\substack{\b{n}'\in\prod_{i=0}^{3}\left(\N\cap[1,X]\right)\\ p|n_0'n_1'n_2'n_3'\Rightarrow p|r\\ r^2|n_0'n_1'n_2'n_3'\\ \mathrm{gcd}(2,n_i)=1\;\forall i}}\left(\prod_{i=0}^{3}g_i(n_i')\right)G(X,\b{n}')
    \end{align*}
    We may bound the $G(X,\b{n}')$ sum using Lemma \ref{divisorlemma}:
    \begin{align*}  G(X,\b{n}') \ll \mathop{\sum\sum\sum\sum}_{\substack{\bf{n}''\in\N^4\\ \|n_0'n_0''c_0,n_1'n_1''c_1\|\cdot\|n_2'n_2''c_2,n_3'n_3''c_3\|\leq X}}1 \ll \frac{X^2\log X}{c_0c_1c_2c_3n_0'n_1'n_2'n_3'}.
    \end{align*}
    Therefore, the contribution coming from terms where $n_i'>w_1$ for some $i$ may be seen to be:
    \begin{align}\label{removesquarefree1errorterm1}
    \begin{split}
    \sum_{r\leq X}\mu(r)\mathop{\sum\sum\sum\sum}_{\substack{\b{n}'\in\prod_{i=0}^{3}\left(\N\cap[1,X]\right)\\ p|n_0'n_1'n_2'n_3'\Rightarrow p|r\\ r^2|n_0'n_1'n_2'n_3'\\ \mathrm{gcd}(2,n_i)=1\;\forall i\\ n_i>w_1\;\text{for some $i$}}}\left(\prod_{i=0}^{3}g_i(n_i')\right)G(X,\b{n}')
    &\ll \sum_{r\leq X}\mathop{\sum\sum\sum\sum}_{\substack{\b{n}'\in\prod_{i=0}^{4}\left(\N\cap[1,X]\right)\\ p|n_0'n_1'n_2'n_3'\Rightarrow p|r\\ r^2|n_0'n_1'n_2'n_3'\\ \mathrm{gcd}(2,n_i')=1\;\forall i\\ n_i'>w_1\;\text{for some $i$}}} \frac{X^2\log X}{c_0c_1c_2c_3n_0'n_1'n_2'n_3'}\\
    &\ll \frac{X^2\log X}{c_0c_1c_2c_3}\sum_{r\leq X}\sum_{\substack{n\in \N\\p|n\Rightarrow p|r\\ r^2|n \\n>w_1}}\frac{\tau_4(n)}{n}
    \end{split}
    \end{align}
    where $\tau_4(n)$ denotes the number of ways $n$ can be written as the product of $4$ positive integers. Then, noting that $\tau_4(a)\leq \tau^3(a)\ll a^{1/12}$ for all $a\in\N$, this expression becomes:
    \[
    \ll \frac{X^2\log X}{w_1^{1/3}c_0c_1c_2c_3}\sum_{r\leq X}\sum_{\substack{n\in \N\\p|n\Rightarrow p|r\\ r^2|n \\n>w_1}}\frac{1}{n^{7/12}}.
    \]
    We now use Lemma $5.7$ from \cite{LRS} with $\epsilon = 1/12$ to determine that
    \[
    \sum_{\substack{n\in \N\\p|n\Rightarrow p|r\\ r^2|n \\n>w_1}}\frac{1}{n^{7/12}} \ll \frac{1}{r^{13/12}}.
    \]
    Therefore:
    \begin{align*}
    \sum_{r\leq X}\mu(r)\mathop{\sum\sum\sum\sum}_{\substack{\b{n}'\in\prod_{i=0}^{4}\left(\N\cap[1,X]\right)\\ p|n_0'n_1'n_2'n_3'\Rightarrow p|r\\ r^2|n_0'n_1'n_2'n_3'\\ \mathrm{gcd}(2,n_i')=1\;\forall i\\ n_i'>w_1\;\text{for some $i$}}}\left(\prod_{i=0}^{3}g_i(n_i')\right)G(X,\b{n}')
    &\ll \frac{X^2\log X}{w_1^{1/3}c_0c_1c_2c_3}\sum_{r\leq X}\frac{1}{r^{13/12}}\\ &\ll \frac{X^2\log X}{w_1^{1/3}c_0c_1c_2c_3}.
    \end{align*}
    Lastly we bound the terms $r>w_0$. Again we use Lemma \ref{divisorlemma} and the bound $\tau_4(a)\ll a^{1/12}$ to obtain:
    \begin{align}\label{removesquarefree1errorterm2}
    \sum_{r> w_0}\mu(r)\mathop{\sum\sum\sum\sum}_{\substack{\b{n}'\in\prod_{i=0}^{4}\left(\N\cap[1,w_1]\right)\\ p|n_0'n_1'n_2'n_3'\Rightarrow p|r\\ r^2|n_0'n_1'n_2'n_3'\\ \mathrm{gcd}(2,n_i')=1\;\forall i}}\left(\prod_{i=0}^{3}g_i(n_i')\right)G(X,\b{n}')
    &\ll \frac{X^2\log X}{c_0c_1c_2c_3}\sum_{r>w_0}\sum_{\substack{n\in \N\\p|n\Rightarrow p|r\\ r^2|n}}\frac{1}{n^{11/12}}.
    \end{align}
    Then, using Lemma $5.7$ of \cite{LRS} with $\epsilon = 5/12$ this will be bounded by
    \[
    \ll \frac{X^2\log X}{c_0c_1c_2c_3}\sum_{r>w_0}\frac{1}{r^{17/12}} \ll \frac{X^2\log X}{w_0^{5/12}c_0c_1c_2c_3}.
    \]
\end{proof}

\begin{lemma}\label{removesquarefree2}
Assume that $g_0,g_1,g_2,g_3:\N\rightarrow \C$ are multiplicative functions with $\lvert g_i(n)\rvert\leq 1$ for all $n\in\N$ and all $0\leq i\leq 3$. Then for all $X\geq 2$, $0\leq z,w_0,w_1\leq X^{1/4}$, and $\b{q}\in\N^4$, $c_{01},c_{23},M\in\N$ with $c_{01},c_{23},M\leq X^{1/4}$,
\begin{align*}
\mathop{\sum\sum\sum\sum}_{\substack{\b{n}\in\N^4\\ \|n_0n_1c_{01},n_2n_3c_{23}\|\cdot M\leq X\\ \b{n}\equiv \b{q}\bmod{8}}}\mu^2(n_0n_1n_2n_3)\left(\prod_{i=0}^{3}g_i(n_i)\right) &= \sum_{r\leq w_0}\mu(r)\mathop{\sum\sum\sum\sum}_{\substack{\b{n}'\in\prod_{i=0}^{3}\left(\N\cap[1,w_1]\right)\\ p|n_0'n_1'n_2'n_3'\Rightarrow p|r\\ r^2|n_0'n_1'n_2'n_3'\\ \mathrm{gcd}(2,n_i')=1\;\forall i}}\left(\prod_{i=0}^{3}g_i(n_i') \right)\Tilde{G}(X,\b{n}')\\
& + O\left(\frac{X^2(\log X)^2}{\min(w_0^{5/12},w_1^{1/3})c_{01}c_{23}M^2}\right)
\end{align*}
where we have defined, for any $\b{a}\in\N^{4}$,
\[
\Tilde{G}(X,\b{a}) = \mathop{\sum\sum\sum\sum}_{\substack{\b{n}''\in\N^4\\ \|a_0a_1n_0''n_1''c_{01},a_2a_3n_2''n_3''c_{23}\|\cdot M\leq X\\ n_i''\equiv q_i/a_i\bmod{8}\;\forall i\\ \mathrm{gcd}(n_i'',r)=1}}\left(\prod_{i=0}^{3}g_i(n_i'')\right),
\]
and the implied constant is absolute.
\end{lemma}

\begin{proof}
    As in the previous proof we use the identity $\mu^2(n_0n_1n_2n_3) = \sum_{r^2|n_0n_1n_2n_3}\mu(r)$ and then write $n_i=n_1'n_i''$ for each $i$ where each $n_i''$ are co-prime to $r$ and all prime factors of each $n_i'$ divide $r$. It follows that
    \begin{align*}
        \mathop{\sum\sum\sum\sum}_{\substack{\b{n}\in\N^4\\ \|n_0n_1c_{01},n_2n_3c_{23}\|\cdot M\leq X\\ \b{n}\equiv \b{q}\bmod{8}}}\mu^2(n_0n_1n_2n_3)\left(\prod_{i=0}^{3}g_i(n_i)\right) &= \sum_{r\leq X}\mu(r)\mathop{\sum\sum\sum\sum}_{\substack{\b{n}'\in\prod_{i=0}^{3}\left(\N\cap[1,X]\right)\\ p|n_0'n_1'n_2'n_3'\Rightarrow p|r\\ r^2|n_0'n_1'n_2'n_3'\\ \mathrm{gcd}(2,n_i')=1\;\forall i}}\left(\prod_{i=0}^{3}g_i(n_i') \right)\Tilde{G}(X,\b{n}').
    \end{align*}
    It remains to bound the large terms. First note the following bound for $\Tilde{G}(X,\b{n'})$:
    \begin{align*}
        \Tilde{G}(X,\b{n}') &\ll \left(\mathop{\sum\sum}_{\substack{n_0''n_1''\leq X/Mn_0'n_1'c_{01}}}1\right)\left(\mathop{\sum\sum}_{\substack{n_2''n_3''\leq X/Mn_2'n_3'c_{23}}}1\right)\\
        &\ll \left(\frac{X(\log X)}{n_0'n_1'c_{01}M}\right)\left(\frac{X(\log X)}{n_2'n_3'c_{23}M}\right)\\
        &\ll \frac{X^2(\log X)^2}{n_0'n_1'n_2'n_3'c_{01}c_{02}M^2}.
    \end{align*}
    Now, we sum over the large $n_i'$ and large $r$ terms as in the previous proof, yielding the same result.
\end{proof}

Finally, we note the following variant of the hyperbola method which will allow us to handle the unusual hyperbolic height conditions which arise: 

\begin{lemma}\label{hyperbolamethod}
    Let $X\geq 2$ and let $2\leq Y\leq X^{1/2}$. Fix some constants $c_0,c_1,c_2,c_3\in\N$. Then for any functions $g_0,g_1,g_2,g_3:\N\rightarrow\C$, we have
    \begin{align*}    \mathop{\sum\sum\sum\sum}_{\substack{n_0,n_1,n_2,n_3\in\N^4\\ \|n_0c_0,n_1c_1\|\cdot\|n_2c_2,n_3c_3\|\leq X}} \left(\prod_{i=0}^3 g_i(n_i)\right) =& \mathop{\sum\sum\sum\sum}_{\substack{\|n_0c_0,n_1c_1\|\leq Y\\ \|n_2c_2,n_3c_3\|\leq X/\|n_0c_0,n_1c_1\|}} \left(\prod_{i=0}^3 g_i(n_i)\right) \\+& \mathop{\sum\sum\sum\sum}_{\substack{\|n_2c_2,n_3c_3\|\leq X/Y\\ \|n_0c_0,n_1c_1\|\leq X/\|n_2c_2,n_3c_3\|}} \left(\prod_{i=0}^3 g_i(n_i)\right) \\ -&  \mathop{\sum\sum\sum\sum}_{\substack{\|n_0c_0,n_1c_1\|\leq Y\\ \|n_2c_2,n_3c_3\|\leq X/Y}} \left(\prod_{i=0}^3 g_i(n_i)\right).
    \end{align*}
\end{lemma}

\section{Geometric Input}\label{GeoInput}
Let $V\subset\P^3\times\P^3$ be the variety over $\Q$ cut out by the equations
\[
y_0x_0^2+y_1x_1^2+y_2x_2^2+y_3x_3^2=0\;\;\text{and}\;\;y_0y_1=y_2y_3,
\]
and let $\pi:V(\Q)\rightarrow U(\Q)$ be the dominant map sending $([x_0:x_1:x_2:x_3],[y_0:y_1:y_2:y_3])\in V(\Q)$ to the point $[y_0:y_1:y_2:y_3]\in U(\Q)$ where $U\subset \P^3$ is the projective variety over $\Q$ cut out by the quadric
\[
y_0y_1=y_2y_3.
\]
We then want to find asymptotics for the following quantity:
\begin{equation}
    N(B) = \#\left\{y\in\pi(V)(\Q):\begin{array}{c}
    -y_0y_2,-y_0y_3\neq\square;\\ \pi^{-1}(y)\;\;\text{has a $\Q$-point};\\H(y)\leq B \end{array}\right\},
\end{equation}
where $H$ is the naive Weil height in $\P^3(\Q)$. The variety $U$ is isomorphic to $\P^1\times\P^1$ via the rational transformation $\phi:\pi(V)\rightarrow \P^1\times\P^1$ given by
\begin{align*}
    y_0 = t_0t_2, \;\; y_1 = t_1t_3,\;\; y_2 = t_1t_2, \;\; y_3 = t_0t_3.
\end{align*}
We may then reformulate our counting function in the following way:
\begin{equation}
    N(B) = \#\left\{t\in\P^1(\Q)\times\P^1(\Q):\begin{array}{c}
    -t_0t_1,-t_2t_3\neq\square;\\ \pi^{-1}(\phi^{-1}(t))\;\;\text{has a $\Q$-point};\\H([t_0:t_1]) H([t_2:t_3])\leq B \end{array}\right\},
\end{equation}
where here $H([a:b])$ is the naive Weil height on $\P^1(\Q)$ and the fibre $\pi^{-1}(\phi^{-1}(t))$ is given by the equation
\[
t_0t_2x_0^2+t_1t_3x_1^2+t_1t_2x_2^2+t_0t_3x_3^2=0.
\]
We now wish to write this counting problem as one over the integers. This is done by noting the correspondence between $\P^1(\Q)\times\P^1(\Q)$ and $\Z^2_{\text{prim}}\times\Z^2_{\text{prim}}$ where $\Z^2_{\text{prim}}$ is the set of all coprime integer pairs $(n,m)$. For each point in $\P^1(\Q)\times\P^1(\Q)$ there are four points in $\Z^2_{\text{prim}}\times\Z^2_{\text{prim}}$ corresponding to it. Therefore our counting problem becomes
\begin{equation*}
    N(B) = \frac{1}{4}\#\left\{\b{t}\in\Z^2_{\text{prim}}\times\Z^2_{\text{prim}} : \begin{array}{c}
    t_0t_2x_0^2+t_1t_3x_1^2+t_1t_2x_2^2+t_0t_3x_3^2=0\;\; \text{has a $\Q$-point},\\-t_0t_1,-t_2t_3\neq\square,\\ \|(t_0,t_1)\|\cdot\|(t_2,t_3)\| \leq B \end{array} \right\}
\end{equation*}
where $\|(t_0,t_1)\| = \max\{\lvert t_0\rvert,\lvert t_1\rvert\}$ and $\|(t_2,t_3)\| = \max\{\lvert t_2\rvert,\lvert t_3\rvert\}$. We will henceforth write this as
\begin{equation}
N(B) = \frac{1}{4}\#\left\{\b{t}\in\Z^4 : \begin{array}{c} t_0t_2x_0^2+t_1t_3x_1^2+t_1t_2x_2^2+t_0t_3x_3^2=0 \;\;\text{has a-$\Q$ point},\\ \mathrm{gcd}(t_0,t_1)=\mathrm{gcd}(t_2,t_3)=1, -t_0t_1,-t_2t_3\neq\square,\\ \|(t_0,t_1)\|\cdot\|(t_2,t_3)\| \leq B \end{array} \right\}.
\end{equation}
To conclude this section, we consider the points $\b{t}$ such that one of the components is $0$. First, we assume $t_0=0$. Then since $\mathrm{gcd}(t_0,t_1)=1$ we must then have $t_1=\pm 1$. In this case we therefore want pairs $(t_2,t_3)$ such that $\|(t_2,t_3)\|\leq B$, $-t_2t_3\neq\square$ and $t_3x_1^2+t_2x_2^2=0$ has a $\Q$ point. For the latter to be true, however, we must have $-t_2t_3$ equal to a square, which is a contradiction. Therefore there are no points with $t_0=0$ included in the count. A symmetric argument may be given for the cases $t_1=0$, $t_2=0$ and $t_3=0$. We may therefore write our counting problem as 
\begin{equation}\label{countingproblem}
N(B) = \frac{1}{4}\#\left\{\b{t}\in(\Z\setminus\{0\})^4 : \begin{array}{c}t_0t_2x_0^2+t_1t_3x_1^2+t_1t_2x_2^2+t_0t_3x_3^2=0 \;\;\text{has a $\Q$-point},\\ \mathrm{gcd}(t_0,t_1)=\mathrm{gcd}(t_2,t_3)=1,-t_0t_1,-t_2t_3\neq\square,\\ \|(t_0,t_1)\|\cdot\|(t_2,t_3)\| \leq B\end{array} \right\}.
\end{equation}

\begin{remark}
    Note that the point where one of the $t_i=0$ correspond to points $[y_0:y_1:y_2:y_3]$ such that two of the $y_j=0$. These are the lines on the quadric surface which have singular fibres under the map $\pi$, and thus we may ignore the desingularisation used in the introduction.
\end{remark}

\section{Local Solubility}\label{localsolutions}
\subsection{Real Points}
Our first step is to guarantee that our quadrics have real points. This will occur whenever the coefficients are not all positive and not all negative. However, we may use the symmetry of our surface to ensure we are counting over purely positive integers and simplify our argument. First, we split $\R^4$ into 16 regions determined by the sign of the $t_i$. For $\b{l}\in\{0,1\}^4$ we will write $R_{\b{l}}$ to be the regions defined by the points $\b{t}\in\R^4$ such that $t_i>0$ if $l_i=0$ and $t_i<0$ if $l_i=1$. For example, $R_{(1,0,0,0)} = \{\b{t}\in \R^4 : t_0<0,\;t_1,t_2,t_3>0\}.$ We will then write
\[
N_{\b{l}}(B) = \#\left\{\b{t}\in(\Z\setminus\{0\})^4\cap R_{\b{l}} : \begin{array}{c}t_0t_2x_0^2+t_1t_3x_1^2+t_1t_2x_2^2+t_0t_3x_3^2=0\;\; \text{has a $\Q$-point};\\ \mathrm{gcd}(t_0,t_1)=\mathrm{gcd}(t_2,t_3)=1; -t_0t_1,-t_2t_3\neq\square;\\ \|(t_0,t_1)\|\cdot\|(t_2,t_3)\| \leq B \end{array} \right\}.
\]
It is clear that $N_{(0,0,0,0)}(B)=N_{(1,1,1,1)}(B)=0$ since in these cases the corresponding quadric have no real solutions. In order to streamline our argument we will prove the following:

\begin{lemma}\label{real} For notation as above we have the following,
\begin{itemize}
\item if $\sum_{i=0}^{3}l_i = 1$ or $3$ then, $N_{\b{l}}(B) = N_{(1,0,0,0)}(B)$,
\item if $\sum_{i=0}^{3}l_i = 2$ and $\b{l}\not\in \{(1,1,0,0),(0,0,1,1)\}$ then, $N_{\b{l}}(B) = N_{(1,0,1,0)}(B)$,
\item if $\b{l}\in\{(1,1,0,0),(0,0,1,1)\}$ then $N_{\b{l}}(B)=0$.
\end{itemize}
\end{lemma}

\begin{proof}
The last of these assertions is immediate from the fact that, if $\b{l}\in R_{(1,1,0,0)}\cup R_{(0,0,1,1)}$, then the coefficients of the equation
\[
t_0t_2x_0^2+t_1t_3x_1^2+t_1t_2x_2^2+t_0t_3x_3^2=0
\]
are all negative.
We now look at the first assertion. The key observation is that, if $\b{t}\in R_{\b{l}}$ where $\sum_i l_i = 1$ or $3$, then we may find a unique point $\Tilde{\b{t}}\in R_{(1,0,0,0)}$, of equal height which gives an equivalent quadric. Then the quadric defined by $\b{t}$, say $\mathcal{C}_{\b{t}}$, has a rational point if and only if the quadric defined by $\Tilde{\b{t}}$, say $\mathcal{C}_{\Tilde{\b{t}}}$, has a rational point. First suppose that $\sum_{i} l_i = 1 $. Then only one of the components is negative. If $l_0 = 1$, then the result is trivial. For a point $\b{t}=(t_0,t_1,t_2,t_3)\in R_{(0,1,0,0)}$ the corresponding quadric is
\[
\mathcal{C}_{\b{t}}: t_0t_2x_0^2+t_1t_3x_1^2+t_1t_2x_2^2+t_0t_3x_3^2=0
\]
where here, the coefficients of $x_1^2$ and $x_2^2$ are negative. We map $\b{t}$ to the point $\Tilde{\b{t}} = (t_1,t_0,t_2,t_3)\in R_{(1,0,0,0)}$. Then the quadric corresponding to $\Tilde{\b{t}}$ is
\[
\mathcal{C}_{\Tilde{\b{t}}}: t_1t_2x_0^2+t_0t_3x_1^2+t_0t_2x_2^2+t_1t_3^2=0.
\]
It is clear that $\mathcal{C}_{\Tilde{\b{t}}}$ is equivalent to $\mathcal{C}_{\b{t}}$ since we have only permuted the coefficients. This mapping from $R_{(0,1,0,0)}$ to $R_{(1,0,0,0)}$ is clearly a bijection and it is easy to show that it preserves height.
Next we consider $l_2=1$. Here, $\b{t}=(t_0,t_1,t_2,t_3)\in R_{(0,0,1,0)}$, with quadric
\[
\mathcal{C}_{\b{t}}: t_0t_2x_0^2+t_1t_3x_1^2+t_1t_2x_2^2+t_0t_3x_3^2=0,
\]
where the coefficients of $x_0^2$ and $x_2^2$, $t_0t_2$ and $t_1t_2$, are negative. We map $\b{t}$ to $\Tilde{\b{t}}=(t_2,t_3,t_0,t_1)$, which has the quadric
\[
\mathcal{C}_{\Tilde{\b{t}}}: t_0t_2x_0^2+t_1t_3x_1^2+t_3t_0x_2^2+t_2t_1x_3^2=0.
\]
Again, $\mathcal{C}_{\Tilde{\b{t}}}$ is equivalent to $\mathcal{C}_{\b{t}}$ since we have only permuted the coefficients of $x_2^2$ and $x_3^2$. This mapping is also a bijection and height preserving. Finally, if $\b{t}=(t_0,t_1,t_2,t_3)\in R_{(0,0,0,1)}$, then we map to $\Tilde{\b{t}}=(t_3,t_2,t_1,t_0)\in R_{(1,0,0,0)}$. As before this is a height preserving, bijective map such that $\mathcal{C}_{\b{t}}$ and $\mathcal{C}_{\Tilde{\b{t}}}$ are equivalent quadrics.
If $\sum_i l_i=3$, then we reduce to one of the above cases by sending $\b{t}=(t_0,t_1,t_2,t_3)\in R_{\b{l}}$ to $\Tilde{\b{t}}=(-t_0,-t_1,-t_2,-t_3)\in R_{\Tilde{\b{l}}}$ where $\sum_i \Tilde{l}_i=1$. This mapping is a height preserving bijection and furthermore $\mathcal{C}_{\b{t}}$ and $\mathcal{C}_{\Tilde{\b{t}}}$ are the same quadric. We have now proved the first assertion.
The strategy for the second assertion is the same: if $\sum_i l_i=2$ and $\b{l}\not\in\{(1,1,0,0),(0,0,1,1)\}$ we find a height and quadric preserving, bijective mapping from $R_{\b{l}}$ to $R_{(1,0,1,0)}$.
For $R_{(1,0,1,0)}$ this is trivial. For $\b{t}=(t_0,t_1,t_2,t_3)\in R_{(0,1,0,1)}$ we have
\[
\mathcal{C}_{\b{t}}: t_0t_2x_0^2+t_1t_3x_1^2+t_1t_2x_2^2+t_0t_3x_3^2=0,
\]
where the coefficients of $x_2^2$ and $x_3^2$, $t_1t_2$ and $t_0t_3$ are negative. We send $\b{t}$ to $\Tilde{\b{t}}=(t_1,t_0,t_3,t_2)\in R_{(1,0,1,0)}$. Again, this map is bijective and height preserving, and the quadric corresponding to $\Tilde{\b{t}}$ is
\[
\mathcal{C}_{\Tilde{\b{t}}}: t_1t_3x_0^2+t_0t_2x_1^2+t_0t_3x_2^2+t_1t_2x_3^2=0,
\]
which is equivalent to $\mathcal{C}_{\b{t}}$. For $\b{t}=(t_0,t_1,t_2,t_3)\in R_{(1,0,0,1)}$ we map it to $\Tilde{\b{t}}=(t_0,t_1,t_3,t_2)\in R_{(1,0,1,0)}$ and for $\b{t}=(t_0,t_1,t_2,t_3)\in R_{(0,1,1,0)}$ we map it to $\Tilde{\b{t}}=(t_1,t_0,t_2,t_3)\in R_{(1,0,1,0)}$. Both of these are easily checked to be height preserving bijections resulting in points with equivalent quadric. This completes the proof.
\end{proof}
Lemma \ref{real} allows us to rephrase our counting problem so that we only count over positive integers while ensuring all quadrics considered have real points. This is encoded in the lemma below.

\begin{proposition}\label{positivereal}
We have $N(B) = 2N_1(B) + N_2(B)$, where
\begin{align*}
N_1(B) &= \#\left\{\b{t}\in\N^4 : \begin{array}{c}-t_0t_2x_0^2+t_1t_3x_1^2+t_1t_2x_2^2-t_0t_3x_3^2=0\; \text{has a $\Q$ point};\\ \mathrm{gcd}(t_0,t_1)=\mathrm{gcd}(t_2,t_3)=1,t_0t_1\neq\square;\\ \|(t_0,t_1)\|\cdot\|(t_2,t_3)\| \leq B\end{array} \right\},
\end{align*}
and
\begin{align*}
N_2(B) &= \#\left\{\b{t}\in\N^4 :\begin{array}{c}t_0t_2x_0^2+t_1t_3x_1^2-t_1t_2x_2^2-t_0t_3x_3^2=0\; \text{has a $\Q$ point};\\ \mathrm{gcd}(t_0,t_1)=\mathrm{gcd}(t_2,t_3)=1,t_0t_1,t_2t_3\neq\square;\\ \|(t_0,t_1)\|\cdot\|(t_2,t_3)\| \leq B\end{array} \right\},
\end{align*}
\end{proposition}

\begin{proof}
We may write $4N(B) = \sum_{\b{l}\in\{0,1\}^4} N_{\b{l}}(B)$ hence by Lemma \ref{real},
\[
N(B) = \frac{1}{4} \left(8N_{(1,0,0,0)}(B)+4N_{(1,0,1,0)}(B)\right) = 2N_1(B) + N_2(B).
\]
This last equality comes from noting that $N_{(1,0,0,0)}(B) = N_1(B)$ and $N_{(1,0,1,0)}(B) = N_2(B)$.
\end{proof}

\subsection{$p$-adic Points}
Using the Hasse principle for quadrics we may equate the problem of detecting rational points to detecting $p$-adic points. We have already ensured that all quadrics we are considering have a real point and so we only need to detect $\Q_p$ points for $p$ a prime. The advantage of this is that we will be able to express solubility conditions as a sum over Jacobi symbols. To do this we define, for $(a_0,a_1,a_2,a_3)\in\Z^4$, the indicator function,
\[
\langle a_0,a_1,a_2,a_3 \rangle_p = \begin{cases}
1\;\text{if}\;\mathcal{D}_{\b{a}}\;\text{has a $\Q_p$-point}\\
0\;\text{otherwise}.
\end{cases}
\]
where $\mathcal{D}_{\b{a}}$ is the quadric defined by the equation
\[
a_0x_0^2+a_1x_1^2+a_2x^2+a_3x_3^2=0.
\]
Then we obtain the following result:
\begin{lemma}\label{oddlocalconditions}
Let $p$ be an odd prime and suppose $a_0,a_1,a_2,a_3$ are square-free, non-zero integers such that $\gcd(a_0,a_1,a_2,a_3)=1$. Then:
\begin{itemize}
    \item[(a)] If $v_p(a_0a_1a_2a_3)\neq 2$ then $\langle a_0,a_1,a_2,a_3\rangle_p = 1$;
    \item[(b)] Otherwise, if $p\mid a_i,a_j$ for any distinct $i,j\in\{0,1,2,3\}$ and $p\nmid a_ka_l$ for the distinct $k,l\in\{0,1,2,3\}\setminus\{i,j\}$ then 
    \[
    \langle a_0,a_1,a_2,a_3 \rangle_p = \frac{1}{4}\left(3+\left(\frac{-a_ka_l}{p}\right)+\left(\frac{-(a_ia_j)/p^2}{p}\right)-\left(\frac{-a_ka_l}{p}\right)\left(\frac{-(a_ia_j)/p^2}{p}\right)\right)
    \]
    where $\left(\frac{\cdot}{\cdot}\right)$ is the quadratic Jacobi symbol.
\end{itemize}
\end{lemma}

\begin{proof}
Suppose $v_p(a_0a_1a_2a_3)=1$. We may then assume without loss in generality that $p|a_0$ and $p\nmid a_1a_2a_3$. Then $\mathcal{D}_{\b{a}}$ will reduce to the smooth ternary quadratic
\[
\Tilde{\mathcal{D}}_{\Tilde{\b{a}}}: \Tilde{a}_1x_1^2+\Tilde{a}_2x_2^2+\Tilde{a}_3x_3^2 = 0
\]
over $\mathbb{F}_p$. By the Chevalley--Warning theorem, all smooth ternary quadratics over $\mathbb{F}_p$ have a non-zero point, $(\Tilde{y}_1,\Tilde{y}_2,\Tilde{y}_3)$, which we may lift to a $\Q_p$-point $(y_1,y_2,y_3)$ on the ternary quadratic $a_1x_1^2+a_2x_2^2+a_3x_3^2=0$ by Hensel's lifting lemma. Then $(0,y_1,y_2,y_3)$ will be a $\Q_p$-point on $\mathcal{D}_{\b{a}}$.\\
If $v_p(a_0a_1a_2a_3)=3$ then we assume without loss in generality that $p\mid a_0,a_1,a_2$. We may multiply the quadric $\mathcal{D}_{\b{a}}$ by $p$ and then apply the birational transformation $(x_0,x_1,x_2,x_3)\mapsto(x_0/p,x_1/p,x_2/p,x_3)$ to obtain an equivalent quadric $\mathcal{D}_{(a_0/p,a_1/p,a_2/p,pa_3)}$ which is of the form considered above. The proof for $v_p(a_0a_1a_2a_3)=0$ is similar to that for $v_p(a_0a_1a_2a_3)=1$. This completes the proof of $(a)$ as the $a_i$ are square-free integers with $\textrm{gcd}(a_0,a_1,a_2,a_3)=1$, so that $v_p(a_0a_1a_2a_3)\neq 2$ implies either $v_p(a_0a_1a_2a_3)=0,1$ or $3$.\\
Finally we turn to case $(b)$. Assume without loss in generality that $p\mid a_0,a_1$ and $p\nmid a_2a_3$. Then it is clear that there is a $\Q_p$ solution to the quadric
$\mathcal{D}_{\b{a}}$
if and only if there exists some $C\in\Q_p$ such that
\[
a_0x_0^2+a_1x_1^2 = C \;\;\text{and}\;\; a_2x_2^2+a_3x_3^2=-C.
\]
This is true if and only if the ternary quadratics
\[
a_0x_0^2+a_1x_1^2 = Cz^2\;\;\text{and}\;\; a_2x_2^2+a_3x_3^2=-Cw^2
\]
have a solution in $\Q_p$. Let $(n,m)_p$ denote the Hilbert symbol for $\Q_p$ and write $a_0=pu_0$, $a_1=pu_1$ and $C=p^{\alpha}v$ with $u_0$, $u_1$ and $v$ coprime to $p$. Then we have that $\mathcal{D}_{\b{a}}$ has a solution in $\Q_p$ if and only if
\[
\left(\frac{a_0}{C},\frac{a_1}{C}\right)_p=\left(\frac{(-u_0u_1)^{1-\alpha}}{p}\right)=1\;\;\text{and}\;\;\left(\frac{-a_2}{C},\frac{-a_3}{C}\right)_p=\left(\frac{(-a_2a_3)^{\alpha}}{p}\right)=1.
\]
If $\alpha\equiv0\bmod{2}$ then this condition simplifies to requiring that $\left(\frac{-u_0u_1}{p}\right)=1$. If $\alpha\equiv1\bmod{2}$, the condition requires that $\left(\frac{-a_2a_3}{p}\right)=1$. For the backwards direction, if $\left(\frac{-u_0u_1}{p}\right)=1$, then we may choose $C=1$; then $\alpha=0$ in the expressions above, and both equalities hold. Similarly, if $\left(\frac{-a_2a_3}{p}\right)=1$ then we may choose $C=p$, in which case $\alpha=1$ in the above expressions and both equalities above hold. Therefore, recalling that $u_0=\frac{a_0}{p}$ and $u_1=\frac{a_1}{p}$, we have proven that $\mathcal{D}_{\b{a}}$ has a solution in $\Q_p$ if and only if
\[
\left(\frac{-(a_0a_1)/p^2}{p}\right)=1\;\;\text{or}\;\;\left(\frac{-a_2a_3}{p}\right)=1.
\]
We put these two conditions together to obtain the formula
\[
\langle a_0,a_1,a_2,a_3 \rangle_p = \frac{1}{4}\left(3+\left(\frac{-a_2a_3}{p}\right)+\left(\frac{-(a_0a_1)/p^2}{p}\right)-\left(\frac{-a_2a_3}{p}\right)\left(\frac{-(a_0a_1)/p^2}{p}\right)\right).
\]
\end{proof}

\subsection{$2$-adic points}\label{2adicpoints}
Our strategy will be similar as in the previous section, however we will only deal with vectors $\b{a}=(a_0,a_1,a_2,a_3)$ such that the $a_i$ are square-free, non-zero integers with $\mathrm{gcd}(a_0,a_1,a_2,a_3)=1$ and $v_2(a_0a_1a_2a_3)=0$ or $2$. We begin by defining two sets:
\[
\mathcal{A}_1 = \{\b{q}\in(\Z/8\Z)^{*4}:\b{q}\;\text{satisfies}\;\eqref{Q2condition1}\}
\;\;\text{and}\;\;
\mathcal{A}_2 = \{\b{q}\in(\Z/8\Z)^{*4}:\b{q}\;\text{satisfies}\;\eqref{Q2condition2}\}
\]
where
\begin{equation}\label{Q2condition1}
    \begin{cases}
        q_i+q_j = 0,4\;\text{for at least one pair}\;(i,j)\in\{0,1\}\times\{2,3\}\;\text{or,}\\
        (q_0+q_1,q_2+q_3)\in\{(0,0),(2,0),(2,6),(0,6),(6,0),(6,2),(0,2)\}
    \end{cases}
\end{equation}
and
\begin{equation}\label{Q2condition2}
    \begin{cases}
    \text{There is at least one choice of $i,j,k,l$ such that}\;\{(i,j),(k,l)\}=\{(0,1),(2,3)\},\\
    \text{and some}\;v\in(\Z/8\Z)^{*}\;\text{such that}\;
        q_i+q_j=0\;\;\text{and}\;\;(q_k+v)(q_l+v)= 0\;\text{or},\\
        q_i+q_j=2v\;\;\text{and}\;\;(q_k+v)(q_l+v)=0.
    \end{cases}
\end{equation}
Then we have the following.

\begin{lemma}\label{evenlocalconditions}
    Suppose $a_0,a_1,a_2,a_3\in\N$ are square-free and non-zero satisfying the condition $\mathrm{gcd}(a_0,a_1,a_2,a_3)=1$. Then:
    \begin{itemize}
        \item[(a)] if $2\nmid a_0a_1a_2a_3$ then $\langle a_0,a_1,a_2,a_3\rangle_2 = 1$ if and only if $(a_0,a_1,a_2,a_3)$ reduces to a vector in $\mathcal{A}_1$ modulo $8$;
        \item[(b)] if $2\mid a_i,a_j$ for any distinct $i,j\in\{0,1,2,3\}$ and $2\nmid a_ka_l$ for the distinct $k,l\in\{0,1,2,3\}\setminus\{i,j\}$ then $\langle a_0,a_1,a_2,a_3 \rangle_2 = 1$ if and only if $(a_i/2,a_j/2,a_k,a_l)$ reduces to a vector in $\mathcal{A}_2$ modulo 8.
    \end{itemize}
\end{lemma}

\begin{proof}
Following the same strategy used in the proof of Lemma \ref{oddlocalconditions}$(b)$ we have that the quadric $\mathcal{D}_{\b{a}}$ has a solution in $\Q_2$ if and only if there exists a $C\in \Q_2$ such that
\begin{equation}\label{2adicHilbertSymbol}
\left(\frac{a_0}{C},\frac{a_1}{C}\right)_2 = 1\;\text{and}\;\left(-\frac{a_2}{C},-\frac{a_3}{C}\right)_2 = 1
\end{equation}
where $(\cdot,\cdot)_2$ is the Hilbert symbol over $\Q_2$. Let us first consider part $(a)$. Writing $C=2^{\alpha}v$ for a unit $v\in\Q_2$, we use the well known formulae for these Hilbert symbols (for example see Chapter $3$, Theorem $1$ of \cite{Serre}) to obtain the equivalent condition:
\begin{equation}\label{2hilbertsymbol1}
\frac{(a_0v^{-1}-1)(a_1v^{-1}-1)}{4} + \alpha\frac{(a_0^2v^{-2}+a_1^2v^{-2}-2)}{8}\equiv 0\bmod{2}
\end{equation}
and
\begin{equation}\label{2hilbertsymbol2}
\frac{(a_2v^{-1}+1)(a_3v^{-1}+1)}{4}+ \alpha\frac{(a_2^2v^{-2}+a_3^2v^{-2}-2)}{8}\equiv 0\bmod{2}.
\end{equation}
We have two cases: $\alpha \equiv 0\bmod{2}$ and $\alpha \equiv 1\bmod{2}$. In the former case we simplify to the condition
\[
\left(\frac{(a_0v^{-1}-1)(a_1v^{-1}-1)}{4}\equiv 0\bmod{2}\right)\;\;\textrm{and}\;\;
\left(\frac{(a_2v^{-1}+1)(a_3v^{-1}+1)}{4}\equiv 0\bmod{2}\right),
\]
which is equivalent to asking
\[
\left(a_0-v\equiv 0\bmod{4}\; \text{or}\; a_1-v\equiv 0\bmod{4}\right)\;\;\textrm{and}\;\;
\left(a_2+v\equiv 0\bmod{4}\; \text{or}\; a_3+v\equiv 0\bmod{4}\right).
\]
Putting these together we obtain the first condition in \eqref{Q2condition1}. Alternatively, if $(a_0,a_1,a_2,a_3)$ satisfies one of these conditions then we may set $C=a_0$ or $a_1$ to ensure that both Hilbert symbols are $1$, giving the backwards direction. Now, if $\alpha\equiv 1\bmod{2}$ we write \eqref{2hilbertsymbol1} and \eqref{2hilbertsymbol2} as
\[
2\left(a_0v^{-1}-1\right)\left(a_1v^{-1}-1\right) + \left(a_0^2v^{-2}+a_1^2v^{-2}-2\right)\equiv 0\bmod{16}
\]
and
\[
2\left(a_2v^{-1}+1\right)\left(a_3v^{-1}+1\right) + \left(a_2^2v^{-2}+a_3^2v^{-2}-2\right)\equiv 0\bmod{16}.
\]
Rearranging and collecting terms we obtain
\begin{equation}\label{First2HilbertSymbolmod16}
v^{-2}(a_0+a_1)(a_0+a_1-2v)\equiv 0\bmod{16}
\end{equation}
and
\begin{equation}\label{Second2HilbertSymbolmod16}
v^{-2}(a_2+a_3)(a_2+a_3+2v)\equiv 0\bmod{16}.
\end{equation}
Now suppose $x\in\Z$ is any even integer and $v\in\Z$ any odd integer. Then
\begin{equation}\label{mod16equation}
x(x+2v)\equiv 0\bmod{16}
\end{equation}
has a solution if and only if
\[
\frac{x}{2}\left(\frac{x}{2}+v\right)\equiv 0\bmod{4}.
\]
Since $v$ is odd, $\frac{x}{2}$ and $\left(\frac{x}{2}+v\right)$ have opposite parity. It follows that \eqref{mod16equation} has a solution if and only if $4|\frac{x}{2}$ or $4|\left(\frac{x}{2}+v\right)$. Equivalently, \eqref{mod16equation} has a solution if and only if
\[
x\equiv 0\bmod{8}\;\;\text{or}\;\;x+2v\equiv 0\bmod{8}.
\]
Substituting $x=a_0+a_1$ and $x=a_2+a_3$ into this tells us that \eqref{First2HilbertSymbolmod16} and \eqref{Second2HilbertSymbolmod16} are equivalent to
\[
(a_0+a_1)\equiv 0\bmod{8}\;\;\text{or}\;\;(a_0+a_1)\equiv 2v\bmod{8}
\]
and
\[
(a_2+a_3)\equiv 0\bmod{8}\;\;\text{or}\;\;(a_2+a_3)\equiv -2v\bmod{8}.
\]
Now $2v\equiv 2\bmod{8}$ or $2v\equiv 6\bmod{8}$ depending on whether $v\equiv 1,5\bmod{8}$ or $v\equiv 3,7\bmod{8}$. Thus we only need to consider $v\equiv 1,3\bmod{8}$. Substituting these cases into the two conditions above we obtain the second condition in \eqref{Q2condition1}. Alternatively, if one of these conditions is met then we may set $C=2v$, where $v = 1$ in the first four cases and $v=3$ in the last three cases, to ensure that the Hilbert symbols are both $1$. Thus we are done with case $(a)$.\\

For case $(b)$ we may assume without loss in generality that $2|a_0,a_1$ and $2\nmid a_2a_3$. This time we substitute $a_0 = 2u_0$, $a_1 = 2u_1$ and $C=2^{\alpha}v$, where $u_0,u_1,v$ are units in $\Q_2$, into the formulae for the Hilbert symbols. Then \eqref{2adicHilbertSymbol} becomes,
\begin{equation}\label{2hilbertsymbol3}
\frac{(u_0v^{-1}-1)(u_1v^{-1}-1)}{4} + (1-\alpha)\frac{(u_0^2v^{-2}+u_1^2v^{-2}-2)}{8}\equiv 0\bmod{2}
\end{equation}
and
\begin{equation}\label{2hilbertsymbol4}
\frac{(a_2v^{-1}+1)(a_3v^{-1}+1)}{4} + \alpha\frac{(a_2^2v^{-2}+a_3^2v^{-2}-2)}{8}\equiv 0\bmod{2}.
\end{equation}
Again we consider the cases where $\alpha$ is even and $\alpha$ is odd separately. If $\alpha$ is even then we may simplify \eqref{2hilbertsymbol3} as in the second case of $(a)$. Doing this and then combining it with \eqref{2hilbertsymbol4} we obtain the condition
\[
u_0+u_1\equiv 0,2v\bmod{8}\;\;\text{and}\;\;(a_2+v)(a_3+v)\equiv 0\bmod{8}.
\]
Since $v=1,3,5$ or $7\bmod{8}$ we split into cases. Doing this, it can be seen that $(u_0,u_1,a_2,a_3)$ must satisfy one of the conditions of \eqref{Q2condition2} with $(i,j)=(0,1)$ and $(k,l)=(2,3)$. Alternatively, if any of these are satisfied then we may choose $v=1,3,5$ or $7$ appropriately and set $C=2v$ so that each Hilbert symbol is $1$. If $\alpha$ is odd then we simplify \eqref{2hilbertsymbol4} as in case $(a)$ and combine it with \eqref{2hilbertsymbol3} to obtain the condition
\[
a_2+a_3\equiv 0,-2v\bmod{8}\;\;\text{and}\;\;(u_0-v)(u_1-v)\equiv 0\bmod{8}.
\]
Once more splitting into cases for $v$ we that $(u_0,u_1,a_2,a_3)$ must satisfy \eqref{Q2condition2} with $(i,j)=(2,3)$ and $(k,l)=(0,1)$. Finally, if either of these $8$ conditions are satisfied then we may choose the appropriate $v=1,3,5$ or $7$ such that the Hilbert symbols are $1$ by choosing $C=v$. This concludes the proof.
\end{proof}

\section{Simplification}\label{simplification}
In this section we simplify the functions $N_r(B)$ from Proposition \ref{positivereal} and express them through quadratic symbols using the Hasse Principle. Let
\begin{align*}
\delta_r = \begin{cases}
    \:\:\;1 \;&\text{if}\; r=1,\\
    -1\;&\text{if}\; r=2,
\end{cases}
\end{align*}
and define the quadrics
\begin{equation}\label{genericP1P1conic}
\mathcal{C}_{r,\b{t}}: -\delta_rt_0t_2x_0^2 + t_1t_3x_1^2 +\delta_rt_1t_2x_2^2 - t_0t_3x_3^2=0.
\end{equation}

\subsection{Reduction to square-free and co-prime coefficients}\label{reductions} To begin our simplification we remove the square parts of the $t_i$. Write $t_i=a_ib_i^2$ for $a_i$ square-free integers. Noting that $\mathcal{C}_{r,\b{t}}$ is equivalent to the quadric $\mathcal{C}_{r,\b{a}}$ we have that
\[
N_r(B) = \mathop{\sum}_{\substack{\b{b}\in\mathbb{N}^4,\;b_i\leq B^{1/2}\\\eqref{bgcdconditions}}}N_{r,\b{b}}(B),
\]
where
\[
N_{r,\b{b}}(B) = \#\left\{\b{a}\in\N^4 : \begin{array}{c}-\delta_r a_0a_2x_0^2+a_1a_3x_1^2+\delta_r a_1a_2x_2^2-a_0a_3x_3^2=0\; \text{has a $\Q$ point};\\ \mathrm{gcd}(a_0b_0,a_1b_1)=\mathrm{gcd}(a_2b_2,a_3b_3)=1; a_0a_1,\frac{1-\delta_r}{2}a_2a_3\neq1,\\ \mu^2(a_i)=1\;\forall\; 0\leq i\leq 3;\; \|a_0b_0^2,a_1b_1^2\|\cdot\|a_2b_2^2,a_3b_3^2\|\leq B\end{array} \right\},
\]
and
\begin{equation}\label{bgcdconditions}
    \mathrm{gcd}(b_0,b_1)=\mathrm{gcd}(b_2,b_3)=1
\end{equation}
Next, we want to remove any common factors of the $a_i$. Writing $a_0 = s_0m_{02}m_{03}$, $a_1 = s_1m_{12}m_{13}$, $a_2 = s_2m_{02}m_{12}$ and $a_3 = s_3m_{03}m_{13}$ where
\begin{align*}
    m_{02} = \mathrm{gcd}(a_0,a_2);\;\;
    m_{03} = \mathrm{gcd}(a_0,a_3);\;\;
    m_{12} = \mathrm{gcd}(a_1,a_2);\;\;
    m_{13} = \mathrm{gcd}(a_1,a_3),
\end{align*}
it is clear that $\mu^2(s_0s_1s_2s_3m_{02}m_{03}m_{12}m_{13})=1$, since $\gcd(a_0,a_1)=\gcd(a_2,a_3)=1$. Next we note that the quadric $\mathcal{C}_{r,\b{a}}$ is equivalent to the quadric
\[
\mathcal{C}_{r,\b{s},\b{m}}: -\delta_rs_0 s_2 m_{03} m_{12} x_0^2+s_1 s_3 m_{03} m_{12} x_1^2+\delta_rs_1 s_2 m_{02} m_{13} x_2^2-s_0 s_3 m_{02} m_{13} x_3^2 = 0.
\]
We then write
\[
N_{r,\b{b}}(B) = \mathop{\sum}_{\substack{\b{m}\in\N^4,\;m_{ij}\leq B\\\eqref{mconditions}}} N_{r,\b{b},\b{m}}(B)
\]
where
\[
N_{r,\b{b},\b{m}}(B) = \#\left\{\b{s}\in\N^4 : \mathcal{C}_{r,\b{s},\b{m}}\;\text{has a $\Q$ point};\;\eqref{gcdconditionsandsquarefree};\;\eqref{heightconditions};\;\eqref{nonsquareconditions}\right\},
\]
with
\begin{equation}\label{mconditions}
\begin{cases}
\gcd(m_{02},b_1b_3)=\gcd(m_{03},b_1b_2)=\gcd(m_{12},b_0b_3)=\gcd(m_{13},b_0b_2)=1,\\
\mu^2(m_{02}m_{03}m_{12}m_{13})=1
\end{cases}
\end{equation}

\begin{equation}\label{gcdconditionsandsquarefree} \begin{cases}
\gcd(s_0,m_{02}m_{03}m_{12}m_{13}b_1)=\gcd(s_1,m_{02}m_{03}m_{12}m_{13}b_0)=1,\\
\gcd(s_2,m_{02}m_{03}m_{12}m_{13}b_3)=\gcd(s_3,m_{02}m_{03}m_{12}m_{13}b_2)=1,\\
\mu^2(s_0s_1s_2s_3)=1,
\end{cases}
\end{equation}

\begin{equation}\label{heightconditions}    
\|s_0m_{02}m_{03}b_0^2,s_1m_{12}m_{13}b_1^2\|\cdot\|s_2m_{02}m_{12}b_2^2,s_3m_{03}m_{13}b_3^2\|\leq B,
\end{equation}

\begin{equation}\label{nonsquareconditions}
    \begin{cases}
    s_0s_1m_{02}m_{03}m_{12}m_{13}\neq 1,\\
    \frac{1-\delta_r}{2}s_2s_3m_{02}m_{03}m_{12}m_{13}\neq 1.
    \end{cases}
\end{equation}

We deal with large values of $b_i$ and $m_{ij}$ using Lemma \ref{divisorlemma} which shows that $N_{r,\b{b},\b{m}}(B)$ is
\begin{align*}
&\ll \#\left\{\b{t}\hspace{-3pt}\in\hspace{-3pt}\N^4\hspace{-3pt}: \hspace{-3pt}\|t_0,t_1\|\hspace{-3pt}\cdot\hspace{-3pt}\|t_2,t_3\|\hspace{-3pt}\leq\hspace{-3pt} B;\; b_0^2m_{03}m_{02}|t_0;\;b_1^2m_{13}m_{12}|t_1;\;b_2^2m_{02}m_{12}|t_2;\;b_3^2m_{03}m_{13}|t_3\right\}\\
&\ll \frac{B^2(\log B)}{(b_0b_1b_2b_3m_{02}m_{03}m_{12}m_{13})^2}.
\end{align*}
Thus summing over $\b{m}$ and $\b{b}$ where at least one $m_{ij}$ or $b_i$ is greater than $z_0 = (\log B)^{A}$ for some $A>0$ we obtain
\begin{equation*}
    N_r(B) = \mathop{\sum\sum}_{\substack{\b{b}\in\N^4,\b{m}\in\N^4,\eqref{bgcdconditions},\eqref{mconditions}\\b_i\leq z_0,m_{ij}\leq z_0}} N_{r,\b{b},\b{m}}(B) + O\left(\frac{B^2(\log B)}{z_0}\right).
\end{equation*}
We have now reduced the counting problem to one over square-free and pairwise co-prime positive integers. Next, we aim to remove factors of $2$. Set $\sigma_i=v_2(s_i)$ where $v_2$ is the $2$-adic valuation, and (relabelling $s_i$ to henceforth be the odd part of 
the $s_i$ above) and define
\begin{align*}
\mathcal{C}_{r,\b{s},\b{m},\boldsymbol{\sigma}}:& -\delta_r2^{\sigma_0+\sigma_2}s_0 s_2 m_{03} m_{12} x_0^2+2^{\sigma_1+\sigma_3}s_1 s_3 m_{03} m_{12} x_1^2\\&+\delta_r2^{\sigma_1+\sigma_2}s_1 s_2 m_{02} m_{13} x_2^2-2^{\sigma_0+\sigma_3}s_0 s_3 m_{02} m_{13} x_3^2 = 0.
\end{align*}
Then we write
\[
N_{r,\b{b},\b{m}}(B) = \mathop{\sum}_{\substack{\boldsymbol{\sigma}\in\{0,1\}^4,\eqref{sigmaconditions}}} N_{r,\b{b},\b{m},\boldsymbol{\sigma}}(B)
\]
where,
\[
N_{r,\b{b},\b{m},\boldsymbol{\sigma}}(B) = \#\left\{\b{s}\in\N_{\textrm{odd}}^4 : \mathcal{C}_{r,\b{s},\b{m},\boldsymbol{\sigma}}\;\text{has a $\Q$ point},\;\eqref{gcdconditionsandsquarefree2},\;\eqref{heightconditions2},\eqref{nonsquareconditions2}\right\},
\]

\begin{equation}\label{sigmaconditions}
 \begin{cases}
     \sigma_0+\sigma_1+\sigma_2+\sigma_3\leq 1,\;\;\mathrm{gcd}(2^{\sigma_0+\sigma_1+\sigma_2+\sigma_3},m_{02}m_{03}m_{12}m_{13})=1\\
     \mathrm{gcd}(2^{\sigma_0},b_1)=\mathrm{gcd}(2^{\sigma_1},b_0)=\mathrm{gcd}(2^{\sigma_2},b_3)=\mathrm{gcd}(2^{\sigma_3},b_2)=1,
 \end{cases}   
\end{equation}

\begin{equation}\label{gcdconditionsandsquarefree2} 
\begin{cases}
\gcd(s_0,2^{\sigma_1}m_{02}m_{03}m_{12}m_{13}b_1)=\gcd(s_1,2^{\sigma_0}m_{02}m_{03}m_{12}m_{13}b_0)=1,\\
\gcd(s_2,2^{\sigma_3}m_{02}m_{03}m_{12}m_{13}b_3)=\gcd(s_3,2^{\sigma_2}m_{02}m_{03}m_{12}m_{13}b_2)=1,\\
\mu^2(2s_0s_1s_2s_3)=1,
\end{cases}
\end{equation}

\begin{equation}\label{heightconditions2}
\|2^{\sigma_0}s_0m_{02}m_{03}b_0^2,2^{\sigma_1}s_1m_{12}m_{13}b_1^2\|\cdot\|2^{\sigma_2}s_2m_{02}m_{12}b_2^2,2^{\sigma_3}s_3m_{03}m_{13}b_3^2\|\leq B,
\end{equation}

\begin{equation}\label{nonsquareconditions2}
\begin{cases}
2^{\sigma_0+\sigma_1}s_0s_1m_{02}m_{03}m_{12}m_{13}\neq 1,\\
\frac{1-\delta_r}{2}2^{\sigma_2+\sigma_3}s_2s_3m_{02}m_{03}m_{12}m_{13}\neq 1.
\end{cases}
\end{equation}

We may now express our $N_{r,\b{b},\b{m},\boldsymbol{\sigma}}(B)$ as
\begin{equation}\label{simplifiedcountingproblem}
N_{r,\b{b},\b{m},\boldsymbol{\sigma}}(B) = \mathop{\sum}_{\substack{\b{s}\in\N^4,\eqref{gcdconditionsandsquarefree2}\\\eqref{heightconditions2},\;\eqref{nonsquareconditions2}}} \mu^2(2s_0s_1s_2s_3)\langle \b{s},\b{m},\boldsymbol{\sigma}\rangle_r
\end{equation}
where
\[
\langle \b{s},\b{m},\boldsymbol{\sigma}\rangle_r=\begin{cases} 1\;\;\text{if}\;\mathcal{C}_{r,\b{s},\b{m},\boldsymbol{\sigma}}\;\text{has a $\Q$-point},\\
    0\;\;\textrm{otherwise}.
\end{cases}
\]
\subsection{Application of the Hasse Principle}\label{HasseprincipleApp} We now use the Hasse principle for quadrics to write our indicator function in terms of local conditions. For each prime $p$ we define the functions
\[
\langle \b{s},\b{m},\boldsymbol{\sigma} \rangle_{r,p} = \begin{cases}    1\;\;\text{if}\;\mathcal{C}_{r,\b{s},\b{m},\boldsymbol{\sigma}}\;\text{has a $\Q_p$-point},\\
    0\;\;\textrm{otherwise}.
\end{cases}
\]
Using Lemma \ref{oddlocalconditions} we obtain the following:
\begin{lemma}\label{oddlocalconditionindicators}
Let $p$ be an odd prime. Then for any $\b{m}\in\N^4$ satisfying \eqref{mconditions} and any $\boldsymbol{\sigma}\in\{0,1\}^4$ satisfying \eqref{sigmaconditions},
\begin{itemize}
    \item[(a)] $\langle \b{s},\b{m},\boldsymbol{\sigma} \rangle_{r,p} = 1$ if $p\nmid s_0s_1s_2s_3m_{02}m_{03}m_{12}m_{13}$;
    \item[(b)]  If $p|m_{02}m_{03}m_{12}m_{13}$ then,
    $$\langle \b{s},\b{m},\boldsymbol{\sigma} \rangle_{r,p} = \frac{1}{2}\left(1+\left(\frac{\delta_r 2^{\sigma_0+\sigma_1+\sigma_2+\sigma_3}s_0s_1s_2s_3}{p}\right)\right);$$
    \item[(c)] If $p|s_0s_1$ then,
    $$\langle \b{s},\b{m},\boldsymbol{\sigma} \rangle_{r,p} = \frac{1}{2}\left(1+\left(\frac{-\delta_r 2^{\sigma_2+\sigma_3}s_2s_3m_{02}m_{03}m_{12}m_{13}}{p}\right)\right);$$
    \item[(d)] If $p|s_2s_3$ then,
    $$\langle \b{s},\b{m},\boldsymbol{\sigma} \rangle_{r,p} = \frac{1}{2}\left(1+\left(\frac{\delta_r 2^{\sigma_0+\sigma_1}s_0s_1m_{02}m_{03}m_{12}m_{13}}{p}\right)\right).$$ 
\end{itemize}
\end{lemma}

\begin{proof}
Part $(a)$ is an immediate application of part $(a)$ of Lemma \ref{oddlocalconditions}. For part $(b)$ split into two cases: $p|m_{03}m_{12}$ and $p|m_{02}m_{13}$. In the former $p$ divides the coefficients of $x_0^2$ and $x_1^2$ so that Lemma \ref{oddlocalconditions}$(b)$ will yield
    \begin{align*}
    \langle \b{s},\b{m},\boldsymbol{\sigma} \rangle_{r,p}
    &=\frac{1}{2}\left(1+\left(\frac{\delta_r 2^{\sigma_0+\sigma_1+\sigma_2+\sigma_3}s_0s_1s_2s_3}{p}\right)\right)
    \end{align*}
since the relation on the coefficients of our quadrics ensures that the Jacobi symbols from Lemma \ref{oddlocalconditions} are equal, and so the indicator functions simplify to the above. The case where $p|m_{02}m_{13}$ is dealt with in the same way and yields the same result.
Now consider part $(c)$. Here $p|s_0s_1$. Suppose first that $p|s_0$, then $p$ divides the coefficients of $x_0^2$ and $x_3^2$. Using Lemma \ref{oddlocalconditions}$(b)$, again noting that the Legendre symbols simplify, we will obtain:
    \begin{align*}
    \langle \b{s},\b{m},\boldsymbol{\sigma} \rangle_{r,p} &= \frac{1}{2}\left(1+\left(\frac{-\delta_r 2^{\sigma_2+\sigma_3}s_2s_3m_{02}m_{03}m_{12}m_{13}}{p}\right)\right).
    \end{align*}
The same result will be obtained if $p|s_1$. We also note that part $(d)$ may be obtained by the same methods, but the negative disappears since if $p|s_2$, $p$ divides the coefficients of $x_0^2$ and $x_2^2$ and one is positive while the other is negative. The same happens if $p|s_3$.
\end{proof}

Next we want to collect this information to obtain an expression for the indicator function $\langle \b{s},\b{m},\boldsymbol{\sigma} \rangle_r$ by applying the Hasse Principle. Henceforth, let $(n)_{\mathrm{odd}}$ denote the odd part of $n$. Further, for any fixed vectors $\b{m},\b{s},\b{d},\Tilde{\b{d}}\in\N^{4}$ such that $d_{ij}\Tilde{d_{ij}}=(m_{ij})_{\textrm{odd}}$, and any $\boldsymbol{\sigma}\in\{0,1\}$, we define
\[
N_{r,\b{m},\b{d},\Tilde{\b{d}},\boldsymbol{\sigma}}(\b{s},B) = \mathop{\sum\sum}_{\substack{\b{k},\b{l}\in\N^4\\k_i l_i=s_i}} \frac{(-1)^{f_r(\b{d},\b{k})}}{\tau\left(k_0l_0k_1l_1k_2l_2k_3l_3(m_{02}m_{03}m_{12}m_{13})_{\mathrm{odd}}\right)}\Theta(\b{d},\Tilde{\b{d}},\b{k},\b{l})
\]
where
\begin{equation}\label{reciprocityfactor}
    f_r(\b{d},\b{k}) = \frac{((2-\delta_r)k_0k_1k_2k_3d_{02}d_{03}d_{12}d_{13}-d_{02}d_{03}d_{12}d_{13}+k_0k_1-k_2k_3-(1-\delta_r))}{4}
\end{equation}
and
\begin{align*}
\Theta(\b{d},\Tilde{\b{d}},\b{k},\b{l},\boldsymbol{\sigma}) &= \left(\frac{2^{\sigma_0+\sigma_1+\sigma_2+\sigma_3}l_0l_1l_2l_3}{d_{02}d_{03}d_{12}d_{13}}\right)\left(\frac{2^{v_2(m_{02}m_{03}m_{12}m_{13})}}{k_0k_1k_2k_3}\right)\\ &\;\;\;\;\;\;\times\left(\frac{2^{\sigma_2+\sigma_3}l_2l_3\Tilde{d}_{02}\Tilde{d}_{03}\Tilde{d}_{12}\Tilde{d}_{13}}{k_0k_1}\right)\left(\frac{2^{\sigma_0+\sigma_1}l_0l_1\Tilde{d}_{02}\Tilde{d}_{03}\Tilde{d}_{12}\Tilde{d}_{13}}{k_2k_3}\right).
\end{align*}
We now activate the Hasse Principle and use the local conditions for odd primes given in Lemma \ref{oddlocalconditionindicators} and quadratic reciprocity to express our indicator function as a sum over Jacobi symbols.

\begin{lemma}\label{globalcondition}
Fix some $\b{b}\in\N^4$. Suppose that $\b{m}\in\N^4$ satisfies \eqref{mconditions}, that $\boldsymbol{\sigma}\in\{0,1\}^4$ satisfies \eqref{sigmaconditions} and that $\b{s}\in\N^4$ satisfies \eqref{gcdconditionsandsquarefree2}. Then $\mu^2(2^{\sigma_0+\sigma_1+\sigma_2+\sigma_3}s_0s_1s_2s_3m_{02}m_{03}m_{12}m_{13})=1$ and
\begin{equation}
    \langle \b{s},\b{m},\boldsymbol{\sigma} \rangle_r = \langle \b{s},\b{m},\boldsymbol{\sigma} \rangle_{r,2} \mathop{\sum\sum}_{\substack{\b{d},\Tilde{\b{d}}\in\N^4\\d_{ij}\Tilde{d}_{ij}=(m_{ij})_{\textrm{odd}}}} N_{r,\b{m},\b{d},\Tilde{\b{d}}}(\b{s},B).
\end{equation}
\end{lemma}

\begin{proof}
    This follows from Lemma \ref{oddlocalconditionindicators} and the Hasse Principle for quadrics. Indeed, using the Hasse Principle we obtain
    \begin{align*}
    \langle \b{s},\b{m},\boldsymbol{\sigma} \rangle_r &=\langle \b{s},\b{m},\boldsymbol{\sigma} \rangle_{r,2} \prod_{\substack{p|s_0s_1s_2s_3m_{02}m_{03}m_{12}m_{13}\\ p\neq 2}} \langle \b{s},\b{m},\boldsymbol{\sigma} \rangle_{r,p}\\
    &=\langle \b{s},\b{m},\boldsymbol{\sigma} \rangle_{r,2} \prod_{\substack{p|m_{02}m_{03}m_{12}m_{13}\\p\neq 2}}\left(\langle \b{s},\b{m},\boldsymbol{\sigma} \rangle_{r,p}\right)\prod_{\substack{p|s_0s_1\\p\neq 2}}\left(\langle \b{s},\b{m},\boldsymbol{\sigma} \rangle_{r,p} \right)\prod_{\substack{p|s_2s_3\\p\neq 2}}\left(\langle \b{s},\b{m},\boldsymbol{\sigma} \rangle_{r,p} \right).
    \end{align*}
    By now applying Lemma \ref{oddlocalconditionindicators} we obtain a factor of $2^{-\omega(s_0s_1s_2s_3(m_{02}m_{03}m_{12}m_{13})_{\mathrm{odd}})}$ where $\omega(n)$ is the number of distinct primes dividing $n$. Since $s_0s_1s_2s_3(m_{02}m_{03}m_{12}m_{13})_{\mathrm{odd}}$ is square-free by assumption, this factor is exactly equal to $\tau(s_0s_1s_2s_3(m_{02}m_{03}m_{12}m_{13})_{\mathrm{odd}})^{-1}$. Therefore:
    \begin{align*}
    \langle \b{s},\b{m},\boldsymbol{\sigma} \rangle_{r} &=\frac{\langle \b{s},\b{m},\boldsymbol{\sigma} \rangle_{r,2}}{\tau(s_0s_1s_2s_3(m_{02}m_{03}m_{12}m_{13})_{\mathrm{odd}})} \prod_{\substack{p|m_{02}m_{03}m_{12}m_{13}\\p\neq 2}}\left(1+\left(\frac{\delta_r2^{\sigma_0+\sigma_1+\sigma_2+\sigma_3}s_0s_1s_2s_3}{p}\right)\right)\\&\;\;\;\;\times \prod_{\substack{p|s_0s_1\\p\neq 2}}\left(1+\left(\frac{-\delta_r2^{\sigma_2+\sigma_3}s_2s_3m_{02}m_{03}m_{12}m_{13}}{p}\right)\right)\\&\;\;\;\;\times \prod_{\substack{p|s_2s_3\\p\neq 2}}\left(1+\left(\frac{\delta_r2^{\sigma_0+\sigma_1}s_0s_1m_{02}m_{03}m_{12}m_{13}}{p}\right)\right).
    \end{align*}
    Next we multiply out these products:
    \begin{align*}
    \langle \b{s},\b{m},\boldsymbol{\sigma} \rangle_{r} =\frac{\langle \b{s},\b{m},\boldsymbol{\sigma} \rangle_{r,2}}{\tau(s_0s_1s_2s_3(m_{02}m_{03}m_{12}m_{13})_{\mathrm{odd}})}&\mathop{\sum\sum}_{\substack{\b{d},\Tilde{\b{d}}\in\N^4\\d_{ij}\Tilde{d}_{ij}=(m_{ij})_{\mathrm{odd}}}}\left(\mathop{\sum\sum}_{\substack{\b{k},\b{l}\in\N^4\\k_i l_i=s_i}}F(\b{d},\Tilde{\b{d}},\b{k},\b{l},\boldsymbol{\sigma})\right),
    \end{align*}
    where 
    \begin{align*}
    F(\b{d},\Tilde{\b{d}},\b{k},\b{l},\boldsymbol{\sigma}) =& \left(\frac{\delta_r 2^{\sigma_0+\sigma_1+\sigma_2+\sigma_3}k_0l_0k_1l_1k_2l_2k_3l_3}{d_{02}d_{03}d_{12}d_{13}}\right)\\ &\times\left(\frac{-\delta_r 2^{\sigma_2+\sigma_3}2^{v_2(m_{02}m_{03}m_{12}m_{13})}k_2l_2k_3l_3d_{02}\Tilde{d}_{02}d_{03}\Tilde{d}_{03}d_{12}\Tilde{d}_{12}d_{13}\Tilde{d}_{13}}{k_0k_1}\right)\\ &\times\left(\frac{\delta_r 2^{\sigma_0+\sigma_1}2^{v_2(m_{02}m_{03}m_{12}m_{13})}k_0l_0k_1l_1d_{02}\Tilde{d}_{02}d_{03}\Tilde{d}_{03}d_{12}\Tilde{d}_{12}d_{13}\Tilde{d}_{13}}{k_2k_3}\right).
    \end{align*}
    We are left to show that $F(\b{d},\Tilde{\b{d}},\b{k},\b{l},\boldsymbol{\sigma}) = (-1)^{f_r(\b{d},\b{k})}\Theta(\b{d},\Tilde{\b{d}},\b{k},\b{l},\boldsymbol{\sigma})$. This follows by using quadratic reciprocity for Jacobi symbols and the fact that Jacobi symbols are multiplicative in each variable. Indeed using multiplicativity:
    \begin{align*}
    F(\b{d},\Tilde{\b{d}},\b{k},\b{l},\boldsymbol{\sigma}) &= \left(\frac{-1}{k_0k_1}\right)\left(\frac{k_0k_1k_2k_3}{d_{02}d_{03}d_{12}d_{13}}\right)\left(\frac{d_{02}d_{03}d_{12}d_{13}}{k_0k_1k_2k_3}\right)\left(\frac{k_2k_3}{k_0k_1}\right)\left(\frac{k_0k_1}{k_2k_3}\right)\\&\times\left(\frac{\delta_r}{d_{02}d_{03}d_{12}d_{13}k_0k_1k_2k_3}\right)\left(\frac{2^{\sigma_0+\sigma_1+\sigma_2+\sigma_3}l_0l_1l_2l_3}{d_{02}d_{03}d_{12}d_{13}}\right)\left(\frac{2^{v_2(m_{02}m_{03}m_{12}m_{13})}}{k_0k_1k_2k_3}\right)\\&\times\left(\frac{2^{\sigma_2+\sigma_3}l_2l_3\Tilde{d}_{02}\Tilde{d}_{03}\Tilde{d}_{12}\Tilde{d}_{13}}{k_0k_1}\right)\left(\frac{2^{\sigma_0+\sigma_1}l_0l_1\Tilde{d}_{02}\Tilde{d}_{03}\Tilde{d}_{12}\Tilde{d}_{13}}{k_2k_3}\right)\\
    & = \left(\frac{-1}{k_0k_1}\right)\left(\frac{k_0k_1k_2k_3}{d_{02}d_{03}d_{12}d_{13}}\right)\left(\frac{d_{02}d_{03}d_{12}d_{13}}{k_0k_1k_2k_3}\right)\\
    &\times\left(\frac{k_2k_3}{k_0k_1}\right)\left(\frac{k_0k_1}{k_2k_3}\right)\left(\frac{\delta_r}{d_{02}d_{03}d_{12}d_{13}k_0k_1k_2k_3}\right) \Theta(\b{d},\Tilde{\b{d}},\b{k},\b{l},\boldsymbol{\sigma}).
    \end{align*}
    Finally we apply quadratic reciprocity of Jacobi symbols which states that for odd integers $n,m$ we have
    \[
    \left(\frac{n}{m}\right)\left(\frac{m}{n}\right) = (-1)^{\frac{(n-1)(m-1)}{4}}.
    \]
    We also note that for an odd integer $n$,
    \[
    \left(\frac{-1}{n}\right) = (-1)^{\frac{(n-1)}{2}}.
    \]
    Applying these to the first six Jacobi symbols in the expression for $F(\b{d},\Tilde{\b{d}},\b{k},\b{l},\boldsymbol{\sigma})$ and collecting the powers of $-1$ will yield $(-1)^{f_r(\b{d},\b{k})}$ as required.
\end{proof}

Next we deal with the indicator function for $2$-adic points. For this we will use the conditions given by Lemma \ref{evenlocalconditions} to split our sum into arithmetic progressions modulo $8$. We will require some notation. Recall the set $\mathcal{A}_2$ defined in \S \ref{2adicpoints}. We will define twists of this set. Let $i,j,k,l$ be distinct elements of the set $\{0,1,2,3\}$. Then define
\begin{align*}
    \mathcal{A}_{i,j,k,l} = \{\b{q}\in((\Z/8\Z)^*)^4:(q_i,q_j,q_k,q_l)\in\mathcal{A}_2\}.
\end{align*}
In particular, $\mathcal{A}_{0,1,2,3} = \mathcal{A}_2$. Now for $\b{m}\in\N^4$ satisfying \eqref{mconditions} and $\boldsymbol{\sigma}\in\{0,1\}^{4}$ satisfying \eqref{sigmaconditions} define the set function
\begin{align*}
    \mathcal{A}(\b{m},\boldsymbol{\sigma}) = \begin{cases}
        \mathcal{A}_1\;\;\text{if}\;\;2\nmid m_{02}m_{03}m_{12}m_{13}\;\text{and}\;\sigma_i=0\;\forall\; i\in\{0,1,2,3\},\\
        \mathcal{A}_{0,1,2,3}\;\;\text{if}\;\;2\mid m_{03}m_{12},\;2\nmid m_{02}m_{13}\;\text{and}\;\sigma_i=0\;\forall\; i\in\{0,1,2,3\},\\
        \mathcal{A}_{2,3,0,1}\;\;\text{if}\;\;2\mid m_{02}m_{13},\;2\nmid m_{03}m_{12}\;\text{and}\;\sigma_i=0\;\forall\; i\in\{0,1,2,3\},\\
        \mathcal{A}_{0,3,1,2}\;\;\text{if}\;\;2\nmid m_{02}m_{03}m_{12}m_{13}\;\text{and}\; \sigma_0=1\;\text{and}\;\sigma_i=0\;\forall\;i\in\{0,1,2,3\}\setminus\{0\},\\
        \mathcal{A}_{1,2,0,3}\;\;\text{if}\;\;2\nmid m_{02}m_{03}m_{12}m_{13}\;\text{and}\; \sigma_1=1\;\text{and}\;\sigma_i=0\;\forall\;i\in\{0,1,2,3\}\setminus\{1\},\\
        \mathcal{A}_{0,2,1,3}\;\;\text{if}\;\;2\nmid m_{02}m_{03}m_{12}m_{13}\;\text{and}\; \sigma_2=1\;\text{and}\;\sigma_i=0\;\forall\;i\in\{0,1,2,3\}\setminus\{2\},\\
        \mathcal{A}_{1,3,0,2}\;\;\text{if}\;\;2\nmid m_{02}m_{03}m_{12}m_{13}\;\text{and}\; \sigma_3=1\;\text{and}\;\sigma_i=0\;\forall\;i\in\{0,1,2,3\}\setminus\{3\}.
    \end{cases}
\end{align*}

\begin{remark}
    To understand this notation, notice that the vectors $\b{m}\in\N^4$ and $\boldsymbol{\sigma}\in\{0,1\}$ indicate which of the coefficients are even. The value of $\mathcal{A}(\b{m},\boldsymbol{\sigma})$ therefore only orders the coefficients in the way required by Lemma \ref{evenlocalconditions}.
\end{remark}

The following two lemmas regard the solubility in $\Q_2$.

\begin{lemma}\label{2adiclemma}
Fix some $\b{m}=(m_{02},m_{03},m_{12},m_{13})\in\N^{4}$ and $\boldsymbol{\sigma}=(\sigma_0,\sigma_1,\sigma_2,\sigma_3)\in\{0,1\}^{4}$ such that the conditions \eqref{mconditions} and \eqref{sigmaconditions} hold. Then $\langle \b{s},\b{m},\boldsymbol{\sigma}\rangle_{r,2} = 1$ if and only if
\begin{equation}\label{twoadicconditions}
(-\delta_r s_0s_2(m_{03}m_{12})_{\mathrm{odd}},s_1s_3(m_{03}m_{12})_{\mathrm{odd}},\delta_r s_1s_2(m_{02}m_{13})_{\mathrm{odd}},-s_0s_3(m_{02}m_{13})_{\mathrm{odd}})\equiv \b{q}\bmod{8}
\end{equation}
for some $\b{q}\in\mathcal{A}(\b{m},\boldsymbol{\sigma})$.
\end{lemma}

\begin{proof}
By the conditions \eqref{mconditions} and \eqref{sigmaconditions}, the set $\mathcal{A}(\b{m},\boldsymbol{\sigma})$ is well defined. Then the result is immediate using Lemma \ref{evenlocalconditions}.
\end{proof}

For fixed choices of $\b{b},\b{m}\in\N^4$ and $\boldsymbol{\sigma}\in\{0,1\}^4$ satisfying \eqref{mconditions} and \eqref{sigmaconditions}, any $\b{d},\Tilde{\b{d}}\in\N^4$ such that $d_{ij}\Tilde{d}_{ij}=(m_{ij})_{\textrm{odd}}$ and any $\b{q}\in\mathcal{A}(\b{m},\boldsymbol{\sigma})$ define
\[
N_{r,\b{b},\b{m},\b{d},\Tilde{\b{d}},\boldsymbol{\sigma},\b{q}}(B) = \mathop{\sum\sum}_{\substack{\b{k},\b{l}\in\N^4\\\eqref{gcdconditionsandsquarefree3},\;\eqref{heightconditions3}\\\eqref{nonsquareconditions3},\;\eqref{twoadicconditions2}}} \frac{(-1)^{f_r(\b{d},\b{k})}\mu^2(2k_0l_0k_1l_1k_2l_2k_3l_3)}{\tau\left(k_0l_0k_1l_1k_2l_2k_3l_3\right)}\Theta(\b{d},\Tilde{\b{d}},\b{k},\b{l},\boldsymbol{\sigma})
\]
where
\begin{equation}\label{gcdconditionsandsquarefree3} \begin{cases}
\gcd(k_0l_0,2^{\sigma_1}m_{02}m_{03}m_{12}m_{13}b_1)=\gcd(k_1l_1,2^{\sigma_0}m_{02}m_{03}m_{12}m_{13}b_0)=1,\\
\gcd(k_2l_2,2^{\sigma_3}m_{02}m_{03}m_{12}m_{13}b_3)=\gcd(k_3l_3,2^{\sigma_2}m_{02}m_{03}m_{12}m_{13}b_2)=1,\\
\end{cases}
\end{equation}

\begin{equation}\label{heightconditions3}
\|2^{\sigma_0}k_0l_0m_{02}m_{03}b_0^2,2^{\sigma_1}k_1l_1m_{12}m_{13}b_1^2\|\cdot\|2^{\sigma_2}k_2l_2m_{02}m_{12}b_2^2,2^{\sigma_3}k_3l_3m_{03}m_{13}b_3^2\|\leq B,
\end{equation}

\begin{equation}\label{nonsquareconditions3}
\begin{cases}
2^{\sigma_0+\sigma_1}k_0l_0k_1l_1m_{02}m_{03}m_{12}m_{13}\neq 1,\\
\frac{1-\delta_r}{2}2^{\sigma_2+\sigma_3}k_2l_2k_3l_3m_{02}m_{03}m_{12}m_{13}\neq 1,
\end{cases}
\end{equation}
and
\begin{equation}\label{twoadicconditions2}
\begin{cases}
-\delta_r k_0l_0 k_2l_2 (m_{03}m_{12})_{\mathrm{odd}}\equiv q_0\bmod{8};\;
k_1l_1 k_3l_3 (m_{03}m_{12})_{\mathrm{odd}}\equiv q_1\bmod{8};\\
\delta_r k_1l_1 k_2l_2 (m_{02}m_{13})_{\mathrm{odd}}\equiv q_2\bmod{8};\;
-k_0l_0 k_3l_3 (m_{02}m_{13})_{\mathrm{odd}})\equiv q_3\bmod{8}.
\end{cases}
\end{equation}

\begin{lemma}\label{ConcreteSum}
For a fixed choice of $\b{b}$, $\b{m}$ and $\boldsymbol{\sigma}$ satisfying \eqref{mconditions} and \eqref{sigmaconditions} we have:
\begin{equation*}
    N_{r,\b{b},\b{m},\boldsymbol{\sigma}}(B) = \frac{1}{\tau\left((m_{02}m_{03}m_{12}m_{13})_{\mathrm{odd}}\right)}\mathop{\sum}_{\b{q}\in\mathcal{A}(\b{m},\boldsymbol{\sigma})}\mathop{\sum\sum}_{\substack{\b{d},\Tilde{\b{d}}\in\N^4\\d_{ij}\Tilde{d}_{ij}=(m_{ij})_{\mathrm{odd}}}} N_{r,\b{b},\b{m},\b{d},\Tilde{\b{d}},\boldsymbol{\sigma},\b{q}}(B).
\end{equation*}
\end{lemma}

\begin{proof}
By applying Lemma \ref{2adiclemma} we may write
\[
N_{r,\b{b},\b{m},\boldsymbol{\sigma}}(B) = \mathop{\sum}_{\b{q}\in\mathcal{A}(\b{m},\boldsymbol{\sigma})}\mathop{\sum}_{\substack{\b{s}\in\N_{\mathrm{odd}}^4,\eqref{gcdconditionsandsquarefree2}\\\eqref{heightconditions2},\eqref{nonsquareconditions2},\eqref{twoadicconditions}}}\mu^2(2s_0s_1s_2s_3) \langle \b{s},\b{m},\boldsymbol{\sigma} \rangle_r.
\]
Then applying Lemma \ref{globalcondition} and swapping the summation of $\b{s}\in\N^4$ and $\b{d},\Tilde{\b{d}}\in\N^4$ gives,
\begin{align*}
N_{r,\b{b},\b{m},\boldsymbol{\sigma}}(B) &= \mathop{\sum}_{\b{q}\in\mathcal{A}(\b{m},\boldsymbol{\sigma})}\mathop{\sum\sum}_{\substack{\b{d},\Tilde{\b{d}}\in\N^4\\d_{ij}\Tilde{d}_{ij}=(m_{ij})_{\mathrm{odd}}}} \mathop{\sum}_{\substack{\b{s}\in\N^4,\eqref{gcdconditionsandsquarefree2}\\ \eqref{heightconditions2},\eqref{nonsquareconditions2},\eqref{twoadicconditions}}}N_{r,\b{m},\b{d},\Tilde{\b{d}},\boldsymbol{\sigma}}(\b{s},B).
\end{align*}
Now observe that,
\begin{align*}
\mathop{\sum}_{\substack{\b{s}\in\N^4,\eqref{gcdconditionsandsquarefree2}\\\eqref{nonsquareconditions2},\eqref{heightconditions2},\eqref{twoadicconditions}}}\hspace{-10pt}N_{r,\b{m},\b{d},\Tilde{\b{d}},\boldsymbol{\sigma}}(\b{s},B)\hspace{-2pt} = \hspace{-2pt}\mathop{\sum}_{\substack{\b{s}\in\N^4,\eqref{gcdconditionsandsquarefree2}\\\eqref{nonsquareconditions2},\eqref{heightconditions2},\eqref{twoadicconditions}}}\hspace{-4pt}\mathop{\sum\sum}_{\substack{\b{k},\b{l}\in\N^4\\k_i l_i=s_i}} \hspace{-1pt}\frac{(-1)^{f_r(\b{d},\b{k})}\Theta(\b{d},\Tilde{\b{d}},\b{k},\b{l})}{\tau\left(k_0l_0k_1l_1k_2l_2k_3l_3(m_{02}m_{03}m_{12}m_{13})_{\mathrm{odd}}\right)}.
\end{align*}
Swapping the order of summation of the $s_i$ with the $k_i$ and $l_i$ changes \eqref{gcdconditionsandsquarefree2} to \eqref{gcdconditionsandsquarefree3}, \eqref{heightconditions2} to \eqref{heightconditions3}, \eqref{nonsquareconditions2} to \eqref{nonsquareconditions3} and \eqref{twoadicconditions} to \eqref{twoadicconditions2}. Then, using the multiplicativity of $\tau$ and condition \eqref{gcdconditionsandsquarefree3} we obtain:
\[
\mathop{\sum}_{\substack{\b{s}\in\N^4,\eqref{gcdconditionsandsquarefree2}\\\eqref{nonsquareconditions2},\eqref{heightconditions2},\eqref{twoadicconditions}}}N_{r,\b{m},\b{d},\Tilde{\b{d}},\boldsymbol{\sigma}}(\b{s},B)\hspace{-2pt} = \frac{N_{r,\b{b},\b{m},\b{d},\Tilde{\b{d}},\boldsymbol{\sigma},\b{q}}(B)}{\tau\left((m_{02}m_{03}m_{12}m_{13})_{\mathrm{odd}}\right)} 
\]
which concludes the proof.
\end{proof}

Condition \eqref{twoadicconditions2} is still an issue that needs to be considered as it involves products of our variables. To deal with this we note that, for a fixed $\b{m}$ and $\boldsymbol{\sigma}$, this condition is solely determined on the reduction of $\b{k}$ and $\b{l}$ modulo $8$. Therefore we will now split these vectors into appropriate arithmetic progressions modulo $8$. In doing so we may also remove any even ordered characters such as $(-1)^{f_r(\b{d},\b{k})}$ and any Jacobi symbol involving a power of $2$ from the sum over $\b{k}$ and $\b{l}$.

Fix some $\b{b},\b{m}\in\N^4$ and $\boldsymbol{\sigma}\in\{0,1\}^4$ satisfying \eqref{mconditions} and \eqref{sigmaconditions}, some $\b{d},\Tilde{\b{d}}\in\N^4$ such that $d_{ij}\Tilde{d}_{ij}=(m_{ij})_{\textrm{odd}}$ and some $\b{q}\in\mathcal{A}(\b{m},\boldsymbol{\sigma})$. Then for any $\b{K},\b{L}\in(\Z/8\Z)^{*4}$ define
\begin{align}
&\Theta_{r,1}(\b{d},\b{K},\boldsymbol{\sigma}) = (-1)^{f_r(\b{d},\b{K})}\left(\frac{2^{\sigma_0+\sigma_1+\sigma_2+\sigma_3}}{d_{02}d_{03}d_{12}d_{13}}\right)\left(\frac{2^{\sigma_2+\sigma_3}}{K_0K_1}\right)\left(\frac{2^{\sigma_0+\sigma_1}}{K_2K_3}\right)\left(\frac{2^{v_2(m_{02}m_{03}m_{12}m_{13})}}{K_0K_1K_2K_3}\right),\label{evenJacobisymbols}\\
&\Theta_2(\b{d},\Tilde{\b{d}},\b{k},\b{l}) = \left(\frac{l_0l_1l_2l_3}{d_{02}d_{03}d_{12}d_{13}}\right)\left(\frac{\Tilde{d}_{02}\Tilde{d}_{03}\Tilde{d}_{12}\Tilde{d}_{13}}{k_0k_1k_2k_3}\right)\left(\frac{l_0l_1}{k_2k_3}\right)\left(\frac{l_2l_3}{k_0k_1}\right),\label{oddJacobisymbols}\\
&N_{r,\b{b},\b{m},\b{d},\Tilde{\b{d}},\boldsymbol{\sigma},\b{q}}(\b{K},\b{L},B) = \mathop{\sum\sum\sum\sum}_{\substack{\b{k},\b{l}\in\N^4\\(\b{k},\b{l})\equiv(\b{K},\b{L})\bmod{8}\\\eqref{gcdconditionsandsquarefree3},\;\eqref{heightconditions3},\;\eqref{nonsquareconditions3}}} \frac{\mu^2(2k_0l_0k_1l_1k_2l_2k_3l_3)}{\tau\left(k_0l_0k_1l_1k_2l_2k_3l_3\right)}\Theta_2(\b{d},\Tilde{\b{d}},\b{k},\b{l}).
\end{align}
Then we may split $N_{r,\b{b},\b{m},\b{d},\Tilde{\b{d}},\boldsymbol{\sigma},\b{q}}(B)$ into arithmetic progressions modulo $8$. Doing so will allow certain conditions to be separated from our main sums since they only depend on such conditions.
\begin{lemma}\label{ConcreteSumsintoAPs}
For a fixed choice of $\b{b}$, $\b{m}$ and $\boldsymbol{\sigma}$ satisfying \eqref{mconditions} and \eqref{sigmaconditions} we have:
\begin{align*} N_{r,\b{b},\b{m},\b{d},\Tilde{\b{d}},\boldsymbol{\sigma},\b{q}}(B) = \mathop{\sum\sum\sum\sum}_{\substack{\b{K},\b{L}\in(\Z/8\Z)^{*4}\\\eqref{twoadicconditions2}}}\Theta_{r,1}(\b{d},\b{K},\boldsymbol{\sigma})N_{r,\b{b},\b{m},\b{d},\Tilde{\b{d}},\boldsymbol{\sigma},\b{q}}(\b{K},\b{L},B).
\end{align*}
\end{lemma}

\begin{proof}
Using the multiplicity of Jacobi symbols we may write
\[
(-1)^{f_r(\b{d},\b{K})}\Theta(\b{d},\Tilde{\b{d}},\b{k},\b{l},\boldsymbol{\sigma}) = \Theta_{r,1}(\b{d},\b{k},\boldsymbol{\sigma})\Theta_2(\b{d},\Tilde{\b{d}},\b{k},\b{l}).
\]
Then by splitting the inner sum of $N_{r,\b{b},\b{m},\boldsymbol{\sigma}}(B)$ into arithmetic progressions modulo $8$ we obtain.
\begin{align*} 
N_{r,\b{b},\b{m},\b{d},\Tilde{\b{d}},\boldsymbol{\sigma},\b{q}}(B) &= \mathop{\sum\sum}_{\substack{\b{K},\b{L}\in(\Z/8\Z)^{*4}}}\mathop{\sum\sum}_{\substack{\b{k},\b{l}\in\N^4\\(\b{k},\b{l})\equiv(\b{K},\b{L})\bmod{8}\\\eqref{gcdconditionsandsquarefree3},\;\eqref{heightconditions3},\;\eqref{nonsquareconditions3}\\\eqref{twoadicconditions2}}} \frac{\mu^2(2k_0l_0k_1l_1k_2l_2k_3l_3)\Theta_{r,1}(\b{d},\b{k},\boldsymbol{\sigma})\Theta_2(\b{d},\Tilde{\b{d}},\b{k},\b{l})}{\tau\left(k_0l_0k_1l_1k_2l_2k_3l_3\right)}.
\end{align*}
Now notice that, by \eqref{sigmaconditions}, $\Theta_{r,1}$ will be $(-1)^{f_1(\b{d},\b{k})}$ by itself or multiplying a Jacobi symbol of the form $\left(\frac{2}{\cdot}\right)$. For a fixed $\b{d}$, both are determined completely by the the reduction modulo $8$ of the $\b{k}$ and $\b{l}$. Thus it is enough to assert that these congruence classes satisfy \eqref{twoadicconditions2} and bring out $\Theta_{r,1}$ from the inner sum by replacing $\b{k}$ and $\b{l}$ with $\b{K}$ and $\b{L}$ in them respectively.
\end{proof}

\subsection{Isolating main terms and error terms}\label{deconstruction} In this section we describe our strategy for the remainder of the proof of Theorem \ref{MAINTHEOREM}. We aim to use the character sum results presented in \S\ref{technicalstuff} to handle the sums $N_{r,\b{b},\b{m},\b{d},\Tilde{\b{d}},\boldsymbol{\sigma},\b{q}}(\b{K},\b{L},B)$. To do so we must first ensure that the size of any characters are of order $(\log B)^{C}$ for some $C>0$. To ensure this, we split the sum over $\b{k}$ and $\b{l}$ into smaller regions and manipulate the expression into a form considered in \S\ref{technicalstuff}. Let $z_1 = (\log B)^{150A}$ for $A>0$ as before. The regions we will use are the following:
    \begin{equation}\label{H1}
        \mathcal{H}_1=\{(\b{k},\b{l})\in\N^8: \|k_0,k_1\|\leq z_1,\;\;\|k_2,k_3\|\leq z_1,\;\;\|l_0,l_1\|\leq z_1,\;\;\|l_2,l_3\|\leq z_1 \},
    \end{equation}
    \begin{equation}\label{H2}
        \mathcal{H}_2=\{(\b{k},\b{l})\in\N^8: \|k_0,k_1\|\leq z_1,\;\;\|k_2,k_3\|\leq z_1,\;\;\|l_0,l_1\|>z_1,\;\;\|l_2,l_3\|>z_1 \},
    \end{equation}
    \begin{equation}\label{H3}
        \mathcal{H}_3=\{(\b{k},\b{l})\in\N^8: \|k_0,k_1\|>z_1,\;\;\|k_2,k_3\|>z_1,\;\;\|l_0,l_1\|\leq z_1,\;\;\|l_2,l_3\|\leq z_1 \},
    \end{equation}
    \begin{equation}\label{H4}
        \mathcal{H}_4=\{(\b{k},\b{l})\in\N^8: \|k_2,k_3\|\leq z_1,\;\;\|l_2,l_3\|\leq z_1 \},
    \end{equation}
    \begin{equation}\label{H5}
        \mathcal{H}_5=\{(\b{k},\b{l})\in\N^8: \|k_0,k_1\|\leq z_1,\;\;\|l_0,l_1\|\leq z_1 \},
    \end{equation}
    \begin{equation}\label{H6}
        \mathcal{H}_{6}=\{(\b{k},\b{l})\in\N^8: \left(\|k_0,k_1\|>z_1\;\&\;\|l_2,l_3\|>z_1\right) \;\text{or}\;\left(\|k_2,k_3\|>z_1\;\&\;\|l_0,l_1\|>z_1\right)\}.
    \end{equation}
These regions cover $\N^8$ and the only intersections are between $\mathcal{H}_1$, $\mathcal{H}_4$ and $\mathcal{H}_5$ whose pairwise intersections are just $\mathcal{H}_1$. The following describes the contributions from the sum over each of these regions.
\begin{itemize}
    \item $\mathcal{H}_1$ will trivially contribute an error term.
    \item $\mathcal{H}_2$ and $\mathcal{H}_3$ will have an oscillating part which will be shown to contribute an error term of order $O\left(\frac{B^2}{(\log B)(\log\log B)^{A}}\right)$ for any $A>0$ by use Propositions \ref{symmetrictypeaverage1} and \ref{Asymmetrictypeaverage1}. There will also be a non-oscillating part which will become the main term of Theorem \ref{MAINTHEOREM} via Proposition \ref{maintermproposition}.
    \item $\mathcal{H}_4$ and $\mathcal{H}_5$ both contribute an error term of order $O\left(\frac{B^2\sqrt{\log\log B}}{\log B}\right)$ through the use of Lemma \ref{fixedconductorlemmaforvanishingmainterms}, and Propositions \ref{symmetrictypeaverage2} and \ref{Asymmetrictypeaverage2}. It is for the contribution from these regions that the non-square conditions of Theorem \ref{MAINTHEOREM} play a crucial role as they force certain characters modulo $8$ to be non-trivial. If this were not the case, the sums over these regions would have a non-oscillating part which would contribute a term of order $B^2$ as in the work of Browning--Lyczak--Sarapin, \cite{BLS}.
    \item $\mathcal{H}_6$ will contribute an error term of size $O\left(\frac{B^2}{(\log B)^A}\right)$ for any $A>0$. The main tool here is the previous work of the author on sums in the Jacobi symbol over hyperbolic regions, \cite{Me}.
\end{itemize}

We now express each contribution separately and bring in the other variables to express the overall main terms and error terms as sums which are malleable to the results of \S\ref{technicalstuff}. Suppressing the dependence on $\b{b},\boldsymbol{\sigma}$, and $\b{q}$, we define
\begin{align}\label{generalinnersum}
H_{r,i}(\b{d},\Tilde{\b{d}},\b{K},\b{L},B) = \mathop{\sum\sum}_{\substack{(\b{k},\b{l})\in\mathcal{H}_i\\(\b{k},\b{l})\equiv(\b{K},\b{L})\bmod{8}\\\eqref{gcdconditionsandsquarefree3},\;\eqref{heightconditions3},\;\eqref{nonsquareconditions3}}} \frac{\mu^2(2k_0l_0k_1l_1k_2l_2k_3l_3)}{\tau\left(k_0l_0k_1l_1k_2l_2k_3l_3\right)}\Theta_2(\b{d},\Tilde{\b{d}},\b{k},\b{l}).
\end{align}
Then, it is clear that
\begin{equation}\label{sumofsmallinnerregions}
N_{r,\b{b},\b{m},\b{d},\Tilde{\b{d}},\boldsymbol{\sigma},\b{q}}(\b{K},\b{L},B) = \sum_{i=1}^{6}H_{r,i}(\b{d},\Tilde{\b{d}},\b{K},\b{L},B) - 2H_{r,1}(\b{d},\Tilde{\b{d}},\b{K},\b{L}).
\end{equation}
for each $r=1,2$. The main term will be obtained from the sums $i=2,3$ in the cases where either $\Tilde{\b{d}}$ or $\b{d}$ is $(1,1,1,1)$ respectively. Bringing in the other variables, our overall main terms may therefore be expressed as
\begin{equation}\label{MAINTERMR2}
\mathcal{M}_{r,2}(B,\b{b}) = \mathop{\sum}_{\substack{\b{m}\in\N^4\\ m_{ij}\leq z_0\\ \eqref{mconditions}}} \mathop{\sum}_{\substack{\boldsymbol{\sigma}\in\{0,1\}^4\\ \eqref{sigmaconditions}}}\mathop{\sum}_{\substack{\b{q}\in\mathcal{A}(\b{m},\boldsymbol{\sigma})}}\mathop{\sum}_{\substack{\b{L}\in(\Z/8\Z)^{*4}\\ \eqref{twoadicmaintermconditions}}} \frac{H_{r,2}(\b{1},\b{m}_{\textrm{odd}},\b{1},\b{L},B)}{\tau\left((m_{02}m_{03}m_{12}m_{13})_{\mathrm{odd}}\right)}
\end{equation}
and
\begin{equation}\label{MAINTERMR3}
\mathcal{M}_{r,3}(B,\b{b}) = \mathop{\sum}_{\substack{\b{m}\in\N^4\\ m_{ij}\leq z_0\\ \eqref{mconditions}}} \mathop{\sum}_{\substack{\boldsymbol{\sigma}\in\{0,1\}^4\\ \eqref{sigmaconditions}}}\mathop{\sum}_{\substack{\b{q}\in\mathcal{A}(\b{m},\boldsymbol{\sigma})}}\mathop{\sum}_{\substack{\b{K}\in(\Z/8\Z)^{*4}\\ \eqref{twoadicmaintermconditions}}} \frac{\Theta_{r,1}(\b{m}_{\text{odd}},\b{K},\boldsymbol{\sigma})H_{r,3}(\b{m}_{\textrm{odd}},\b{1},\b{K},\b{1},B)}{\tau\left((m_{02}m_{03}m_{12}m_{13})_{\mathrm{odd}}\right)}
\end{equation}
where $\b{m}_{\textrm{odd}} = ((m_{02})_{\textrm{odd}},(m_{03})_{\textrm{odd}},(m_{12})_{\textrm{odd}},(m_{13})_{\textrm{odd}})$, $\b{1}=(1,1,1,1)$ and
\begin{equation}\label{twoadicmaintermconditions}
    \begin{cases}
        L_0L_2(m_{03}m_{12})_{\textrm{odd}}\equiv -\delta_r q_0\bmod{8},\;
        L_1L_3(m_{03}m_{12})_{\textrm{odd}}\equiv q_1\bmod{8},\\
        L_1L_2(m_{02}m_{13})_{\textrm{odd}}\equiv \delta_r q_2\bmod{8},\;
        L_0L_3(m_{02}m_{13})_{\textrm{odd}}\equiv -q_3\bmod{8}.\\
    \end{cases}
\end{equation}
The even characters from $\Theta_{r,1}$ are absent in the sums $\mathcal{M}_{r,2}(B,\b{b})$ as this corresponds to the terms where $\b{d}=\b{1}$ and $\b{K}\equiv \b{1}\bmod{8}$, in which case this function is trivially always $1$. Notice that due to the height conditions in $\mathcal{H}_{2}$ and $\mathcal{H}_3$, the non-square condition \eqref{nonsquareconditions3} implicitly holds and may therefore be ignored. This is not the case in $\mathcal{H}_{4}$ and $\mathcal{H}_5$. In these regions, we use \eqref{nonsquareconditions3} to force $\Theta_1$ to be a non-principal Dirichlet character modulo $8$. By re-ordering the sums over $\mathcal{H}_4$ and $\mathcal{H}_5$, oscillation of this non-principal even character will result in an error term. We will call the contribution of the sums where this method is necessary ``vanishing main terms''. These are given by
\begin{equation}\label{VANISHMAINTERMR4}
\mathcal{V}_{r,4}(B,\b{b}) = \mathop{\sum}_{\substack{\b{m}\in\N^4,\eqref{mconditions}\\ (m_{02}m_{03}m_{12}m_{13})_{\textrm{odd}}=1}}\mathop{\sum}_{\substack{\boldsymbol{\sigma}\in\{0,1\}^4\\ \eqref{sigmaconditions}}}\mathop{\sum}_{\substack{\b{q}\in\mathcal{A}(\b{m},\boldsymbol{\sigma})}}\mathop{\sum\sum}_{\substack{\b{K},\b{L}\in(\Z/8\Z)^{*4}\\ \eqref{twoadicvanishingmaintermconditions4}}} \Theta_{r,1}(\b{1},\b{K},\boldsymbol{\sigma})H_{r,4}(\b{1},\b{1},\b{K},\b{L},B)
\end{equation}
and
\begin{equation}\label{VANISHMAINTERMR5}
\mathcal{V}_{r,5}(B,\b{b}) = \mathop{\sum}_{\substack{\b{m}\in\N^4,\eqref{mconditions}\\ (m_{02}m_{03}m_{12}m_{13})_{\textrm{odd}}=1}}\mathop{\sum}_{\substack{\boldsymbol{\sigma}\in\{0,1\}^4\\ \eqref{sigmaconditions}}}\mathop{\sum}_{\substack{\b{q}\in\mathcal{A}(\b{m},\boldsymbol{\sigma})}}\mathop{\sum\sum}_{\substack{\b{K},\b{L}\in(\Z/8\Z)^{*4}\\ \eqref{twoadicvanishingmaintermconditions5}}} \Theta_{r,1}(\b{1},\b{K},\boldsymbol{\sigma})H_{r,5}(\b{1},\b{1},\b{K},\b{L},B)
\end{equation}
where
\begin{equation}\label{twoadicvanishingmaintermconditions4}
    \begin{cases}
        K_0L_0\equiv -\delta_r q_0\bmod{8},\;
        K_1L_1\equiv q_1\bmod{8},\\
        K_1L_1\equiv \delta_r q_2\bmod{8},\;
        K_0L_0\equiv -q_3\bmod{8},\\
        K_2,L_2,K_3,L_3\equiv 1\bmod{8},
    \end{cases}
\end{equation}
and
\begin{equation}\label{twoadicvanishingmaintermconditions5}
    \begin{cases}
        K_2L_2\equiv -\delta_r q_0\bmod{8},\;
        K_3L_3\equiv q_1\bmod{8},\\
        K_2L_2\equiv \delta_r q_2\bmod{8},\;
        K_3L_3\equiv -q_3\bmod{8},\\
        K_0,L_0,K_1,L_1\equiv 1\bmod{8}.
    \end{cases}
\end{equation}

The remaining terms contribute to the error term. For convenience we split these errors into similar sections. Define $\Flatsum_{\substack{\b{m},\boldsymbol{\sigma},\b{q},\b{K},\b{L}}}$ as the sum over the conditions $\b{m}\in\N^4, m_{ij}\leq z_0,\eqref{mconditions};\boldsymbol{\sigma}\in\{0,1\}^4,\eqref{sigmaconditions};\b{q}\in\mathcal{A}(\b{m},\boldsymbol{\sigma});\b{K},\b{L}\in(\Z/8\Z)^{*4},\eqref{twoadicconditions2}$. Define also,
\begin{equation}\label{innersumforerrortermstwistedbyJacobi}
T_{r,i}(\b{d},\Tilde{\b{d}}) = \frac{\Theta_{r,1}(\b{d},\b{K},\boldsymbol{\sigma})H_{r,i}(\b{d},\Tilde{\b{d}},\b{K},\b{L},B)}{\tau((m_{02}m_{03}m_{12}m_{13})_{\textrm{odd}})}.
\end{equation}
Then the remaining error terms are
\begin{align}
    \mathcal{E}_{r,1}(B,\b{b}) &= \Flatsum_{\substack{\b{m},\boldsymbol{\sigma},\b{q}\\\b{K},\b{L}}} \mathop{\sum\sum}_{\substack{\b{d},\Tilde{\b{d}}\in\N^4\\ d_{ij}\Tilde{d}_{ij}=(m_{ij})_{\textrm{odd}}}} T_{r,1}(\b{d},\Tilde{\b{d}}),\;\;
    \mathcal{E}_{r,2}(B,\b{b}) =\Flatsum_{\substack{\b{m},\boldsymbol{\sigma},\b{q}\\\b{K},\b{L}}}\mathop{\sum\sum}_{\substack{\b{d},\Tilde{\b{d}}\in\N^4\\ d_{ij}\Tilde{d}_{ij}=(m_{ij})_{\textrm{odd}}\\ \b{K}\equiv \b{1}\bmod{8} \Rightarrow \b{d}\neq\b{1}}} T_{r,2}(\b{d},\Tilde{\b{d}}),  \label{Errorterms1and2}\\
    \mathcal{E}_{r,3}(B,\b{b}) &=\Flatsum_{\substack{\b{m},\boldsymbol{\sigma},\b{q}\\\b{K},\b{L}}}\mathop{\sum\sum}_{\substack{\b{d},\Tilde{\b{d}}\in\N^4\\ d_{ij}\Tilde{d}_{ij}=(m_{ij})_{\textrm{odd}}\\ \b{L}\equiv \b{1}\bmod{8} \Rightarrow \Tilde{\b{d}}\neq\b{1}}}\hspace{-5pt}T_{r,3}(\b{d},\Tilde{\b{d}}),\;\;
    \mathcal{E}_{r,4}(B,\b{b}) = \Flatsum_{\substack{\b{m},\boldsymbol{\sigma},\b{q}\\\b{K},\b{L}}}\mathop{\sum\sum}_{\substack{\b{d},\Tilde{\b{d}}\in\N^4\\ d_{ij}\Tilde{d}_{ij}=(m_{ij})_{\textrm{odd}}\\ \b{d}=\Tilde{\b{d}}=\b{1}\Rightarrow \;\text{at least one of}\\K_2,L_2,K_3,L_3\not\equiv 1\bmod{8}}} \hspace{-5pt}T_{r,4}(\b{d},\Tilde{\b{d}}),\label{Errorterms3and4}\\
    \mathcal{E}_{r,5}(B,\b{b})&=\Flatsum_{\substack{\b{m},\boldsymbol{\sigma},\b{q}\\\b{K},\b{L}}}\mathop{\sum\sum}_{\substack{\b{d},\Tilde{\b{d}}\in\N^4\\ d_{ij}\Tilde{d}_{ij}=(m_{ij})_{\textrm{odd}}\\ \b{d}=\Tilde{\b{d}}=\b{1}\Rightarrow \;\text{at least one of}\\K_0,L_0,K_1,L_1\not\equiv 1\bmod{8}}} T_{r,5}(\b{d},\Tilde{\b{d}}),\;\;
    \mathcal{E}_{r,6}(B,\b{b})=\Flatsum_{\substack{\b{m},\boldsymbol{\sigma},\b{q}\\\b{K},\b{L}}}\mathop{\sum\sum}_{\substack{\b{d},\Tilde{\b{d}}\in\N^4\\ d_{ij}\Tilde{d}_{ij}=(m_{ij})_{\textrm{odd}}}} T_{r,6}(\b{d},\Tilde{\b{d}}). \label{Errorterms5and6}
\end{align}
We summarise this section by bringing these contributions together and expressing each $N_r(B)$ in a concise manner. Define
\begin{equation}\label{regionalmaintermcontributions}
    N_{r,i}(B) = 
     \sum_{\substack{\b{b}\in\N^4\\ b_i\leq z_0\\\eqref{bgcdconditions}}}\left(\mathcal{M}_{r,i}(B,\b{b}) + \mathcal{E}_{r,i}(B,\b{b})\right)
\end{equation}
for $i=2,3$,
\begin{equation}\label{regionalvanishingmaintermcontributions}
    N_{r,i}(B) = 
     \sum_{\substack{\b{b}\in\N^4\\ b_i\leq z_0\\\eqref{bgcdconditions}}}\left(\mathcal{V}_{r,i}(B,\b{b}) + \mathcal{E}_{r,i}(B,\b{b})\right)
\end{equation}
for $i=4,5$ and
\begin{equation}\label{regionalerrortermcontributions}
    N_{r,i}(B) = 
     \sum_{\substack{\b{b}\in\N^4\\ b_i\leq z_0\\\eqref{bgcdconditions}}}\mathcal{E}_{r,i}(B,\b{b})
\end{equation}
for $i=1,6$.
The following therefore follows from summing \eqref{sumofsmallinnerregions} over the remaining variables:
\begin{proposition}\label{Simplificationsummary}
    For $B\geq 3$,
    \[
    N_{r}(B) = \sum_{i=1}^{6}N_{r,i}(B) - 2N_{r,1}(B) + O\left(\frac{B^2(\log B)}{z_0}\right).
    \]
\end{proposition}

Moving forward, we will suppress some notation by writing:
\begin{align*}
    M_0 = 2^{\sigma_0}m_{02}m_{03}b_0^2,\;
    M_1 = 2^{\sigma_1}m_{12}m_{13}b_1^2,\;
    M_2 = 2^{\sigma_2}m_{02}m_{12}b_2^2,\;
    M_3 = 2^{\sigma_3}m_{03}m_{13}b_3^2.
\end{align*}
At this stage, we invite the reader to consider the summary of notation found in \S\ref{Notation}. This appendix section summarizes only the notation that will be used moving forward.

\section{Large Conductor Error Terms}\label{largeconductors}
In this section we bound the $\mathcal{E}_{r,6}(B,\b{b})$ using Lemma \ref{Me}. We begin by bounding the inner sums, $H_{r,6}(\b{d},\Tilde{\b{d}},\b{K},\b{L},B)$:

\begin{lemma}
Fix some $\b{b},\b{m}\in\N^4$ and $\boldsymbol{\sigma}\in\{0,1\}^4$ satisfying $\mu^2(m_{02}m_{03}m_{12}m_{13})=1$, \eqref{mconditions} and \eqref{sigmaconditions}, some $\b{d},\Tilde{\b{d}}\in\N^4$ such that $d_{ij}\Tilde{d}_{ij}=(m_{ij})_{\textrm{odd}}$. Fix also some $\b{K},\b{L}\in(\Z/8\Z)^{*4}$ satisfying \eqref{twoadicconditions2}. Then for $r=1,2$ we have
\begin{align*}
H_{r,6}(\b{d},\Tilde{\b{d}},\b{K},\b{L},B)\ll\frac{B^2(\log B)^6}{M_0M_1M_2M_3z_1^{1/2}}.
\end{align*}
\end{lemma}

\begin{proof}
First we recall the height conditions for these expressions. The first is that given in \eqref{H6} and the second is the hyperbolic height condition \eqref{heightconditions3}. To handle \eqref{heightconditions3} we will write
\begin{equation}\label{hyperbolicheighttricksection9}
    \mathds{1}(\|k_0l_0M_0,k_1l_1M_1\|\cdot\|k_2l_2M_2,k_3l_3M_3\|\leq B) = \prod_{(u,v)\in\{0,1\}\times\{2,3\}}\mathds{1}(k_ul_uM_uk_vl_vM_v\leq B).
\end{equation}

Starting from height conditions from \eqref{H6}, we will partition the space even further. First suppose that $\|k_0,k_1\|>z_1$ and $\|l_2,l_3\|>z_1$. Then we have $4$ cases:
\begin{itemize}
\item[$\mathcal{R}_{1}$:] $\left(k_1>z_1,\;\;\;\text{and}\;\;\; l_3>z_1\right)$;
\item[$\mathcal{R}_{2}$:] $\left(k_0>z_1,\;k_1\leq z_1\;\;\; \text{and} \;\;\; l_3>z_1\right)$;
\item[$\mathcal{R}_{3}$:] $\left(k_1>z_1\;\;\;\text{and}\;\;\; l_2>z_1,\;l_3\leq z_1\right)$;
\item[$\mathcal{R}_{4}$:] $\left(k_0>z_1,\;k_1\leq z_1\;\;\; \text{and} \;\;\; l_2>z_1,\;l_3\leq z_1\right)$.
\end{itemize}
We also have regions where $\|k_0,k_1\|\leq z_1$ or $\|l_2,l_3\|\leq z_1$ but $\|k_2,k_3\|>z_1$ and $\|l_0,l_1\|>z_1$. This gives $4$ more regions:
\begin{itemize}
\item[$\mathcal{R}_{5}$:] $(\|k_0,k_1\|\leq z_1\;\text{or}\;\|l_2,l_3\|\leq z_1)$\;\text{and}\;$(k_3>z_1\;\text{and}\;l_1>z_1)$;
\item[$\mathcal{R}_{6}$:] $(\|k_0,k_1\|\leq z_1\;\text{or}\;\|l_2,l_3\|\leq z_1)$\;\text{and}\;$(k_2>z_1,\;k_3\leq z_1\;\text{and}\; l_1>z_1);$
\item[$\mathcal{R}_{7}$:] $(\|k_0,k_1\|\leq z_1\;\text{or}\;\|l_2,l_3\|\leq z_1)$\;\text{and}\;$(k_3>z_1\;\text{and}\; l_0>z_1,\;l_1\leq z_1)$;
\item[$\mathcal{R}_{8}$:] $(\|k_0,k_1\|\leq z_1\;\text{or}\;\|l_2,l_3\|\leq z_1)$\;\text{and}\; $(k_2>z_1,\;k_3\leq z_1\;\text{and}\; l_0>z_1,\;l_1\leq z_1)$.
\end{itemize}
Then we define
\[
S_{h}(B) = \mathop{\sum\sum\sum\sum}_{\substack{(\b{k},\b{l})\in\mathcal{R}_{1a}\\(\b{k},\b{l})\equiv(\b{K},\b{L})\bmod{8}\\\eqref{gcdconditionsandsquarefree3},\;\eqref{nonsquareconditions3}}} \frac{\mu^2(2k_0l_0k_1l_1k_2l_2k_3l_3)}{\tau\left(k_0l_0k_1l_1k_2l_2k_3l_3\right)}\Theta_2(\b{d},\Tilde{\b{d}},\b{k},\b{l})
\]
for $1\leq h\leq 8$. As our final notational manipulation of the section, we use the multiplicativity of the Jacobi symbol to write $\Theta_2(\b{d},\Tilde{\b{d}},\b{k},\b{l})$ as 
\begin{align*}
    \Theta_2(\b{d},\Tilde{\b{d}},\b{k},\b{l}) = \left(\prod_{i=0}^{3}\theta_1(\b{d},l_i)\theta_2(\Tilde{\b{d}},k_i)\right)\left(\frac{l_0}{k_2}\right)\left(\frac{l_0}{k_3}\right)\left(\frac{l_1}{k_2}\right)\left(\frac{l_1}{k_3}\right)\left(\frac{l_2}{k_0}\right)\left(\frac{l_2}{k_1}\right)\left(\frac{l_3}{k_0}\right)\left(\frac{l_3}{k_1}\right),
\end{align*}
where
\[
\theta_1(\b{d},l_i) = \left(\frac{l_i}{d_{02}d_{03}d_{12}d_{13}}\right)\;\;\textrm{and}\;\;
\theta_2(\Tilde{\b{d}},k_i) = \left(\frac{\Tilde{d}_{02}\Tilde{d}_{03}\Tilde{d}_{12}\Tilde{d}_{13}}{k_i}\right).
\]
Then, using \eqref{hyperbolicheighttricksection9}, $\mu^2(2k_0l_0k_1l_1k_2l_2k_3l_3)=1$ and the multiplicativity of $\tau$ we may arrange the order of summation of each $S_{h}(B)$ for $1\leq h\leq 8$ and apply the triangle inequality to find that all of these satisfy an upper bound of the form
\begin{align*}
S_{h}(B) \ll &\mathop{\sum\sum\sum\sum}_{\substack{k_ul_uM_u\leq B\\k_vl_vM_v\leq B\\k_ul_uk_vl_vM_uM_v\leq B}}\frac{1}{\tau(k_ul_u)\tau(k_vl_v)}\mathop{\sum\sum}_{\substack{l_i\leq B\\ k_j\leq B}}\left\lvert\mathop{\sum\sum}_{\substack{z_1<k_i\leq B/(l_ik_jM_iM_j)\\ z_1<l_j\leq B/(l_ik_jM_iM_j)\\k_il_j\leq B/l_ik_jM_iM_j\\ \eqref{gcdconditionsandsquarefree3},\eqref{nonsquareconditions3}}}a_{k_i}b_{l_j}\left(\frac{l_j}{k_i}\right)\right\rvert,
\end{align*}
where the indices $i,j,u,v\in\{0,1,2,3\}$ are all distinct and whose assignments depend uniquely on $1\leq h\leq 8$ and where $|a_{k_i}|,|b_{l_j}|\leq 1$ are complex sequences depending independently on $k_i$ and $l_j$ respectively. These sequences contain the $\textrm{gcd}$ conditions from the $\mu^2$ factor, the congruence conditions on the variables $k_i$ and $l_j$, $\frac{1}{\tau}$ factors, superfluous characters containing $k_i$ or $l_j$ and superfluous height conditions (in the form of an indicator function). The innermost sums here are exactly of the form considered in Lemma \ref{Me}. We therefore use this lemma to obtain:
\begin{align*}
    S_{h}(B) &\ll \mathop{\sum\sum\sum\sum}_{\substack{k_ul_uM_u\leq B\\k_vl_vM_v\leq B\\k_ul_uk_vl_vM_uM_v\leq B}}\frac{1}{\tau(k_ul_u)\tau(k_vl_v)}\mathop{\sum\sum}_{\substack{l_i, k_j\leq B}}\frac{B(\log B)^3}{l_ik_jM_iM_jz_1^{1/2}}\\
    &\ll \mathop{\sum\sum\sum\sum}_{\substack{k_ul_uM_u\leq B\\k_vl_vM_v\leq B\\k_ul_uk_vl_vM_uM_v\leq B}}\frac{1}{\tau(k_ul_u)\tau(k_vl_v)} \frac{B(\log B)^5}{M_iM_jz_1^{1/2}}\\
    &\ll \frac{B(\log B)^5}{M_iM_jz_1^{1/2}} \mathop{\sum\sum}_{nm\leq B/(M_uM_v)} 1 \ll \frac{B^2(\log B)^6}{M_0M_1M_2M_3z_1^{1/2}}
\end{align*}
for each $1\leq h\leq 8$.
\end{proof}

\begin{proposition}\label{Overall_Large_Conductor_Error_Term}
    Let $B\geq 3$. Then for $r=1,2$,
    \[
    N_{r,6}(B) \ll \frac{B^2(\log B)^6}{z_1^{1/2}}.
    \]
\end{proposition}

\begin{proof}
    Using the previous lemma we have
    \begin{align*}
    \mathcal{E}_{r,6}(B,\b{b}) &= \mathop{\sum}_{\substack{\b{m}\in\N^4\\ m_{ij}\leq z_0\\ \eqref{mconditions}}} \mathop{\sum}_{\substack{\boldsymbol{\sigma}\in\{0,1\}^4\\ \eqref{sigmaconditions}}}\mathop{\sum}_{\substack{\b{q}\in\mathcal{A}(\b{m},\boldsymbol{\sigma})}}\mathop{\sum\sum}_{\substack{\b{K},\b{L}\in(\Z/8\Z)^{*4}\\ \eqref{twoadicconditions2}}}\mathop{\sum\sum}_{\substack{\b{d},\Tilde{\b{d}}\in\N^4\\ d_{ij}\Tilde{d}_{ij}=(m_{ij})_{\textrm{odd}}}}\hspace{-8pt}\frac{\Theta_{r,1}(\b{d},\b{K},\boldsymbol{\sigma})H_{r,6}(\b{d},\Tilde{\b{d}},\b{K},\b{L},B)}{\tau\left((m_{02}m_{03}m_{12}m_{13})_{\mathrm{odd}}\right)}\\
    &\ll \mathop{\sum}_{\substack{\b{m}\in\N^4\\ m_{ij}\leq z_0\\ \eqref{mconditions}}} \mathop{\sum}_{\substack{\boldsymbol{\sigma}\in\{0,1\}^4\\ \eqref{sigmaconditions}}}\mathop{\sum}_{\substack{\b{q}\in\mathcal{A}(\b{m},\boldsymbol{\sigma})}}\mathop{\sum\sum}_{\substack{\b{K},\b{L}\in(\Z/8\Z)^{*4}\\ \eqref{twoadicconditions2}}} \frac{B^2(\log B)^6}{m_{02}^2m_{03}^2m_{12}^2m_{13}^2b_0^2b_1^2b_2^2b_3^2z_1^{1/2}}.
\end{align*}
Summing this over $\b{b}$ gives
\[
N_{r,6}(B) \ll \sum_{\substack{\b{b}\in\N^4\\ b_i\leq z_0}}\mathop{\sum}_{\substack{\b{m}\in\N^4\\ m_{ij}\leq z_0\\ \eqref{mconditions}}} \mathop{\sum}_{\substack{\boldsymbol{\sigma}\in\{0,1\}^4\\ \eqref{sigmaconditions}}}\mathop{\sum}_{\substack{\b{q}\in\mathcal{A}(\b{m},\boldsymbol{\sigma})}}\mathop{\sum\sum}_{\substack{\b{K},\b{L}\in(\Z/8\Z)^{*4}\\ \eqref{twoadicconditions2}}} \frac{B^2(\log B)^6}{m_{02}^2m_{03}^2m_{12}^2m_{13}^2b_0^2b_1^2b_2^2b_3^2z_1^{1/2}}.
\]
The result follows since there are only finitely many $\boldsymbol{\sigma},\b{q},\b{K}$ and $\b{L}$ to consider and the sums over $\b{m}$ and $\b{b}$ converge.
\end{proof}

\section{Small Conductor Error Term}\label{smallconductors}
In this section we will bound the error terms $\mathcal{E}_{r,j}(B,\b{b})$ for $1\leq j\leq 5$ using Propositions \ref{symmetrictypeaverage1}, \ref{Asymmetrictypeaverage1}, \ref{symmetrictypeaverage2}, \ref{Asymmetrictypeaverage2}.

\subsection{The Error Terms $\mathcal{E}_{r,1}(B,\b{b})$}\label{smallconductortriv}
Here it is enough to use a trivial bound, since the variables in this region are all bounded by a power of $\log B$. We obtain:

\begin{proposition}
    Let $B\geq 3$. Then for $r=1,2$ we have
    \[
    N_{r,1}(B) \ll (\log B)^{1208A}.
    \]
\end{proposition}

\begin{proof}
Recall $\mathcal{E}_{r,1}(B,\b{b})$ from \eqref{Errorterms1and2}. Now, by summing trivially over the height conditions in the region $\mathcal{H}_1$, we have
\[
\lvert T_{r,1}(\b{d},\Tilde{\b{d}})\rvert = \frac{H_{r,1}(\b{d},\Tilde{\b{d}},\b{K},\b{L},B)}{\tau((m_{02}m_{03}m_{12}m_{13})_{\textrm{odd}})} \ll z_1^{8} = \frac{(\log B)^{1200A}}{\tau((m_{02}m_{03}m_{12}m_{13})_{\textrm{odd}})}.
\]
Then,
\[
\mathop{\sum\sum}_{\substack{\b{d},\Tilde{\b{d}}\in \N^4\\ d_{ij}\Tilde{d}_{ij}=(m_ij)_{\textrm{odd}} }}\lvert T_{r,1}(\b{d},\Tilde{\b{d}})\rvert \ll \frac{\tau((m_{02}m_{03}m_{12}m_{13})_{\textrm{odd}})(\log B)^{1200A}}{\tau((m_{02}m_{03}m_{12}m_{13})_{\textrm{odd}})} = (\log B)^{1200A}
\]
where we have implicitly used \eqref{mconditions} to simplify the product of the $\tau(m_{ij})$. Since there are only finitely many $\b{K},\b{L},\b{q}$ and $\boldsymbol{\sigma}$ to consider and the sum over $\b{m}$ is trivially bounded by $z_0^4 = (\log B)^{4A}$ we obtain
\[
N_{r,1}(B) \ll \sum_{\substack{\b{b}\in\N^4\\ b_i\leq z_0}} \Flatsum_{\substack{\b{m},\boldsymbol{\sigma},\b{q}\\\b{K},\b{L}}} (\log B)^{1200A} \ll \sum_{\substack{\b{b}\in\N^4\\ b_i\leq z_0}} (\log B)^{1204A} \ll (\log B)^{1208A}.
\]
\end{proof}

\subsection{The Error Terms $\mathcal{E}_{r,2}(B,\b{b})$ and $\mathcal{E}_{r,3}(B,\b{b})$}\label{smallconductor23} These error terms are bounded using the fact that, given the conditions on the variables $\b{m},\b{K},\b{L},\b{d}$ and $\Tilde{\b{d}}$, the sums $H_{r,2}(\b{d},\Tilde{\b{d}},\b{K},\b{L},B)$ and $H_{r,3}(\b{d},\Tilde{\b{d}},\b{K},\b{L},B)$ are of a type covered by Propositions \eqref{symmetrictypeaverage1} or \eqref{Asymmetrictypeaverage1}. We first remark that $\mathcal{E}_{r,2}(B,\b{b})$ and $\mathcal{E}_{r,3}(B,\b{b})$ are symmetrically equivalent, the latter being of the same form as the former with the variables $\Tilde{\b{d}},\b{K}$ and $\b{k}$ switching roles with the variables $\b{d},\b{L}$ and $\b{l}$. For this reason we will restrict our focus to $\mathcal{E}_{r,2}(B,\b{b})$. Our first aim is to examine the sums $H_{r,2}(\b{d},\Tilde{\b{d}},\b{K},\b{L},B)$.

\begin{lemma}\label{Errorterm2boundlemma}
    Fix some $\b{b}\in\N^4$, some $\b{m}\in\N^4$ satisfying \eqref{mconditions}, some $\boldsymbol{\sigma}\in\{0,1\}^4$ satisfying \eqref{sigmaconditions} and some $\b{q}\in\mathcal{A}(\b{m},\boldsymbol{\sigma})$. Suppose that $\b{K},\b{L}\in(\Z/8\Z)^{*4}$ and $\b{d},\Tilde{\b{d}}\in\N^4$ satisfy the conditions
    \begin{equation}\label{nontrivialcharacterconditions1}
        \begin{cases}
        \b{K},\b{L}\;\text{satisfy}\;\eqref{twoadicconditions2},\\
        d_{ij}\Tilde{d}_{ij} = (m_{ij})_{\textrm{odd}}\;\forall\;ij\in\{02,03,12,13\},\\
        \b{K}\equiv\b{1}\bmod{8}\Rightarrow \b{d}\neq\b{1}.
        \end{cases}
    \end{equation}
    Then
    \begin{align*}
     H_{r,2}(\b{d},\Tilde{\b{d}},\b{K},\b{L},B)&\ll_{A} \frac{B^2\mathcal{MAX}_1(B)}{M_0M_1M_2M_3}
    \end{align*}
    where $\mathcal{MAX}_1(B)$ is defined as
    \[
     \max\left\{\mathds{1}(\b{d}=1)\frac{\tau(m_{02}m_{03}m_{12}m_{13})^2\tau(b_0)\tau(b_1)\tau(b_2)\tau(b_3)}{(\log B)(\log \log B)^{66A}}, \frac{d_{02}^2d_{03}^2d_{12}^2d_{13}^2}{(\log B)^{132A}},\frac{(\log B)(\log\log B)^4}{(\log B)^{A/3}}\right\}.
    \]
\end{lemma}

\begin{proof}
Recall \eqref{oddJacobisymbols},\eqref{H2} and \eqref{generalinnersum}. We order $H_{r,2}(\b{d},\Tilde{\b{d}},\b{K},\b{L},B)$ to sum over $\b{l}$ first and write
\[
\Theta_2(\b{d},\Tilde{\b{d}},\b{k},\b{l}) =\chi_{\Tilde{\b{d}}}(\b{k})\chi_{\b{d},k_2k_3}(l_0l_1)\chi_{\b{d},k_0k_1}(l_2l_3)
\]
where
\[
\chi_{\Tilde{\b{d}}}(\b{k}) = \left(\frac{\Tilde{d}_{02}\Tilde{d}_{03}\Tilde{d}_{12}\Tilde{d}_{13}}{k_0k_1k_2k_3}\right)\;\;\text{and}\;\;
\chi_{\b{d},Q}(n)=\left(\frac{n}{d_{02}d_{03}d_{12}d_{13}Q}\right)
\]
for any $Q,n\in\N$. Then,
\[
H_{r,2}(\b{d},\Tilde{\b{d}},\b{K},\b{L},B) = \mathop{\sum\sum\sum\sum}_{\substack{\|k_0,k_1\|,\|k_2,k_3\|\leq z_1 \\ \b{k}\equiv\b{K}\bmod{8} \\\eqref{kgcdconditions}}}\frac{\mu^2(k_0k_1k_2k_3)\chi_{\Tilde{\b{d}}}(\b{k})}{\tau(k_0k_1k_2k_3)} H_{r,2}(\b{d},\Tilde{\b{d}},\b{K},\b{L},\b{k},B)
\]
where
\[
H_{r,2}(\b{d},\Tilde{\b{d}},\b{K},\b{L},\b{k},B) = \mathop{\sum\sum\sum\sum}_{\substack{\|l_0,l_1\|,\|l_2,l_3\|>z_1\\\b{l}\equiv\b{L}\bmod{8}\\\eqref{heightconditions3},\eqref{lgcdcondition}}}\frac{\mu^2(l_0l_1l_2l_3)}{\tau(l_0)\tau(l_1)\tau(l_2)\tau(l_3)}\chi_{\b{d},k_2k_3}(l_0l_1)\chi_{\b{d},k_0k_1}(l_2l_3),
\]

\begin{equation}\label{kgcdconditions} \begin{cases}
\gcd(k_0,2^{\sigma_1}m_{02}m_{03}m_{12}m_{13}b_1)=\gcd(k_1,2^{\sigma_0}m_{02}m_{03}m_{12}m_{13}b_0)=1,\\
\gcd(k_2,2^{\sigma_3}m_{02}m_{03}m_{12}m_{13}b_3)=\gcd(k_3,2^{\sigma_2}m_{02}m_{03}m_{12}m_{13}b_2)=1,\\
\end{cases}
\end{equation}
and
\begin{equation}\label{lgcdcondition} \begin{cases}
\gcd(l_0,2^{\sigma_1}m_{02}m_{03}m_{12}m_{13}k_0k_1b_1)=\gcd(l_1,2^{\sigma_0}m_{02}m_{03}m_{12}m_{13}k_0k_1b_0)=1,\\
\gcd(l_2,2^{\sigma_3}m_{02}m_{03}m_{12}m_{13}k_2k_3b_3)=\gcd(l_3,2^{\sigma_2}m_{02}m_{03}m_{12}m_{13}k_2k_3b_2)=1.\\
\end{cases}
\end{equation}
Notice that these sums are now very similar to those considered in Propositions \ref{symmetrictypeaverage1} and \ref{Asymmetrictypeaverage1}, except for the $\mu^2(l_0l_1l_2l_3)$ term in the inner sum. To deal with this we apply Lemma \ref{removesquarefree} to the inner sums with: $w_0=w_1=z_0=(\log B)^{A}$, $c_i=k_iM_i$ for $0\leq i\leq 3$, and $g_i$ encoding the characters, the $\frac{1}{\tau}$ factors and the $\mathrm{gcd}$ conditions \eqref{lgcdcondition}. Then 
\begin{align*}
H_{r,2}(\b{d},\Tilde{\b{d}},\b{K},\b{L},\b{k},B) &= \sum_{s\leq z_0
} \mu(s) \mathop{\sum\sum\sum\sum}_{\substack{\|l_0',l_1'\|,\|l_2',l_3'\|\leq z_0 \\ p|l_0'l_1'l_2'l_3'\Rightarrow p|s\\ s^2|l_0'l_1'l_2'l_3', \eqref{l'gcdcondition}}} \frac{\chi_{\b{d},k_2k_3}(l_0'l_1')\chi_{\b{d},k_0k_1}(l_2'l_3')}{\tau(l_0')\tau(l_1')\tau(l_2')\tau(l_3')}H'_{r,2}(\b{d},\Tilde{\b{d}},\b{K},\b{L},\b{k},\b{l}',B) \\
&+O\left(\frac{B^2(\log B)}{k_0k_1k_2k_3M_0M_1M_2M_3M_4z_0^{1/3}}\right)
\end{align*}
where
\[
H'_{r,2}(\b{d},\Tilde{\b{d}},\b{K},\b{L},\b{k},\b{l}',B) = \mathop{\sum\sum\sum\sum}_{\substack{\|l_0'l_0'',l_1'l_1''\|,\|l_2'l_2'',l_3'l_3''\|>z_1\\ l_i''\equiv L_i/l_i'\bmod{8}\;\forall\;0\leq i\leq 3\\ \eqref{l''gcdcondition},\eqref{heightconditions4}}}\frac{1}{\tau(l_0'')\tau(l_1'')\tau(l_2'')\tau(l_3'')}\chi_{\b{d},k_2k_3}(l_0''l_1'')\chi_{\b{d},k_0k_1}(l_2''l_3''),
\]
\begin{equation}\label{l'gcdcondition} \begin{cases}
\gcd(l_0',2^{\sigma_1}m_{02}m_{03}m_{12}m_{13}k_0k_1b_1)=\gcd(l_1',2^{\sigma_0}m_{02}m_{03}m_{12}m_{13}k_0k_1b_0)=1,\\
\gcd(l_2',2^{\sigma_3}m_{02}m_{03}m_{12}m_{13}k_2k_3b_3)=\gcd(l_3',2^{\sigma_2}m_{02}m_{03}m_{12}m_{13}k_2k_3b_2)=1.\\
\end{cases}
\end{equation}
\begin{equation}\label{l''gcdcondition} \begin{cases}
\gcd(l_0'',2^{\sigma_1}m_{02}m_{03}m_{12}m_{13}k_0k_1sb_1)=\gcd(l_1'',2^{\sigma_0}m_{02}m_{03}m_{12}m_{13}k_0k_1sb_0)=1,\\
\gcd(l_2'',2^{\sigma_3}m_{02}m_{03}m_{12}m_{13}k_2k_3sb_3)=\gcd(l_3'',2^{\sigma_2}m_{02}m_{03}m_{12}m_{13}k_2k_3sb_2)=1.\\
\end{cases}
\end{equation}
and
\begin{equation}\label{heightconditions4}
\|2^{\sigma_0}k_0l_0'l_0''m_{02}m_{03}b_0^2,2^{\sigma_1}k_1l_1'l_1''m_{12}m_{13}b_1^2\|\cdot\|2^{\sigma_2}k_2l_2'l_2''m_{02}m_{12}b_2^2,2^{\sigma_3}k_3l_3'l_3''m_{03}m_{13}b_3^2\|\leq B.
\end{equation}
Next, we swap the summation order of the $k_i$ and $l_i'$. We obtain:
\begin{align*}
    H_{r,2}(\b{d},\Tilde{\b{d}},\b{K},\b{L},B) &=  \sum_{r\leq z_0} \mu(r) \mathop{\sum\sum\sum\sum}_{\substack{\|l_0',l_1'\|,\|l_2',l_3'\|\leq z_0 \\ p|l_0'l_1'l_2'l_3'\Rightarrow p|r\\ r^2|l_0'l_1'l_2'l_3', \eqref{l'gcdcondition2}}} \frac{\chi_{\b{d}}(\b{l'})}{\tau(l_0')\tau(l_1')\tau(l_2')\tau(l_3')}H'_{r,2}(\b{d},\Tilde{\b{d}},\b{K},\b{L},\b{l}',B)\\
    &+O_{A}\left(\frac{B^2(\log B)(\log\log B)^4}{M_0M_1M_2M_3z_0^{1/3}}\right)
\end{align*}
where
\[
H'_{r,2}(\b{d},\Tilde{\b{d}},\b{K},\b{L},\b{l}',B) = \mathop{\sum\sum\sum\sum}_{\substack{\|k_0,k_1\|,\|k_2,k_3\|\leq z_1 \\ \b{k}\equiv\b{K}\bmod{8} \\\eqref{kgcdconditions2}}}\frac{\mu^2(k_0k_1k_2k_3)\chi_{\Tilde{\b{d}},l_0'l_1'}(k_2k_3)\chi_{\Tilde{\b{d}},l_3'l_2'}(k_0k_1)}{\tau(k_0k_1k_2k_3)} H_{r,2}(\b{d},\Tilde{\b{d}},\b{K},\b{L},\b{k},\b{l}',B),
\]
\begin{equation}\label{l'gcdcondition2} \begin{cases}
\gcd(l_0',2^{\sigma_1}m_{02}m_{03}m_{12}m_{13}b_1)=\gcd(l_1',2^{\sigma_0}m_{02}m_{03}m_{12}m_{13}b_0)=1,\\
\gcd(l_2',2^{\sigma_3}m_{02}m_{03}m_{12}m_{13}b_3)=\gcd(l_3',2^{\sigma_2}m_{02}m_{03}m_{12}m_{13}b_2)=1,\\
\end{cases}
\end{equation}
\begin{equation}\label{kgcdconditions2} \begin{cases}
\gcd(k_0,2^{\sigma_1}m_{02}m_{03}m_{12}m_{13}l_0'l_1'b_1)=\gcd(k_1,2^{\sigma_0}m_{02}m_{03}m_{12}m_{13}l_0'l_1'b_0)=1,\\
\gcd(k_2,2^{\sigma_3}m_{02}m_{03}m_{12}m_{13}l_2'l_3'b_3)=\gcd(k_3,2^{\sigma_2}m_{02}m_{03}m_{12}m_{13}l_2'l_3'b_2)=1,\\
\end{cases}
\end{equation}
and for any $Q,n\in \N_{\textrm{odd}}$, $\b{d},\b{l}'\in(\N\setminus\{0\})^{4}$ we have set
\[
\Tilde{\chi}_{\b{d}}(\b{l}') = \left(\frac{l_0'l_1'l_2'l_3'}{d_{02}d_{03}d_{12}d_{13}}\right)\;\;\text{and}\;\;
\Tilde{\chi}_{\b{d},Q}(n) = \left(\frac{d_{02}d_{03}d_{12}d_{13}Q}{n}\right).
\]
Now we claim that the sums $H'_{r,2}(\b{d},\Tilde{\b{d}},\b{K},\b{L},\b{l}',B)$ are either of the form considered in Proposition \ref{symmetrictypeaverage1} or of the form considered in Proposition \ref{Asymmetrictypeaverage1}. To do so we compare notation as follows:
\begin{itemize}
    \item The $n_i$ in Propostions \ref{symmetrictypeaverage1} and \ref{Asymmetrictypeaverage1} correspond to the $l_i''$;
    \item The $m_i$ in Propostions \ref{symmetrictypeaverage1} and \ref{Asymmetrictypeaverage1} correspond to the $k_i$.
    \item The $d_i$ in Propostions \ref{symmetrictypeaverage1} and \ref{Asymmetrictypeaverage1} correspond to $l_i'$.
    \item The $c_i$ in Propostions \ref{symmetrictypeaverage1} and \ref{Asymmetrictypeaverage1} correspond to the product $l_i'M_i$.
    \item The $Q_i$ in Propostions \ref{symmetrictypeaverage1} and \ref{Asymmetrictypeaverage1} correspond to products of $l_i'$ and $d_{ij}$, though we note specifically that the product corresponding to $Q_0$ and $Q_2$ in Proposition \ref{symmetrictypeaverage1} are independent of the $l_i'$, and that $Q_1$ from Proposition \ref{Asymmetrictypeaverage1} is equal to $1$ all cases where this proposition is applied (it is the product of the $d_{ij}$ here). All characters in this application are Jacobi symbols of the corresponding modulus.
    \item The $r_i$ in Propostions \ref{symmetrictypeaverage1} and \ref{Asymmetrictypeaverage1} are $\frac{m_{02}m_{03}m_{12}m_{13}sb_j}{d_{02}d_{03}d_{12}d_{13}}$, with $j=1,0,3,2$ for $i=0,1,2,3$ respectively.
\end{itemize}
Using this dictionary we find that conditions \ref{mconditions}, \ref{nontrivialcharacterconditions1} and \ref{kgcdconditions2} ensure that at least one of the characters $\chi_{\b{d},k_0k_1}$ or $\chi_{\b{d},k_2k_3}$ is non-trivial, ensuring that at least one of these results can be used. For the cases which we use Proposition \ref{Asymmetrictypeaverage1}, which are the cases when
\[
d_{02}d_{03}d_{12}d_{13} = 1,\;\|k_0,k_1\|=1,\;\|k_2,k_3\|>1\;\;\text{and}\;\;
d_{02}d_{03}d_{12}d_{13} = 1,\;\|k_0,k_1\|>1,\;\|k_2,k_3\|=1,
\]
we note also that the constants $c_i$, here given by $l_i'M_i$, are all $\ll (\log B)^{5A}$, the constants $d_i$ are $l_0',l_1',l_2',l_3'\leq (\log B)^{A}$ and the lower bound in the inner most sum $H_{r,2}(\b{d},\Tilde{\b{d}},\b{K},\b{L},\b{k},\b{l}',B)$ is $(\log B)^{150A}$, meaning that the constants satisfy the desired bounds. Applying these propositions then give:
\begin{align*}
H'_{r,2}(\b{d},\Tilde{\b{d}},\b{K},\b{L},\b{l}',B) &\ll_{A} \frac{\mathds{1}(\b{d}=\b{1})\tau(m_0m_1m_2m_3)^2\tau(s)^2\tau(b_0)\tau(b_1)\tau(b_2)\tau(b_3)B^2}{l_0'l_1'l_2'l_3'M_0M_1M_2M_3(\log B)(\log \log B)^{132A}} \\&+ \frac{d_{02}^2d_{03}^2d_{12}^2d_{13}^2B^2}{l_0'l_1'l_2'l_3'M_0M_1M_2M_3(\log X)^{66A}}.
\end{align*}
Therefore we may deduce
\begin{align*}
 H_{r,2}(\b{d},\Tilde{\b{d}},\b{K},\b{L},B) \ll_{A} \hspace{-5pt}\frac{\mathcal{R}B^2\mathcal{MAX}_1(B)}{M_0M_1M_2M_3}
\end{align*}
where
\begin{align}\label{breakingsquarefreeconstant}
\mathcal{R}=\sum_{s\leq z_0} \mathop{\sum\sum\sum\sum}_{\substack{\|l_0',l_1'\|,\|l_2',l_3'\|\leq z_0 \\ p|l_0'l_1'l_2'l_3'\Rightarrow p|s\\ s^2|l_0'l_1'l_2'l_3'}} \frac{\tau(s)^2}{l_0'l_1'l_2'l_3'\tau(l_0')\tau(l_1')\tau(l_2')\tau(l_3')}.
\end{align}
To conclude the proof we show that $\mathcal{R}\ll 1$. We have $\tau(l_0')\tau(l_1')\tau(l_2')\tau(l_3')\geq \tau(l_0'l_1'l_2'l_3')$. Then, by writing $u=l_0'l_1'l_2'l_3'$, we have
\[
\mathcal{R} \ll \sum_{s\leq z_0} \sum_{\substack{u\leq z_0^4\\ p|u\Rightarrow p|s\\s^2|u}}\frac{\tau(s)^2\tau_4(u)}{u\tau(u)}\ll \sum_{s\leq z_0} \sum_{\substack{u\leq z_0^4\\ p|u\Rightarrow p|s\\s^2|u}}\frac{\tau(s)^2}{u^{3/4}}
\]
where in the last step above we have used the bound $\tau_4(u)\leq (\tau(u))^4 \ll \tau(u)u^{1/4}$. Now using Lemma $5.7$ from \cite{LRS} with $\epsilon=1/4$ we have,
\[
\sum_{\substack{u\leq z_0^4\\ p|u\Rightarrow p|s\\s^2|u}}\frac{1}{u^{3/4}} \ll \frac{1}{s^{5/4}},
\]
thus $\mathcal{R} \ll \sum_{s\leq z_0}\frac{\tau(s)^2}{s^{5/4}}\ll 1$.
\end{proof}

\begin{proposition}\label{Errorterms2}
    Fix some $\b{b}\in\N^4$. Then
    \[
    \mathcal{E}_{r,2}(B,\b{b}) \ll_{A} \frac{\tau(b_0)\tau(b_1)\tau(b_2)\tau(b_3)B^2}{b_0^2b_1^2b_2^2b_3^2(\log B)(\log\log B)^{66A}}.
    \]
\end{proposition}

\begin{proof}
Recall $\mathcal{E}_{r,2}(B,\b{b})$ from \eqref{Errorterms1and2}. Now we apply the Lemma \ref{Errorterm2boundlemma} and use trivial bounds for the finite sums over $\boldsymbol{\sigma},\b{q},\b{K}$ and $\b{L}$. This will give:
\[
\mathcal{E}_{r,2}(B,\b{b}) \ll_{A} B^2\mathop{\sum}_{\substack{\b{m}\in\N^4\\ m_{ij}\leq z_0\\ \eqref{mconditions}}} \mathop{\sum\sum}_{\substack{\b{d},\Tilde{\b{d}}\in\N^4\\ d_{ij}\Tilde{d}_{ij}=(m_{ij})_{\textrm{odd}}}} \frac{\mathcal{MAX}_1(B)}{m_{02}^2m_{03}^2m_{12}^2m_{13}^2b_0^2b_1^2b_2^2b_3^2}
\]
since $M_0M_1M_2M_3 = 2^{\sigma_0+\sigma_1+\sigma_2+\sigma_3}m_{02}^2m_{03}^2m_{12}^2m_{13}^2b_0^2b_1^2b_2^2b_3^2$.
If $\b{d}=1$ then,
\begin{align*}
\mathop{\sum}_{\substack{\b{m}\in\N^4\\ m_{ij}\leq z_0\\ \eqref{mconditions}}} \mathcal{MAX}_1(B) &\ll \mathop{\sum}_{\substack{\b{m}\in\N^4\\ m_{ij}\leq z_0\\ \eqref{mconditions}}} \frac{\tau(m_0m_1m_2m_3)\tau(b_0)\tau(b_1)\tau(b_2)\tau(b_3)}{m_{02}^2m_{03}^2m_{12}^2m_{13}^2b_0^2b_1^2b_2^2b_3^2(\log B)(\log \log B)^{66A}}\\ &\ll \frac{\tau(b_0)\tau(b_1)\tau(b_2)\tau(b_3)}{b_0^2b_1^2b_2^2b_3^2(\log B)(\log \log B)^{66A}}.
\end{align*}
Otherwise,
\begin{align*}
\mathop{\sum}_{\substack{\b{m}\in\N^4\\ m_{ij}\leq z_0\\ \eqref{mconditions}}} \mathop{\sum\sum}_{\substack{\b{d},\Tilde{\b{d}}\in\N^4\\ d_{ij}\Tilde{d}_{ij}=(m_{ij})_{\textrm{odd}}}} \mathcal{MAX}_1(B) &\ll \mathop{\sum}_{\substack{\b{m}\in\N^4\\ m_{ij}\leq z_0\\ \eqref{mconditions}}} \frac{1}{b_0^2b_1^2b_2^2b_3^2(\log X)^{132A}} + \frac{(\log B)(\log\log B)^4}{m_{02}^2m_{03}^2m_{12}^2m_{13}^2b_0^2b_1^2b_2^2b_3^2(\log X)^{A/3}}\\
&\ll \frac{1}{b_0^2b_1^2b_2^2b_3^2(\log B)^{A/3 - 2}}.
\end{align*}
\end{proof}

As alluded to above, we may use the same argument with the variables $\b{d},\b{K}$ and $\b{k}$ switching roles with the variables $\Tilde{\b{d}},\b{L}$ and $\b{l}$ to obtain,

\begin{proposition}\label{Errorterms3}
    Fix some $\b{b}\in\N^4$. Then
    \[
    \mathcal{E}_{r,3}(B,\b{b}) \ll_{A} \frac{\tau(b_0)\tau(b_1)\tau(b_2)\tau(b_3)B^2}{b_0^2b_1^2b_2^2b_3^2(\log B)(\log\log B)^{66A}}.
    \]
\end{proposition}

\subsection{The Error Terms $\mathcal{E}_{r,4}(B,\b{b})$ and $\mathcal{E}_{r,5}(B,\b{b})$}\label{smallconductor45} For these error terms we note that the conditions on the variables $\b{m},\b{K},\b{L},\b{d}$ and $\Tilde{\b{d}}$ guarantee that the sums $H_{r,4}(\b{d},\Tilde{\b{d}},\b{K},\b{L},B)$ and $H_{r,5}(\b{d},\Tilde{\b{d}},\b{K},\b{L},B)$ are of types covered by Propositions \ref{symmetrictypeaverage2} or \ref{Asymmetrictypeaverage2}. Similar to the symmetry of $\mathcal{E}_{r,2}(B,\b{b})$ and $\mathcal{E}_{r,3}(B,\b{b})$ in the last section, $\mathcal{E}_{r,4}(B,\b{b})$ and $\mathcal{E}_{r,5}(B,\b{b})$ are symmetrically equivalent, the latter being of the same form of the former with the variables $k_2,l_2,k_3,l_3,K_2,L_2,K_3,L_3$ switching roles with $k_0,l_0,k_1,l_1,K_0,L_0,K_1,L_1$. We will therefore focus on $\mathcal{E}_{r,4}(B,\b{b})$. We first examine $H_{r,4}(\b{d},\Tilde{\b{d}},\b{K},\b{L},B)$.

\begin{lemma}\label{Errorterm4boundLemma}
    Fix some $\b{b}\in\N^4$, some $\b{m}\in\N^4$ satisfying \eqref{mconditions}, some $\boldsymbol{\sigma}\in\{0,1\}^4$ satisfying \eqref{sigmaconditions} and some $\b{q}\in\mathcal{A}(\b{m},\boldsymbol{\sigma})$. Suppose that $\b{K},\b{L}\in(\Z/8\Z)^{*4}$ and $\b{d},\Tilde{\b{d}}\in\N^4$ satisfying the conditions
    \begin{equation}\label{nontrivialcharacterconditions2}
    \begin{cases}
        \b{K},\b{L}\;\text{satisfy \eqref{twoadicconditions2}},\\
        d_{ij}\Tilde{d}_{ij}=(m_{ij})_{\textrm{odd}}\;\forall\; ij\in\{02,03,12,13\},\\
        \b{d}=\Tilde{\b{d}}=\b{1} \Rightarrow \;\text{one of}\; K_2,L_2,K_3,L_3\not\equiv 1\bmod{8}.
    \end{cases}
    \end{equation}
    Then
    \begin{align*}
    H_{r,4}(\b{d},\Tilde{\b{d}},\b{K},\b{L},B) &\ll_{A} \frac{B^2\mathcal{MAX}_2(B)}{M_0M_1M_2M_3(\log B)}
    \end{align*}
    where we define $\mathcal{MAX}_2(B)$ as
    \begin{align*}
    \max\left\{\mathds{1}(\b{d}=\Tilde{\b{d}}=\b{1})\tau(b_0)\tau(b_1)\tau(b_2)\tau(b_3)\sqrt{\log\log B},\frac{(\log B)^3(\log\log B)^4}{(\log B)^{A/3}},\frac{d_{02}^2d_{03}^2d_{12}^2d_{13}^2(\log B)}{(\log B)^{140A}}\right\}.
    \end{align*}
\end{lemma}

\begin{proof}
Recall \eqref{oddJacobisymbols},\eqref{H4} and \eqref{generalinnersum}. We order $H_{r,4}(\b{d},\Tilde{\b{d}},\b{K},\b{L},B)$ to sum over $k_0,l_0,k_1$ and $l_1$ first and therefore write
\[
\Theta_2(\b{d},\Tilde{\b{d}},\b{k},\b{l}) = \chi_{\Tilde{\b{d}}}(k_2k_3)\chi_{\b{d}}(l_2l_3)\chi_{\Tilde{\b{d}},l_2l_3}(k_0k_1)\chi_{\b{d},k_2k_3}(l_0l_1)
\]
where
\[
\chi_{\Tilde{\b{d}}}(k_2k_3) = \left(\frac{\Tilde{d}_{02}\Tilde{d}_{03}\Tilde{d}_{12}\Tilde{d}_{13}}{k_2k_3}\right),\;\;
\chi_{\b{d}}(l_2l_3) = \left(\frac{l_2l_3}{d_{02}d_{03}d_{12}d_{13}}\right),
\]
\[
\chi_{\Tilde{\b{d}},l_2l_3}(k_0k_1) = \left(\frac{\Tilde{d}_{02}\Tilde{d}_{03}\Tilde{d}_{12}\Tilde{d}_{13}l_2l_3}{k_0k_1}\right),\;\;
\chi_{\b{d},k_2k_3}(l_0l_1) = \left(\frac{l_0l_1}{d_{02}d_{03}d_{12}d_{13}k_2k_3}\right).
\]
Then
\[
H_{r,4}(\b{d},\Tilde{\b{d}},\b{K},\b{L},B) = \hspace{-2.5pt}\mathop{\sum\sum\sum\sum}_{\substack{\|k_2,k_3\|,\|l_2,l_3\|\leq z_1\\ k_i\equiv K_i\bmod{8}\;\forall i\in\{2,3\}\\ l_i\equiv L_i\bmod{8}\;\forall i\in\{2,3\}\\ \eqref{23gcdconditions}}}\hspace{-2.7pt}\frac{\mu^2(k_2l_2k_3l_3)\chi_{\Tilde{\b{d}}}(k_2k_3)\chi_{\b{d}}(l_2l_3)}{\tau(k_2l_2k_3l_3)}H_{r,4}(\b{d},\Tilde{\b{d}},\b{K},\b{L},\b{kl}_{23},B)
\]
where $\b{kl}_{23}=(k_2,l_2,k_3,l_3)$, 
\[
H_{r,4}(\b{d},\Tilde{\b{d}},\b{K},\b{L},\b{kl}_{23},B) = \mathop{\sum\sum\sum\sum}_{\substack{k_0,l_0,k_1,l_1\in\N\\ k_i\equiv K_i\bmod{8}\;\forall i\in\{0,1\}\\ l_i\equiv L_i\bmod{8}\;\forall i\in\{0,1\}\\ \eqref{heightconditions3},\eqref{01gcdconditions}}}\frac{\mu^2(k_0l_0k_1l_1)}{\tau(k_0)\tau(l_0)\tau(k_1)\tau(l_1)}\chi_{\Tilde{\b{d}},l_2l_3}(k_0k_1)\chi_{\b{d},k_2k_3}(l_0l_1),
\]
\begin{equation}\label{23gcdconditions}
\gcd(k_2l_2,2^{\sigma_3}m_{02}m_{03}m_{12}m_{13}b_3)=\gcd(k_3l_3,2^{\sigma_2}m_{02}m_{03}m_{12}m_{13}b_2)=1,   
\end{equation}
and
\begin{equation}\label{01gcdconditions} \begin{cases}
\gcd(k_0,2^{\sigma_1}m_{02}m_{03}m_{12}m_{13}k_2k_3b_1)=\gcd(k_1,2^{\sigma_0}m_{02}m_{03}m_{12}m_{13}k_2k_3b_0)=1,\\
\gcd(l_0,2^{\sigma_1}m_{02}m_{03}m_{12}m_{13}l_2l_3b_1)=\gcd(l_1,2^{\sigma_0}m_{02}m_{03}m_{12}m_{13}l_2l_3b_0)=1.
\end{cases}
\end{equation}
Next we aim to remove the $\mu^2$ term in these inner sums. For this we use Lemma \ref{removesquarefree2} with $w_0=w_1=z_0$, $c_{01}=M_1$, $c_{02}=M_2$ and $M = \|k_2l_2M_2,k_3l_3M_3\|$. Then $H_{r,4}(\b{d},\Tilde{\b{d}},\b{K},\b{L},\b{kl}_{23},B)$ is equal to
\begin{align*}
     &\sum_{s\leq z_0}\mu(s)\mathop{\sum\sum\sum\sum}_{\substack{\substack{k_0',l_0',k_1',l_1'\leq z_0 \\ p|k_0'l_0'k_1'l_1'\Rightarrow p|s\\ s^2|k_0'l_0'k_1'l_1',\;\eqref{01'gcdconditions}}}}\frac{\chi_{\Tilde{\b{d}},l_2l_3}(k_0'k_1')\chi_{\b{d},k_2k_3}(l_0'l_1')}{\tau(k_0')\tau(l_0')\tau(k_1')\tau(l_1')}H_{r,4}(\b{d},\Tilde{\b{d}},\b{K},\b{L},\b{kl}_{23},\b{kl}'_{01},B)\\
     &+ O\left(\frac{B^2(\log B)^2}{k_2l_2k_3l_3M_0M_1M_2M_3z_0^{1/3}}\right)
\end{align*}
where $\b{kl}'_{01}=(k_0',l_0',k_1',l_1')$ and
\begin{align*}
    H_{r,4}(\b{d},\Tilde{\b{d}},\b{K},\b{L},\b{kl}_{23},\b{kl}'_{01},B) &= \mathop{\sum\sum\sum\sum}_{\substack{k_0'',l_0'',k_1'',l_1''\in\N\\ k_i\equiv K_i\bmod{8}\;\forall i\in\{0,1\}\\ l_i\equiv L_i\bmod{8}\;\forall i\in\{0,1\}\\ \eqref{heightconditions5},\eqref{01''gcdconditions}}}\frac{\chi_{\Tilde{\b{d}},l_2l_3}(k_0''k_1'')\chi_{\b{d},k_2k_3}(l_0''l_1'')}{\tau(k_0'')\tau(l_0'')\tau(k_1'')\tau(l_1'')},
\end{align*}
\begin{equation}\label{heightconditions5}
\|2^{\sigma_0}k_0'k_0''l_0'l_0''m_{02}m_{03}b_0^2,2^{\sigma_1}k_1'k_1''l_1'l_1''m_{12}m_{13}b_1^2\|\cdot\|2^{\sigma_2}k_2l_2m_{02}m_{12}b_2^2,2^{\sigma_3}k_3l_3m_{03}m_{13}b_3^2\|\leq B,
\end{equation}\\
\begin{equation}\label{01'gcdconditions} \begin{cases}
\gcd(k_0',2^{\sigma_1}m_{02}m_{03}m_{12}m_{13}k_2k_3b_1)=\gcd(k_1',2^{\sigma_0}m_{02}m_{03}m_{12}m_{13}k_2k_3b_0)=1,\\
\gcd(l_0',2^{\sigma_1}m_{02}m_{03}m_{12}m_{13}l_2l_3b_1)=\gcd(l_1',2^{\sigma_0}m_{02}m_{03}m_{12}m_{13}l_2l_3b_0)=1,
\end{cases}
\end{equation}
and
\begin{equation}\label{01''gcdconditions} \begin{cases}
\gcd(k_0'',2^{\sigma_1}m_{02}m_{03}m_{12}m_{13}k_2k_3sb_1)=\gcd(k_1'',2^{\sigma_0}m_{02}m_{03}m_{12}m_{13}k_2k_3sb_0)=1,\\
\gcd(l_0'',2^{\sigma_1}m_{02}m_{03}m_{12}m_{13}l_2l_3sb_1)=\gcd(l_1'',2^{\sigma_0}m_{02}m_{03}m_{12}m_{13}l_2l_3sb_0)=1.
\end{cases}
\end{equation}
Now we swap the summation of $k_0',l_0',k_1',l_1'$ with $k_2,l_2,k_3,l_3$. This will yield,
\begin{align*}
    H_{r,4}(\b{d},\Tilde{\b{d}},\b{K},\b{L},B) &= \sum_{s\leq z_0}\mu(s)\mathop{\sum\sum\sum\sum}_{\substack{\substack{k_0',l_0',k_1',l_1'\leq z_0 \\ p|k_0'l_0'k_1'l_1'\Rightarrow p|s\\ s^2|k_0'l_0'k_1'l_1',\;\eqref{01'gcdconditions2}}}}\frac{\chi_{\b{d}}(l_0'l_1')\chi_{\Tilde{\b{d}}}(k_0'k_1')}{\tau(k_0')\tau(l_0')\tau(k_1')\tau(l_1')}H_{r,4}(\b{d},\Tilde{\b{d}},\b{K},\b{L},\b{kl}'_{01},B)\\
    &+ O\left(\frac{B^2(\log B)^2(\log\log B)^4}{M_0M_1M_2M_3z_0^{1/3}}\right)
\end{align*}
where $H_{r,4}(\b{d},\Tilde{\b{d}},\b{K},\b{L},\b{kl}'_{01},B)$ is equal to
\begin{align*}
\hspace{-10pt}\mathop{\sum\sum\sum\sum}_{\substack{\|k_2,k_3\|,\|l_2,l_3\|\leq z_1\\ k_i\equiv K_i\bmod{8}\;\forall i\in\{2,3\}\\ l_i\equiv L_i\bmod{8}\;\forall i\in\{2,3\}\\ \eqref{23gcdconditions2}}}\hspace{-10pt}&\frac{\mu^2(k_2l_2k_3l_3)\chi_{\Tilde{\b{d}},l_0'l_1'}(k_2k_3)\chi_{\b{d},k_0'k_1'}(l_2l_3)}{\tau(k_2l_2k_3l_3)} H_{r,4}(\b{d},\Tilde{\b{d}},\b{K},\b{L},\b{kl}_{23},\b{kl}'_{01},B),
\end{align*}
\begin{equation}\label{01'gcdconditions2} \begin{cases}
\gcd(k_0',2^{\sigma_1}m_{02}m_{03}m_{12}m_{13}b_1)=\gcd(k_1',2^{\sigma_0}m_{02}m_{03}m_{12}m_{13}b_0)=1,\\
\gcd(l_0',2^{\sigma_1}m_{02}m_{03}m_{12}m_{13}b_1)=\gcd(l_1',2^{\sigma_0}m_{02}m_{03}m_{12}m_{13}b_0)=1,
\end{cases}\\
\end{equation}
and
\begin{equation}\label{23gcdconditions2}
\gcd(k_2l_2,2^{\sigma_3}m_{02}m_{03}m_{12}m_{13}sb_3)=\gcd(k_3l_3,2^{\sigma_2}m_{02}m_{03}m_{12}m_{13}sb_2)=1.
\end{equation}

Now note that the condition \eqref{nontrivialcharacterconditions2} guarantees that $H_{r,4}(\b{d},\Tilde{\b{d}},\b{K},\b{L},\b{kl}'_{01},B)$ is of the form considered in either Proposition \ref{symmetrictypeaverage2} or Proposition \ref{Asymmetrictypeaverage2} as it guarantees that either $\chi_{\Tilde{\b{d}},l_2l_3}$ or $\chi_{\b{d},k_2k_3}$ is non-principal. The notational dictionary is as follows:
\begin{itemize}
    \item the $n_i$ in Propositions \ref{symmetrictypeaverage2} and \ref{Asymmetrictypeaverage2} correspond to $k_0'',l_0'',k_1'',l_1''$;
    \item the $m_i$ in Propositions \ref{symmetrictypeaverage2} and \ref{Asymmetrictypeaverage2} correspond to $k_2,l_2,k_3,l_3$;
    \item the quadruple $(c_0,c_1,c_2,c_3)$ in Proposition \ref{symmetrictypeaverage2} and the quadruple $(c_{01},c_{23},\Tilde{c}_0,\Tilde{c}_1)$ in Proposition \ref{Asymmetrictypeaverage2} both correspond to the quadruple $(k_0'l_0'M_0,k_1'l_1'M_1,M_2,M_3)$;
    \item the $Q_{02}$ and $Q_{13}$ in Proposition \ref{symmetrictypeaverage2} correspond to $d_{02}d_{03}d_{12}d_{13}$ and $\Tilde{d}_{02}\Tilde{d}_{03}\Tilde{d}_{12}\Tilde{d}_{13}$ respectively;
    \item the $r_i$ in Propositions \ref{symmetrictypeaverage2} and \ref{Asymmetrictypeaverage2} correspond to $m_{02}m_{03}m_{12}m_{13}sb_j$ divided by $d_{02}d_{03}d_{12}d_{13}$ or $\Tilde{d}_{02}\Tilde{d}_{03}\Tilde{d}_{12}\Tilde{d}_{13}$ depending on whether $n_i$ corresponds to a $k''$-variable or a $l''$-variable and where $j=0$ or $1$ depending on whether the $n_i$ corresponds to a $''$ variable indexed by $1$ or $0$ respectively;
    \item the $\Tilde{r}_{i}$ in Proposition \ref{Asymmetrictypeaverage2} are just $m_{02}m_{03}m_{12}m_{13}sb_j$ for some $j$.
\end{itemize}

Finally we note that the characters $\chi_{\Tilde{\b{d}},l_0'l_1'}(k_2k_3)$ and $\chi_{\b{d},k_0'k_1'}(l_2l_3)$ in the sum over $k_2,l_2,k_3,l_3$ are of no import in these applications as we first use the triangle inequality to obtain the absolute value of the inner sums. Noting that we only need to apply Proposition \ref{Asymmetrictypeaverage2} when $\b{d}=\Tilde{\b{d}}=\b{1}$ we thus have that $H_{r,4}(\b{d},\Tilde{\b{d}},\b{K},\b{L},\b{kl}'_{01},B)$ is,
\begin{align*}
     \ll_{A}& \frac{\mathds{1}(\b{d}=\Tilde{\b{d}}=\b{1})\tau(b_0)\tau(b_1)\tau(b_2)\tau(b_3)\tau(s)^2B^2\sqrt{\log\log B}}{k_0'l_0'k_1'l_1'M_0M_1M_2M_3(\log B)}+\frac{d_{02}^2d_{03}^2d_{12}^2d_{13}^2B^2}{k_0'l_0'k_1'l_1'M_0M_1M_2M_3(\log B)^{140A}}.
\end{align*}
Now, the sum over $k_0',l_0',k_1',l_1'$ is given by
\[
\sum_{s\leq z_0}\mathop{\sum\sum\sum\sum}_{\substack{\substack{k_0',l_0',k_1',l_1'\leq z_0 \\ p|k_0'l_0'k_1'l_1'\Rightarrow p|s\\ s^2|k_0'l_0'k_1'l_1', \eqref{01'gcdconditions2}}}}\frac{\tau(s)^2}{\tau(k_0')\tau(l_0')\tau(k_1')\tau(l_1')k_0'l_0'k_1'l_1'}\ll \sum_{s\leq z_0}\sum_{\substack{u\leq z_0^4\\ p|u\Rightarrow p|s\\ s^2|u}}\frac{\tau(s)^2}{u^{3/4}} \ll 1.
\]
The details of this bound are the same as those bounding \eqref{breakingsquarefreeconstant} in the previous subsection. Substituting these bounds into $H_{r,4}(\b{d},\Tilde{\b{d}},\b{K},\b{L},B)$ concludes the result.
\end{proof}

\begin{proposition}\label{Errorterms4}
    Fix some $\b{b}\in\N^4$. Then
    \[
    \mathcal{E}_{r,4}(B,\b{b}) \ll_{A} \frac{\tau(b_0)\tau(b_1)\tau(b_2)\tau(b_3)B^2\sqrt{\log\log B}}{b_0^2b_1^2b_2^2b_3^2(\log B)}.
    \]
\end{proposition}

\begin{proof}
Recall $\mathcal{E}_{r,4}(B,\b{b})$ from \eqref{Errorterms3and4}. Then, upon applying the Lemma \ref{Errorterm4boundLemma} we are left with
\[
\mathcal{E}_{r,4}(B,\b{b}) \ll_{A} \frac{B^2}{(\log B)}\mathop{\sum}_{\substack{\b{m}\in\N^4\\ m_{ij}\leq z_0\\ \eqref{mconditions}}} \mathop{\sum}_{\substack{\boldsymbol{\sigma}\in\{0,1\}^4\\ \eqref{sigmaconditions}}}\mathop{\sum}_{\substack{\b{q}\in\mathcal{A}(\b{m},\boldsymbol{\sigma})}}\mathop{\sum\sum}_{\substack{\b{K},\b{L}\in(\Z/8\Z)^{*4}\\ \eqref{twoadicconditions2}}}\mathop{\sum\sum}_{\substack{\b{d},\Tilde{\b{d}}\in\N^4\\ d_{ij}\Tilde{d}_{ij}=(m_{ij})_{\textrm{odd}}}}\frac{\mathcal{MAX}_2(B)}{M_0M_1M_2M_3}.
\]
Now, since $\mathds{1}(\b{d}=\Tilde{\b{d}}=\b{1})=1$ if and only if $\b{m}_{\textrm{odd}}=\b{1}$, and since \eqref{mconditions} guarantees that $\mu^2(m_{02}m_{03}m_{12}m_{13})=1$, we can only have $\b{m}$ equal to $(1,1,1,1),(2,1,1,1),(1,2,1,1),(1,1,2,1)$ or $(1,1,1,2)$ when this condition holds. Thus there are only finitely many $\b{m}$ to consider and since there are only finitely many $\b{q},\b{K},\b{L}$ and $\boldsymbol{\sigma}$ we have,
\[
\mathop{\sum}_{\substack{\b{m}\in\N^4\\ m_{ij}\leq z_0\\ \eqref{mconditions},\b{m}_{\text{odd}}=\b{1}}} \mathop{\sum}_{\substack{\boldsymbol{\sigma}\in\{0,1\}^4\\ \eqref{sigmaconditions}}}\mathop{\sum}_{\substack{\b{q}\in\mathcal{A}(\b{m},\boldsymbol{\sigma})}}\mathop{\sum\sum}_{\substack{\b{K},\b{L}\in(\Z/8\Z)^{*4}\\ \eqref{twoadicconditions2}}}\frac{\mathcal{MAX}_2(B)}{M_0M_1M_2M_3} \ll \frac{\tau(b_0)\tau(b_1)\tau(b_2)\tau(b_3)\sqrt{\log\log B}}{b_0^2b_1^2b_2^2b_3^2}.
\]
Otherwise, we remark that there are only finitely many $\b{q},\b{K},\b{L}$ and $\boldsymbol{\sigma}$ and that the sum over $\frac{1}{m_{ij}^2}$ converges so that the expression becomes bounded by
\begin{align*}
\ll \frac{(\log B)^3(\log\log B)^4}{b_0^2b_1^2b_2^2b_3^2(\log B)^{A/3}} + \frac{(\log B)^{4A}}{b_0^2b_1^2b_2^2b_3^2(\log B)^{140A}}
\end{align*}
Noting that by choosing $A$ to be sufficiently large and that each $b_i\leq (\log B)^{A}$, we obtain the result.
\end{proof}

It follows by the same argument with the variables $k_2,l_2,k_3,l_3,K_2,L_2,K_3,L_3$ switching roles with $k_0,l_0,k_1,l_1,K_0,L_0,K_1,L_1$ that:

\begin{proposition}\label{Errorterms5}
    Fix some $\b{b}\in\N^4$. Then
    \[
    \mathcal{E}_{r,5}(B,\b{b}) \ll_{A} \frac{\tau(b_0)\tau(b_1)\tau(b_2)\tau(b_3)B^2\sqrt{\log\log B}}{b_0^2b_1^2b_2^2b_3^2(\log B)}.
    \]
\end{proposition}

\section{Vanishing Main Terms}\label{vanishingmainterm}
In this section we handle the vanishing main terms $\mathcal{V}_{r,4}(B,\b{b})$ and $\mathcal{V}_{r,5}(B,\b{b})$. As in the arguments of the previous section these are similar, almost obtained from each other by switching the roles of $k_2,l_2,k_3,l_3,K_2,L_2,K_3,L_3$ with $k_0,l_0,k_1,l_1,K_0,L_0,K_1,L_1$. The key obstruction to this is condition \eqref{nonsquareconditions3}, which creates an asymmetry in this ``role'' switching of variables when $r=1$. For this reason, and the fact that the even characters found in $\Theta_{r,1}$ will play a key role in the following arguments and this function changes with the value of $r$, we must separate the cases $r=1$ and $r=2$.

\subsection{The Vanishing Main Term $\mathcal{V}_{1,4}(B,\b{b})$}\label{vanish14} We begin with an examination of the inner sums $H_{1,4}(\b{1},\b{1},\b{K},\b{L},B)$. Recall that
\begin{align*}
H_{1,4}(\b{1},\b{1},\b{K},\b{L},B) = \mathop{\sum\sum\sum\sum}_{\substack{(\b{k},\b{l})\in\mathcal{H}_4\\(\b{k},\b{l})\equiv(\b{K},\b{L})\bmod{8}\\\eqref{gcdconditionsandsquarefree3},\;\eqref{heightconditions3},\;\eqref{nonsquareconditions3}}} \frac{\mu^2(2k_0l_0k_1l_1k_2l_2k_3l_3)}{\tau\left(k_0l_0k_1l_1k_2l_2k_3l_3\right)}\left(\frac{l_0l_1}{k_2k_3}\right)\left(\frac{l_2l_3}{k_0k_1}\right).
\end{align*}
We separate out the terms for which $k_2=l_2=k_3=l_3=1$:
\[
H_{1,4}(\b{1},\b{1},\b{K},\b{L},B) = V_{1,4}(\b{K},\b{L},B) + EV_{1,4}(\b{K},\b{L},B)
\]
where
\[
V_{1,4}(\b{K},\b{L},B) = \mathop{\sum\sum\sum\sum}_{\substack{k_0,l_0,k_1,l_1\in\N^4\\(k_i,l_i)\equiv(K_i,L_i)\bmod{8}\;\forall\;i\in\{0,1\}\\\eqref{gcdconditionsandsquarefree3},\;\eqref{nonsquareconditions3},\;\eqref{heightconditionsvanishingmainterm14}}} \frac{\mu^2(2k_0l_0k_1l_1)}{\tau\left(k_0l_0k_1l_1\right)},
\]
\[
EV_{1,4}(\b{K},\b{L},B) = \mathop{\sum\sum\sum\sum}_{\substack{(\b{k},\b{l})\in\mathcal{H}_4\\(\b{k},\b{l})\equiv(\b{K},\b{L})\bmod{8}\\ k_2l_2k_3l_3\neq 1 \\\eqref{gcdconditionsandsquarefree3},\;\eqref{heightconditions3},\;\eqref{nonsquareconditions3}}} \frac{\mu^2(2k_0l_0k_1l_1k_2l_2k_3l_3)}{\tau\left(k_0l_0k_1l_1k_2l_2k_3l_3\right)}\left(\frac{l_0l_1}{k_2k_3}\right)\left(\frac{l_2l_3}{k_0k_1}\right),
\]
and
\begin{equation}\label{heightconditionsvanishingmainterm14}
\|2^{\sigma_0}k_0l_0m_{02}m_{03}b_0^2,2^{\sigma_1}k_1l_1m_{12}m_{13}b_1^2\|\cdot\|2^{\sigma_2}m_{02}m_{12}b_2^2,2^{\sigma_3}m_{03}m_{13}b_3^2\|\leq B.
\end{equation}
Now $EV_{1,4}(\b{K},\b{L},B)$ may be treated like $H_{r,4}(\b{d},\Tilde{\b{d}},\b{K},\b{L},B)$ in \S \ref{smallconductor45}, by noting that, after breaking the $\mu^2$ function, the condition $k_2l_2k_3l_3\neq 1$ guarantees that Propositions \ref{symmetrictypeaverage2} and \ref{Asymmetrictypeaverage2} may be used. Summing over the $\b{m},\boldsymbol{\sigma},\b{q},\b{K}$ and $\b{L}$ as in Proposition \ref{Errorterms4} we are thus left with:
\begin{align*}
\mathcal{V}_{1,4}(B,\b{b}) &= \mathop{\sum}_{\substack{\b{m}\in\N^4,\eqref{mconditions}\\ (m_{02}m_{03}m_{12}m_{13})_{\textrm{odd}}=1}}\mathop{\sum}_{\substack{\boldsymbol{\sigma}\in\{0,1\}^4\\ \eqref{sigmaconditions}}}\mathop{\sum}_{\substack{\b{q}\in\mathcal{A}(\b{m},\boldsymbol{\sigma})}}\mathop{\sum\sum}_{\substack{\b{K},\b{L}\in(\Z/8\Z)^{*4}\\ \eqref{twoadicvanishingmaintermconditions4}}} \Theta_{1,1}(\b{1},\b{K},\boldsymbol{\sigma})V_{1,4}(\b{K},\b{L},B)\\ &+ O_{A}\left(\frac{\tau(b_0)\tau(b_1)\tau(b_2)\tau(b_3)B^2\sqrt{\log\log B}}{b_0^2b_1^2b_2^2b_3^2(\log B)}\right).
\end{align*}
Now we wish to re-integrate the even characters into the sum over $k_0,l_0,k_1,l_1$. By doing this we obtain that $\mathcal{V}_{1,4}(B,\b{b})$ is equal to
\begin{equation}\label{vanishingmainterm14checkpoint1}
 \mathop{\sum}_{\substack{\b{m}\in\N^4,\eqref{mconditions}\\ (m_{02}m_{03}m_{12}m_{13})_{\textrm{odd}}=1}}\hspace{-3pt}\mathop{\sum}_{\substack{\boldsymbol{\sigma}\in\{0,1\}^4\\ \eqref{sigmaconditions}}}\mathop{\sum}_{\substack{\b{q}\in\mathcal{A}(\b{m},\boldsymbol{\sigma})}}\hspace{-10pt} V'_{1,4}(B) + O_{A}\hspace{-3pt}\left(\frac{\tau(b_0)\tau(b_1)\tau(b_2)\tau(b_3)B^2\sqrt{\log\log B}}{b_0^2b_1^2b_2^2b_3^2(\log B)}\right)
\end{equation}
where
\[
V'_{1,4}(B) = \mathop{\sum\sum\sum\sum}_{\substack{k_0,l_0,k_1,l_1\in\N^4\\ \eqref{nonsquareconditions3},\;\eqref{heightconditionsvanishingmainterm14},\;\eqref{vanishingmaintermgcdconditionsandsquarefree14},\;\eqref{twoadicconditionsvanishingmainterm14unwrapped}}} \frac{\mu^2(2k_0l_0k_1l_1)}{\tau\left(k_0l_0k_1l_1\right)}(-1)^{\frac{(k_0k_1-1)}{2}}\left(\frac{2^{\sigma_2+\sigma_3+v_2(m_{02}m_{03}m_{12}m_{13})}}{k_0k_1}\right),
\]
\begin{equation}\label{vanishingmaintermgcdconditionsandsquarefree14}
\begin{cases}
\gcd(k_0l_0,2^{\sigma_1}m_{02}m_{03}m_{12}m_{13}b_1)=\gcd(k_1l_1,2^{\sigma_0}m_{02}m_{03}m_{12}m_{13}b_0)=1,\\
\gcd(k_0l_0k_1l_1,2)=1,\\
\end{cases}
\end{equation}
and
\begin{equation}
    \label{twoadicconditionsvanishingmainterm14unwrapped}
        k_0l_0\equiv -q_0\bmod{8},\;
        k_1l_1\equiv q_1\bmod{8},\;
        k_1l_1\equiv q_2\bmod{8},\;
        k_0l_0\equiv -q_3\bmod{8}.
\end{equation}
Notice that since $(-1)^{\frac{(k_0k_1-1)}{2}} = \left(\frac{-1}{k_0k_1}\right)$
we may write
\[
(-1)^{\frac{(k_0k_1-1)}{2}}\left(\frac{2^{\sigma_2+\sigma_3+v_2(m_{02}m_{03}m_{12}m_{13})}}{k_0k_1}\right) =\left(\frac{-2^{\sigma_2+\sigma_3+v_2(m_{02}m_{03}m_{12}m_{13})}}{k_0k_1}\right).
\]
Let us now consider \eqref{twoadicconditionsvanishingmainterm14unwrapped}. We first note that these conditions require that any $\b{q}\in\mathcal{A}(\b{m},\boldsymbol{\sigma})$ must satisfy
\begin{equation}\label{twoadicequalitypairing14}
q_0\equiv q_3\bmod{8}\;\text{and}\;q_1\equiv q_2\bmod{8}.
\end{equation}
We now make use of the identity
\begin{equation}\label{orthogonalityidentity}
\mathds{1}(a\equiv q\bmod{8}) = \frac{1}{4}\sum_{\substack{\chi' \; \text{char.}\\ \bmod{8}}} \chi'(a)\overline{\chi'}(q)
\end{equation}
to break to congruence conditions in \eqref{twoadicconditionsvanishingmainterm14unwrapped} using \eqref{twoadicequalitypairing14}. Putting this into $V'_{1,4}(B)$ gives:
\begin{align*}
    V'_{1,4}(B) &= \frac{1}{16}\mathop{\sum\sum}_{\substack{\chi,\chi'\; \text{char.}\bmod{8}}} \overline{\chi}(-q_0)\overline{\chi'}(q_1)V'_{1,4}(\chi,\chi',B) \ll \mathop{\sum\sum}_{\substack{\chi,\chi'\; \text{char.}\bmod{8}}} \left\lvert V'_{1,4}(\chi,\chi',B)\right\rvert.
\end{align*}
where
\[
V'_{1,4}(\chi,\chi',B) = \mathop{\sum\sum\sum\sum}_{\substack{k_0,l_0,k_1,l_1\in\N^4\\ \eqref{nonsquareconditions3},\;\eqref{heightconditionsvanishingmainterm14},\;\eqref{vanishingmaintermgcdconditionsandsquarefree14}}} \frac{\mu^2(2k_0l_0k_1l_1)\chi(k_0l_0)\chi'(k_1l_1)}{\tau\left(k_0l_0k_1l_1\right)}\left(\frac{-2^{\sigma_2+\sigma_3+v_2(m_{02}m_{03}m_{12}m_{13})}}{k_0k_1}\right).
\]
The next lemma tells us that, in all of the cases we consider, we are always summing over a non-principal character.
\begin{lemma}\label{vanishingmainterm14nonprincipalevencharacter}
    Fix $\b{m}$ and $\boldsymbol{\sigma}$ satisfying the conditions \eqref{mconditions}, $(m_{02}m_{03}m_{12}m_{13})_{\textrm{odd}}=1$ and \eqref{sigmaconditions}. Then for any character $\chi$ modulo $8$ at least one of the characters 
    \[
    \left(\frac{-2^{\sigma_2+\sigma_3+v_2(m_{02}m_{03}m_{12}m_{13})}}{\cdot}\right)\chi(\cdot) \;\text{and}\; \chi(\cdot)
    \]
    is not principal.
\end{lemma}

\begin{proof}
    For any choice of $\b{m}$ and $\boldsymbol{\sigma}$ satisfying the conditions, the character $\left(\frac{-2^{\sigma_2+\sigma_3+v_2(m_{02}m_{03}m_{12}m_{13})}}{\cdot}\right)$
    is a non-principal character modulo $8$, since it is either equal to $ \left(\frac{-1}{\cdot}\right)\;\text{or}\;\left(\frac{-2}{\cdot}\right).$
    It follows that, $\left(\frac{-2^{\sigma_2+\sigma_3+v_2(m_{02}m_{03}m_{12}m_{13})}}{\cdot}\right)\chi(\cdot)$
    is the principal character modulo $8$ if and only if $\chi(\cdot) = \left(\frac{-2^{\sigma_2+\sigma_3+v_2(m_{02}m_{03}m_{12}m_{13})}}{\cdot}\right)$
    in which case $\chi$ is not the principal character. 
\end{proof}

Therefore we should expect to see some cancellation in $V'_{1,4}(\chi,\chi',B)$ resulting from the oscillation in these non-principal characters.

\begin{lemma}\label{vanishingmainterm14checkpoint2}
    Fix some $\b{b}\in\N^4$ and fix $\b{m}\in\N^4$ and $\boldsymbol{\sigma}\in\{0,1\}^4$ satisfying the conditions \eqref{mconditions}, $(m_{02}m_{03}m_{12}m_{13})_{\textrm{odd}}=1$ and \eqref{sigmaconditions}. Then
    \[
    V'_{1,4}(B) \ll \frac{\tau(m_0m_1m_2m_3)^4\tau(b_0)\tau(b_1)\tau(b_2)\tau(b_3)B^2}{M_0M_1M_2M_3(\log B)}.
    \]
\end{lemma}

\begin{proof}
    We first consider $V'_{1,4}(\chi,\chi',B)$ for $\chi$, $\chi'$ some characters modulo $8$. The first thing we wish to do is remove the square-free condition. To do this we use Lemma \ref{removesquarefree2} (with $w_0=w_1=z_0$), as well as noting that the characters of even modulus take care of the coprimality to $2$. Then
    \begin{align*}
     V'_{1,4}(\chi,\chi',B) &\ll \sum_{s\leq z_0}\mathop{\sum\sum\sum\sum}_{\substack{\substack{k_0',l_0',k_1',l_1'\leq z_0 \\ p|k_0'l_0'k_1'l_1'\Rightarrow p|s\\ s^2|k_0'l_0'k_1'l_1', \eqref{01'gcdconditionsvaishingmainterm14}}}}\frac{\left\lvert V'_{1,4}(\chi,\chi',\b{kl}_{01}',B)\right\rvert}{\tau(k_0')\tau(l_0')\tau(k_1')\tau(l_1')}+ O\left(\frac{B^2}{M_0M_1M_2M_3(\log B)^{A/3}}\right)
    \end{align*}
    where,
    \begin{align*}
    V'_{1,4}(\chi,\chi',\b{kl}_{01}',B) = \mathop{\sum\sum\sum\sum}_{\substack{k_0'',l_0'',k_1'',l_1''\in\N^4\\ \eqref{nonsquareconditions3},\;\eqref{01''gcdconditionsvanishingmainterm14},\;\eqref{heightconditionsvanishingmainterm142}}} \frac{\chi(k_0''l_0'')\chi'(k_1''l_1'')}{\tau(k_0'')\tau(l_0'')\tau(k_1'')\tau(l_1'')}\left(\frac{-2^{\sigma_2+\sigma_3+v_2(m_{02}m_{03}m_{12}m_{13})}}{k_0''k_1''}\right),
    \end{align*}
    \begin{equation}\label{01'gcdconditionsvaishingmainterm14}
    \begin{cases}    \gcd(k_0'l_0',2^{\sigma_1}m_{02}m_{03}m_{12}m_{13}b_1)=\gcd(k_1'l_1',2^{\sigma_0}m_{02}m_{03}m_{12}m_{13}b_0)=1,\\
    \gcd(k_0'l_0'k_1'l_1',2)=1,\\
    \end{cases}
    \end{equation}
    \begin{equation}\label{01''gcdconditionsvanishingmainterm14}
    \begin{cases}    \gcd(k_0''l_0'',2^{\sigma_1}m_{02}m_{03}m_{12}m_{13}sb_1)=\gcd(k_1''l_1'',2^{\sigma_0}m_{02}m_{03}m_{12}m_{13}sb_0)=1,\\
    \gcd(k_0''l_0''k_1''l_1'',2)=1,\\
    \end{cases}
    \end{equation}
    and
    \begin{equation}\label{heightconditionsvanishingmainterm142}    \|2^{\sigma_0}k_0'k_0''l_0'l_0''m_{02}m_{03}b_0^2,2^{\sigma_1}k_1'k_1''l_1'l_1''m_{12}m_{13}b_1^2\|\cdot\|2^{\sigma_2}m_{02}m_{12}b_2^2,2^{\sigma_3}m_{03}m_{13}b_3^2\|\leq B.
    \end{equation}
    Now we use Lemma \ref{vanishingmainterm14nonprincipalevencharacter} on both $\chi$ and $\chi'$ to note that these sums satisfy the conditions of Lemma \ref{fixedconductorlemmaforvanishingmainterms}: 
    \begin{itemize}
        \item the $n_i$ in Lemma \ref{fixedconductorlemmaforvanishingmainterms} correspond to $k_0'',l_0'',k_1'',l_1''$;
        \item $\chi_0$ in Lemma \ref{fixedconductorlemmaforvanishingmainterms} is thus $\chi(\cdot)\left(\frac{-2^{\sigma_2+\sigma_3+v_2(m_{02}m_{03}m_{12}m_{13})}}{\cdot}\right)$; $\chi_1$ in Lemma \ref{fixedconductorlemmaforvanishingmainterms} is $\chi(\cdot)$; $\chi_2$ in Lemma \ref{fixedconductorlemmaforvanishingmainterms} is $\chi'(\cdot)\left(\frac{-2^{\sigma_2+\sigma_3+v_2(m_{02}m_{03}m_{12}m_{13})}}{\cdot}\right)$ and $\chi_3$ in Lemma \ref{fixedconductorlemmaforvanishingmainterms} is $\chi'(\cdot)$;
        \item the remaining notation is assigned similarly to the applications of Propositions \ref{symmetrictypeaverage1}, \ref{Asymmetrictypeaverage1}, \ref{symmetrictypeaverage2} and \ref{Asymmetrictypeaverage2}.
    \end{itemize}
    Upon using Lemma \ref{fixedconductorlemmaforvanishingmainterms} we obtain
    \[
    V'_{1,4}(\chi,\chi',\b{kl}_{01}',B) \ll \frac{\tau(s)^4\tau(m_0m_1m_2m_3)^4\tau(b_0)\tau(b_1)\tau(b_2)\tau(b_3)B^2}{k_0'l_0'k_1'l_1'M_0M_1M_2M_3(\log B)}.
    \]
    Similar to how we dealt with \eqref{breakingsquarefreeconstant}, it can be shown that
    \[
    \sum_{s\leq z_0}\mathop{\sum\sum\sum\sum}_{\substack{\substack{k_0',l_0',k_1',l_1'\leq z_0 \\ p|k_0'l_0'k_1'l_1'\Rightarrow p|s\\ s^2|k_0'l_0'k_1'l_1', \eqref{01'gcdconditionsvaishingmainterm14}}}}\frac{\tau(s)^4}{k_0'l_0'k_1'l_1'\tau(k_0')\tau(l_0')\tau(k_1')\tau(l_1')}\ll 1.
    \]
    Thus
    \[
    V'_{1,4}(\chi,\chi',B) \ll \frac{\tau(m_0m_1m_2m_3)^4\tau(b_0)\tau(b_1)\tau(b_2)\tau(b_3)B^2}{M_0M_1M_2M_3(\log B)} + \frac{B^2}{M_0M_1M_2M_3(\log B)^{A/3}}.
    \]
    Now, by summing over the finitely many characters modulo $8$:
    \begin{align*}
    V'_{1,4}(B)
    &\ll \mathop{\sum\sum}_{\substack{\chi,\chi'\; \text{char.}\bmod{8}}} \left\lvert V'_{1,4}(\chi,\chi',B)\right\rvert \ll \frac{\tau(m_0m_1m_2m_3)^4\tau(b_0)\tau(b_1)\tau(b_2)\tau(b_3)B^2}{M_0M_1M_2M_3(\log B)}
    \end{align*}
    as required.
\end{proof}

\begin{proposition}\label{vanishingmainterm14proposition}
    Fix some $\b{b}\in\N^4$. Then
    \[
    \mathcal{V}_{1,4}(B,\b{b}) \ll_{A} \frac{\tau(b_0)\tau(b_1)\tau(b_2)\tau(b_3)B^2\sqrt{\log\log B}}{b_0^2b_1^2b_2^2b_3^2(\log B)}.
    \]
\end{proposition}

\begin{proof}
Recall from \eqref{vanishingmainterm14checkpoint1} that $\mathcal{V}_{1,4}(B,\b{b})$ is equal to
\begin{equation*}
  \mathop{\sum}_{\substack{\b{m}\in\N^4,\eqref{mconditions}\\ (m_{02}m_{03}m_{12}m_{13})_{\textrm{odd}}=1}}\mathop{\sum}_{\substack{\boldsymbol{\sigma}\in\{0,1\}^4\\ \eqref{sigmaconditions}}}\mathop{\sum}_{\substack{\b{q}\in\mathcal{A}(\b{m},\boldsymbol{\sigma})}} \hspace{-7pt}V'_{1,4}(B) + O_{A}\left(\frac{\tau(b_0)\tau(b_1)\tau(b_2)\tau(b_3)B^2\sqrt{\log\log B}}{b_0^2b_1^2b_2^2b_3^2(\log B)}\right).
\end{equation*}
Using Lemma \ref{vanishingmainterm14checkpoint2} we therefore have:
\begin{align*}
    \mathcal{V}_{1,4}(B,\b{b}) \ll& \mathop{\sum}_{\substack{\b{m}\in\N^4,\eqref{mconditions}\\ (m_{02}m_{03}m_{12}m_{13})_{\textrm{odd}}=1}}\mathop{\sum}_{\substack{\boldsymbol{\sigma}\in\{0,1\}^4\\ \eqref{sigmaconditions}}}\mathop{\sum}_{\substack{\b{q}\in\mathcal{A}(\b{m},\boldsymbol{\sigma})}}\frac{\tau(m_0m_1m_2m_3)^4\tau(b_0)\tau(b_1)\tau(b_2)\tau(b_3)B^2}{M_0M_1M_2M_3(\log B)} \\ &+ O_{A}\left(\frac{\tau(b_0)\tau(b_1)\tau(b_2)\tau(b_3)B^2\sqrt{\log\log B}}{b_0^2b_1^2b_2^2b_3^2(\log B)}\right).
\end{align*}
For the front term we note that in each case there are only finitely many $\b{q}\in\mathcal{A}(\b{m},\boldsymbol{\sigma})$ and $\boldsymbol{\sigma}\in\{0,1\}^4$. Also, given the conditions \eqref{mconditions} and $(m_{02}m_{03}m_{12}m_{13})_{\text{odd}}=1$, we must have $(m_{02},m_{03},m_{12},m_{13})\in\{(1,1,1,1),(2,1,1,1),(1,2,1,1),(1,1,2,1),(1,1,1,2)\}$. Thus there are only finitely many choices here as well. Thus after summing the first expression, the second error term will clearly dominate, giving the result.
\end{proof}

\subsection{The Vanishing Main Term $\mathcal{V}_{1,5}(B,\b{b})$}\label{vanish15}
The argument in this subsection will differ to the previous one in that the condition \eqref{nonsquareconditions3} will play a key role. We begin once again by examining the inner sums $H_{1,5}(\b{1},\b{1},\b{K},\b{L},B)$, i.e \eqref{generalinnersum} with $\b{d}=\Tilde{\b{d}}=1$.
We split off the terms where $k_0=l_0=k_1=l_1=1$, however in this case we must preserve the condition \eqref{nonsquareconditions3}. Thus we write:
\[
H_{1,5}(\b{1},\b{1},\b{K},\b{L},B) = \mathds{1}(2^{\sigma_0+\sigma_1}m_{02}m_{03}m_{12}m_{13}\neq 1)V_{1,5}(\b{K},\b{L},B) + EV_{1,5}(\b{K},\b{L},B)
\]
where
\[
V_{1,5}(\b{K},\b{L},B) = \mathop{\sum\sum\sum\sum}_{\substack{k_2,l_2,k_3,l_3\in\N^4\\(k_i,l_i)\equiv(K_i,L_i)\bmod{8}\;\forall\;i\in\{2,3\}\\\eqref{gcdconditionsandsquarefree3},\;\eqref{heightconditionsvanishingmainterm15}}} \frac{\mu^2(2k_0l_0k_1l_1)}{\tau\left(k_0l_0k_1l_1\right)},
\]
\[
EV_{1,5}(\b{K},\b{L},B) = \mathop{\sum\sum\sum\sum}_{\substack{(\b{k},\b{l})\in\mathcal{H}_5\\(\b{k},\b{l})\equiv(\b{K},\b{L})\bmod{8}\\ k_0l_0k_1l_1\neq 1 \\\eqref{gcdconditionsandsquarefree3},\;\eqref{heightconditions3}}} \frac{\mu^2(2k_0l_0k_1l_1k_2l_2k_3l_3)}{\tau\left(k_0l_0k_1l_1k_2l_2k_3l_3\right)}\left(\frac{l_0l_1}{k_2k_3}\right)\left(\frac{l_2l_3}{k_0k_1}\right),
\]
and
\begin{equation}\label{heightconditionsvanishingmainterm15}
\|2^{\sigma_0}m_{02}m_{03}b_0^2,2^{\sigma_1}m_{12}m_{13}b_1^2\|\cdot\|2^{\sigma_2}k_2l_2m_{02}m_{12}b_2^2,2^{\sigma_3}k_3l_3m_{03}m_{13}b_3^2\|\leq B.
\end{equation}
Similar to the previous section, we may handle the $EV_{1,5}(\b{K},\b{L},B)$ in the same way as we handled $H_{1,5}(\b{d},\Tilde{\b{d}},\b{K},\b{L},B)$ in \S \ref{smallconductor45} this time by noting that, after breaking the $\mu^2$ function, the condition $k_0l_0k_1l_1\neq 1$ guarantees that the conditions of Propositions \ref{symmetrictypeaverage2} or \ref{Asymmetrictypeaverage2} are satisfied. Summing this error term over the $\b{m},\boldsymbol{\sigma},\b{q},\b{K}$ and $\b{L}$ we are thus left with:
\begin{align*}
\mathcal{V}_{1,5}(B,\b{b}) &= \mathop{\sum}_{\substack{\b{m}\in\N^4,\eqref{mconditions}\\ (m_{02}m_{03}m_{12}m_{13})_{\textrm{odd}}=1}}\mathop{\sum}_{\substack{\boldsymbol{\sigma}\in\{0,1\}^4\\ \eqref{sigmaconditionsforvanishingmainterm15}}}\mathop{\sum}_{\substack{\b{q}\in\mathcal{A}(\b{m},\boldsymbol{\sigma})}}\mathop{\sum\sum}_{\substack{\b{K},\b{L}\in(\Z/8\Z)^{*4}\\ \eqref{twoadicvanishingmaintermconditions5}}} \Theta_{1,1}(\b{1},\b{K},\boldsymbol{\sigma})V_{1,5}(\b{K},\b{L},B)\\ &+ O_{A}\left(\frac{\tau(b_0)\tau(b_1)\tau(b_2)\tau(b_3)B^2\sqrt{\log\log B}}{b_0^2b_1^2b_2^2b_3^2(\log B)}\right).
\end{align*}
Now we wish to re-integrate the even characters into the sum over $k_2,l_2,k_3,l_3$. By doing this we obtain:
\begin{equation}\label{vanishingmainterm15checkpoint1}
 \hspace{-7.5pt}\mathcal{V}_{1,5}(B,\b{b}) = \hspace{-10pt}\mathop{\sum}_{\substack{\b{m}\in\N^4,\eqref{mconditions}\\ (m_{02}m_{03}m_{12}m_{13})_{\textrm{odd}}=1}}\hspace{-3pt}\mathop{\sum}_{\substack{\boldsymbol{\sigma}\in\{0,1\}^4\\ \eqref{sigmaconditionsforvanishingmainterm15}}}\mathop{\sum}_{\substack{\b{q}\in\mathcal{A}(\b{m},\boldsymbol{\sigma})}}\hspace{-10pt} V'_{1,5}(B) + O_{A}\hspace{-3pt}\left(\frac{\tau(b_0)\tau(b_1)\tau(b_2)\tau(b_3)B^2}{b_0^2b_1^2b_2^2b_3^2(\log B)(\log\log B)^{-1/2}}\right)
\end{equation}
where
\[
V'_{1,5}(B) = \mathop{\sum\sum\sum\sum}_{\substack{k_2,l_2,k_3,l_3\in\N^4\\ \eqref{nonsquareconditions3},\;\eqref{heightconditionsvanishingmainterm15},\;\eqref{vanishingmaintermgcdconditionsandsquarefree15},\;\eqref{twoadicconditionsvanishingmainterm15unwrapped}}} \frac{\mu^2(2k_2l_2k_3l_3)}{\tau\left(k_2l_2k_3l_3\right)}\left(\frac{2^{\sigma_0+\sigma_1+v_2(m_{02}m_{03}m_{12}m_{13})}}{k_2k_3}\right),
\]
\begin{equation}\label{vanishingmaintermgcdconditionsandsquarefree15}
\begin{cases}
\gcd(k_2l_2,2^{\sigma_1}m_{02}m_{03}m_{12}m_{13}b_3)=\gcd(k_3l_3,2^{\sigma_0}m_{02}m_{03}m_{12}m_{13}b_2)=1,\\
\gcd(k_2l_2k_3l_3,2)=1,\\
\end{cases}
\end{equation}
\begin{equation}\label{sigmaconditionsforvanishingmainterm15}
     \begin{cases}
     \sigma_0+\sigma_1+\sigma_2+\sigma_3\leq 1,\;\;\mathrm{gcd}(2^{\sigma_0+\sigma_1+\sigma_2+\sigma_3},m_{02}m_{03}m_{12}m_{13})=1\\
     \mathrm{gcd}(2^{\sigma_0},b_1)=\mathrm{gcd}(2^{\sigma_1},b_0)=\mathrm{gcd}(2^{\sigma_2},b_3)=\mathrm{gcd}(2^{\sigma_3},b_2)=1,\\
     2^{\sigma_0+\sigma_1}m_{02}m_{03}m_{12}m_{13} \neq 1,
 \end{cases}   
\end{equation}
and
\begin{equation}
\label{twoadicconditionsvanishingmainterm15unwrapped}
        k_2l_2\equiv -q_0\bmod{8},\;
        k_3l_3\equiv q_1\bmod{8},\;
        k_3l_3\equiv q_2\bmod{8},\;
        k_2l_2\equiv -q_3\bmod{8}.
\end{equation}

\begin{remark}
    An important comparison to the previous case is the absence of a factor coming from $(-1)^{f_1(\b{d},\b{k})}$. This is because, when $d_{02}d_{03}d_{12}d_{13}k_0k_1=1$, as is the case here, $f_1(\b{d},\b{k})=0$ regardless of the choices of $k_2$ and $k_3$. In the previous subsection we relied on this factor to guarantee the non-principality of the characters modulo $8$, see Lemma \ref{vanishingmainterm14nonprincipalevencharacter}; in this subsection we will instead rely on the condition $2^{\sigma_0+\sigma_1}m_{02}m_{03}m_{12}m_{13} \neq 1$ from \eqref{sigmaconditionsforvanishingmainterm15} to play this role, see Lemma \ref{vanishingmainterm15nonprincipalevencharacter} below.
\end{remark}

Given the condition \eqref{twoadicconditionsvanishingmainterm15unwrapped}, any $\b{q}\in\mathcal{A}(\b{m},\boldsymbol{\sigma})$ we consider must satisfy
\begin{equation}\label{twoadicequalitypairing15}
q_0\equiv q_3\bmod{8}\;\text{and}\;q_1\equiv q_2 \bmod{8}.
\end{equation}
Once again we make use of the identity \eqref{orthogonalityidentity}, this time to break to congruence conditions \eqref{twoadicconditionsvanishingmainterm15unwrapped} using \eqref{twoadicequalitypairing15}. Putting this into $V'_{1,5}(B)$ gives:
\begin{align*}
    V'_{1,5}(B) &= \frac{1}{16}\mathop{\sum\sum}_{\substack{\chi,\chi'\; \text{char.}\bmod{8}}} \overline{\chi}(-q_0)\overline{\chi'}(q_1)V'_{1,5}(\chi,\chi',B)\ll \mathop{\sum\sum}_{\substack{\chi,\chi'\; \text{char.}\bmod{8}}} \left\lvert V'_{1,5}(\chi,\chi',B)\right\rvert.
\end{align*}
where
\[
V'_{1,5}(\chi,\chi',B) = \mathop{\sum\sum\sum\sum}_{\substack{k_2,l_2,k_3,l_3\in\N^4\\ \eqref{nonsquareconditions3},\;\eqref{heightconditionsvanishingmainterm15},\;\eqref{vanishingmaintermgcdconditionsandsquarefree15}}} \frac{\mu^2(2k_2l_2k_3l_3)\chi(k_2l_2)\chi'(k_3l_3)}{\tau\left(k_2l_2k_3l_3\right)}\left(\frac{2^{\sigma_0+\sigma_1+v_2(m_{02}m_{03}m_{12}m_{13})}}{k_2k_3}\right).
\]
The next lemma is analogous to Lemma \ref{vanishingmainterm14nonprincipalevencharacter}.

\begin{lemma}\label{vanishingmainterm15nonprincipalevencharacter}
    Fix $\b{m}$ and $\boldsymbol{\sigma}$ satisfying the conditions \eqref{mconditions}, $(m_{02}m_{03}m_{12}m_{13})_{\textrm{odd}}=1$ and \eqref{sigmaconditionsforvanishingmainterm15}. Then for any character $\chi$ modulo $8$ at least one of the characters 
    \[
    \left(\frac{2^{\sigma_0+\sigma_1+v_2(m_{02}m_{03}m_{12}m_{13})}}{\cdot}\right)\chi(\cdot) \;\text{and}\; \chi(\cdot)
    \]
    is not principal.
\end{lemma}

\begin{proof}
    For any $\b{m}$ and $\boldsymbol{\sigma}$ satisfying \eqref{sigmaconditionsforvanishingmainterm15} we have $\left(\frac{2^{\sigma_0+\sigma_1+v_2(m_{02}m_{03}m_{12}m_{13})}}{\cdot}\right) = \left(\frac{2}{\cdot}\right)$.
    It follows that, in order for $\left(\frac{2^{\sigma_0+\sigma_1+v_2(m_{02}m_{03}m_{12}m_{13})}}{\cdot}\right)\chi(\cdot)$
    to be non-principal, we must have
    $\chi(\cdot) = \left(\frac{2}{\cdot}\right)$,
    in which case $\chi$ is non-principal.    
\end{proof}

From here we may now follow directly the argument of the previous subsection, applying Proposition \ref{fixedconductorlemmaforvanishingmainterms} to the $V'_{1,5}(\chi,\chi',B)$ and summing over the remaining variables to obtain the following:

\begin{proposition}\label{vanishingmainterm15proposition}
    Fix some $\b{b}\in\N^4$. Then
    \[
    \mathcal{V}_{1,5}(B,\b{b}) \ll_{A} \frac{\tau(b_0)\tau(b_1)\tau(b_2)\tau(b_3)B^2\sqrt{\log\log B}}{b_0^2b_1^2b_2^2b_3^2(\log B)}.
    \]
\end{proposition}

\subsection{The Vanishing Main Term $\mathcal{V}_{2,4}(B,\b{b})$}\label{vanish24}
Following the same procedure as in the previous subsections we may write
\begin{align*}
\mathcal{V}_{2,4}(B,\b{b}) &= \mathop{\sum}_{\substack{\b{m}\in\N^4,\eqref{mconditions}\\ (m_{02}m_{03}m_{12}m_{13})_{\textrm{odd}}=1}}\mathop{\sum}_{\substack{\boldsymbol{\sigma}\in\{0,1\}^4\\ \eqref{sigmaconditions}}}\mathop{\sum}_{\substack{\b{q}\in\mathcal{A}(\b{m},\boldsymbol{\sigma})}}\mathop{\sum\sum}_{\substack{\b{K},\b{L}\in(\Z/8\Z)^{*4}\\ \eqref{twoadicvanishingmaintermconditions4}}} \Theta_{2,1}(\b{1},\b{K},\boldsymbol{\sigma})V_{2,4}(\b{K},\b{L},B)\\ &+ O_{A}\left(\frac{\tau(b_0)\tau(b_1)\tau(b_2)\tau(b_3)B^2\sqrt{\log\log B}}{b_0^2b_1^2b_2^2b_3^2(\log B)}\right)
\end{align*}
Returning the even characters into the sum over $k_0,l_0,k_1,l_1$ we obtain:
\begin{equation}\label{vanishingmainterm24checkpoint1}
 \hspace{-7.5pt}\mathcal{V}_{2,4}(B,\b{b}) = \hspace{-10pt}\mathop{\sum}_{\substack{\b{m}\in\N^4,\eqref{mconditions}\\ (m_{02}m_{03}m_{12}m_{13})_{\textrm{odd}}=1}}\hspace{-3pt}\mathop{\sum}_{\substack{\boldsymbol{\sigma}\in\{0,1\}^4\\ \eqref{sigmaconditionsforvanishingmainterm24}}}\mathop{\sum}_{\substack{\b{q}\in\mathcal{A}(\b{m},\boldsymbol{\sigma})}}\hspace{-10pt} V'_{2,4}(B) + O_{A}\hspace{-3pt}\left(\frac{\tau(b_0)\tau(b_1)\tau(b_2)\tau(b_3)B^2}{b_0^2b_1^2b_2^2b_3^2(\log B)(\log\log B)^{-1/2}}\right)
\end{equation}
where
\[
V'_{2,4}(B) = \mathop{\sum\sum\sum\sum}_{\substack{k_0,l_0,k_1,l_1\in\N^4\\ \eqref{nonsquareconditions3},\;\eqref{heightconditionsvanishingmainterm24},\;\eqref{vanishingmaintermgcdconditionsandsquarefree24},\;\eqref{twoadicconditionsvanishingmainterm24unwrapped}}} \frac{\mu^2(2k_0l_0k_1l_1)}{\tau\left(k_0l_0k_1l_1\right)}\left(\frac{2^{\sigma_2+\sigma_3+v_2(m_{02}m_{03}m_{12}m_{13})}}{k_0k_1}\right),
\]
\begin{equation}\label{heightconditionsvanishingmainterm24}
\|2^{\sigma_0}k_0l_0m_{02}m_{03}b_0^2,2^{\sigma_1}k_1l_1m_{12}m_{13}b_1^2\|\cdot\|2^{\sigma_2}m_{02}m_{12}b_2^2,2^{\sigma_3}m_{03}m_{13}b_3^2\|\leq B,
\end{equation}
\begin{equation}\label{sigmaconditionsforvanishingmainterm24}
    \begin{cases}
     \sigma_0+\sigma_1+\sigma_2+\sigma_3\leq 1,\;\;\mathrm{gcd}(2^{\sigma_0+\sigma_1+\sigma_2+\sigma_3},m_{02}m_{03}m_{12}m_{13})=1\\
     \mathrm{gcd}(2^{\sigma_0},b_1)=\mathrm{gcd}(2^{\sigma_1},b_0)=\mathrm{gcd}(2^{\sigma_2},b_3)=\mathrm{gcd}(2^{\sigma_3},b_2)=1,\\
     2^{\sigma_2+\sigma_3}m_{02}m_{03}m_{12}m_{13} \neq 1,
 \end{cases}
\end{equation}
\begin{equation}\label{vanishingmaintermgcdconditionsandsquarefree24}
\begin{cases}
\gcd(k_0l_0,2^{\sigma_1}m_{02}m_{03}m_{12}m_{13}b_1)=\gcd(k_1l_1,2^{\sigma_0}m_{02}m_{03}m_{12}m_{13}b_0)=1,\\
\gcd(k_0l_0k_1l_1,2)=1,\\
\end{cases}
\end{equation}
and
\begin{equation}
    \label{twoadicconditionsvanishingmainterm24unwrapped}
        k_0l_0\equiv -q_0\bmod{8},\;
        k_1l_1\equiv q_1\bmod{8},\;
        k_1l_1\equiv -q_2\bmod{8},\;
        k_0l_0\equiv q_3\bmod{8}.
\end{equation}

Again, we do not have any factor coming from $(-1)^{f_2(\b{d},\b{k})}$. This is because
\[
f_2(\b{d},\b{k}) = k_0k_1-1
\]
when $d_{02}d_{03}d_{12}d_{13}k_2k_3 = 1$ as is the case here. Since $k_0k_1$ is always odd in our sums, $k_0k_1-1$ is always even, and so this factor is just $1$. Considering \eqref{twoadicconditionsvanishingmainterm24unwrapped}, we note that any $\b{q}\in\mathcal{A}(\b{m},\boldsymbol{\sigma})$ we consider must satisfy
\begin{equation}\label{twoadicequalitypairing24}
q_0\equiv -q_3\bmod{8}\;\text{and}\;q_1\equiv -q_2\bmod{8}.
\end{equation}
Using this and \eqref{orthogonalityidentity} to break to congruence conditions in \eqref{twoadicconditionsvanishingmainterm24unwrapped} and putting this into $V'_{2,4}(B)$ gives:
\begin{align*}
    V'_{2,4}(B) &= \frac{1}{16}\mathop{\sum\sum}_{\substack{\chi,\chi'\; \text{char.}\bmod{8}}} \overline{\chi}(-q_0)\overline{\chi'}(q_1)V'_{2,4}(\chi,\chi',B)\ll \mathop{\sum\sum}_{\substack{\chi,\chi'\; \text{char.}\bmod{8}}} \left\lvert V'_{2,4}(\chi,\chi',B)\right\rvert.
\end{align*}
where
\[
V'_{2,4}(\chi,\chi',B) = \mathop{\sum\sum\sum\sum}_{\substack{k_0,l_0,k_1,l_1\in\N^4\\ \eqref{nonsquareconditions3},\;\eqref{heightconditionsvanishingmainterm24},\;\eqref{vanishingmaintermgcdconditionsandsquarefree24}}} \frac{\mu^2(2k_0l_0k_1l_1)\chi(k_0l_0)\chi'(k_1l_1)}{\tau\left(k_0l_0k_1l_1\right)}\left(\frac{2^{\sigma_2+\sigma_3+v_2(m_{02}m_{03}m_{12}m_{13})}}{k_0k_1}\right).
\]

\begin{lemma}\label{vanishingmainterm24nonprincipalevencharacter}
    Fix $\b{m}$ and $\boldsymbol{\sigma}$ satisfying the conditions \eqref{mconditions}, $(m_{02}m_{03}m_{12}m_{13})_{\textrm{odd}}=1$ and \eqref{sigmaconditionsforvanishingmainterm24}. Then for any character $\chi$ modulo $8$ at least one of the characters 
    \[
    \left(\frac{2^{\sigma_2+\sigma_3+v_2(m_{02}m_{03}m_{12}m_{13})}}{\cdot}\right)\chi(\cdot) \;\text{and}\; \chi(\cdot)
    \]
    is not principal.
\end{lemma}

\begin{proof}
    For any $\b{m}$ and $\boldsymbol{\sigma}$ satisfying \eqref{sigmaconditionsforvanishingmainterm24} we have $\left(\frac{2^{\sigma_2+\sigma_3+v_3(m_{02}m_{03}m_{12}m_{13})}}{\cdot}\right) = \left(\frac{2}{\cdot}\right)$.
    It follows that, in order for $\left(\frac{2^{\sigma_2+\sigma_3+v_2(m_{02}m_{03}m_{12}m_{13})}}{\cdot}\right)\chi(\cdot)$
    to be non-principal, we must have $\chi(\cdot) = \left(\frac{2}{\cdot}\right)$,
    in which case $\chi$ is non-principal.    
\end{proof}

Now we may repeat the argument of \S \ref{vanish14}, applying Proposition \ref{fixedconductorlemmaforvanishingmainterms} and summing over the remaining variables to obtain:

\begin{proposition}\label{vanishingmainterm24proposition}
    Fix some $\b{b}\in\N^4$. Then
    \[
    \mathcal{V}_{2,4}(B,\b{b}) \ll_{A} \frac{\tau(b_0)\tau(b_1)\tau(b_2)\tau(b_3)B^2\sqrt{\log\log B}}{b_0^2b_1^2b_2^2b_3^2(\log B)}.
    \]
\end{proposition}

\subsection{The Vanishing Main Term $\mathcal{V}_{2,5}(B,\b{b})$}\label{vanish25}
We are left only to bound $\mathcal{V}_{2,5}(B,\b{b})$. We follow the same procedure as in the previous subsections to obtain
\begin{align*}
\mathcal{V}_{2,5}(B,\b{b}) &= \mathop{\sum}_{\substack{\b{m}\in\N^4,\eqref{mconditions}\\ (m_{02}m_{03}m_{12}m_{13})_{\textrm{odd}}=1}}\mathop{\sum}_{\substack{\boldsymbol{\sigma}\in\{0,1\}^4\\ \eqref{sigmaconditions}}}\mathop{\sum}_{\substack{\b{q}\in\mathcal{A}(\b{m},\boldsymbol{\sigma})}}\mathop{\sum\sum}_{\substack{\b{K},\b{L}\in(\Z/8\Z)^{*4}\\ \eqref{twoadicvanishingmaintermconditions5}}} \Theta_{2,1}(\b{1},\b{K},\boldsymbol{\sigma})V_{2,5}(\b{K},\b{L},B)\\ &+ O_{A}\left(\frac{\tau(b_0)\tau(b_1)\tau(b_2)\tau(b_3)B^2\sqrt{\log\log B}}{b_0^2b_1^2b_2^2b_3^2(\log B)}\right).
\end{align*}
Returning the even characters into the sum over $k_2,l_2,k_3,l_3$ we obtain:
\begin{equation}\label{vanishingmainterm25checkpoint1}
 \hspace{-7.5pt}\mathcal{V}_{2,5}(B,\b{b}) = \hspace{-10pt}\mathop{\sum}_{\substack{\b{m}\in\N^4,\eqref{mconditions}\\ (m_{02}m_{03}m_{12}m_{13})_{\textrm{odd}}=1}}\hspace{-3pt}\mathop{\sum}_{\substack{\boldsymbol{\sigma}\in\{0,1\}^4\\ \eqref{sigmaconditionsforvanishingmainterm25}}}\mathop{\sum}_{\substack{\b{q}\in\mathcal{A}(\b{m},\boldsymbol{\sigma})}}\hspace{-10pt} V'_{2,5}(B) + O_{A}\hspace{-3pt}\left(\frac{\tau(b_0)\tau(b_1)\tau(b_2)\tau(b_3)B^2}{b_0^2b_1^2b_2^2b_3^2(\log B)(\log\log B)^{-1/2}}\right)
\end{equation}
where
\[
V'_{2,5}(B) = \mathop{\sum\sum\sum\sum}_{\substack{k_2,l_2,k_3,l_3\in\N^4\\ \eqref{nonsquareconditions3},\;\eqref{heightconditionsvanishingmainterm25},\;\eqref{vanishingmaintermgcdconditionsandsquarefree25},\;\eqref{twoadicconditionsvanishingmainterm25unwrapped}}} \frac{\mu^2(2k_2l_2k_3l_3)}{\tau\left(k_2l_2k_3l_3\right)}(-1)^{\frac{k_2k_3-1}{2}}\left(\frac{2^{\sigma_0+\sigma_1+v_2(m_{02}m_{03}m_{12}m_{13})}}{k_2k_3}\right),
\]
\begin{equation}\label{heightconditionsvanishingmainterm25}
\|2^{\sigma_0}m_{02}m_{03}b_0^2,2^{\sigma_1}m_{12}m_{13}b_1^2\|\cdot\|2^{\sigma_2}k_2l_2m_{02}m_{12}b_2^2,2^{\sigma_3}k_3l_3m_{03}m_{13}b_3^2\|\leq B,
\end{equation}
\begin{equation}\label{sigmaconditionsforvanishingmainterm25}
    \begin{cases}
     \sigma_0+\sigma_1+\sigma_2+\sigma_3\leq 1,\;\;\mathrm{gcd}(2^{\sigma_0+\sigma_1+\sigma_2+\sigma_3},m_{02}m_{03}m_{12}m_{13})=1\\
     \mathrm{gcd}(2^{\sigma_0},b_1)=\mathrm{gcd}(2^{\sigma_1},b_0)=\mathrm{gcd}(2^{\sigma_2},b_3)=\mathrm{gcd}(2^{\sigma_3},b_2)=1,\\
     2^{\sigma_0+\sigma_1}m_{02}m_{03}m_{12}m_{13} \neq 1,
 \end{cases}
\end{equation}
\begin{equation}\label{vanishingmaintermgcdconditionsandsquarefree25}
\begin{cases}
\gcd(k_2l_2,2^{\sigma_3}m_{02}m_{03}m_{12}m_{13}b_3)=\gcd(k_3l_3,2^{\sigma_2}m_{02}m_{03}m_{12}m_{13}b_2)=1,\\
\gcd(k_2l_2k_3l_3,2)=1,\\
\end{cases}
\end{equation}
and
\begin{equation}
    \label{twoadicconditionsvanishingmainterm25unwrapped}
        k_2l_2\equiv -q_0\bmod{8},\;
        k_3l_3\equiv q_1\bmod{8},\;
        k_3l_3\equiv -q_2\bmod{8},\;
        k_2l_2\equiv q_3\bmod{8}.
\end{equation}

We note here that
\[
(-1)^{\frac{k_2k_3-1}{2}}\left(\frac{2^{\sigma_0+\sigma_1+v_2(m_{02}m_{03}m_{12}m_{13})}}{k_2k_3}\right) = \left(\frac{-2^{\sigma_0+\sigma_1+v_2(m_{02}m_{03}m_{12}m_{13})}}{k_2k_3}\right).
\]
Considering \eqref{twoadicconditionsvanishingmainterm25unwrapped}, we again note that any $\b{q}\in\mathcal{A}(\b{m},\boldsymbol{\sigma})$ we consider must satisfy
\begin{equation}\label{twoadicequalitypairing25}
q_0\equiv -q_3\bmod{8}\;\text{and}\;q_1\equiv -q_2\bmod{8},
\end{equation}
and apply \eqref{orthogonalityidentity}
to break to congruence conditions in \eqref{twoadicconditionsvanishingmainterm25unwrapped} using \eqref{twoadicequalitypairing25}. Putting this into $V'_{1,5}(B)$ gives:
\begin{align*}
    V'_{2,5}(B) &= \frac{1}{16}\mathop{\sum\sum}_{\substack{\chi,\chi'\; \text{char.}\bmod{8}}} \overline{\chi}(-q_0)\overline{\chi'}(q_1)V'_{2,5}(\chi,\chi',B) \ll \mathop{\sum\sum}_{\substack{\chi,\chi'\; \text{char.}\bmod{8}}} \left\lvert V'_{2,5}(\chi,\chi',B)\right\rvert,
\end{align*}
where
\[
V'_{2,5}(\chi,\chi',B) = \mathop{\sum\sum\sum\sum}_{\substack{k_2,l_2,k_3,l_3\in\N^4\\ \eqref{nonsquareconditions3},\;\eqref{heightconditionsvanishingmainterm25},\;\eqref{vanishingmaintermgcdconditionsandsquarefree25}}} \frac{\mu^2(2k_2l_2k_3l_3)\chi(k_2l_2)\chi'(k_3l_3)}{\tau\left(k_2l_2k_3l_3\right)}\left(\frac{-2^{\sigma_0+\sigma_1+v_2(m_{02}m_{03}m_{12}m_{13})}}{k_2k_3}\right).
\]

\begin{lemma}\label{vanishingmainterm25nonprincipalevencharacter}
    Fix $\b{m}$ and $\boldsymbol{\sigma}$ satisfying the conditions \eqref{mconditions}, $(m_{02}m_{03}m_{12}m_{13})_{\textrm{odd}}=1$ and \eqref{sigmaconditionsforvanishingmainterm25}. Then for any character $\chi$ modulo $8$ at least one of the characters 
    \[
    \left(\frac{-2^{\sigma_0+\sigma_1+v_2(m_{02}m_{03}m_{12}m_{13})}}{\cdot}\right)\chi(\cdot) \;\text{and}\; \chi(\cdot)
    \]
    is not principal.
\end{lemma}

\begin{proof}
    For any $\b{m}$ and $\boldsymbol{\sigma}$ satisfying \eqref{sigmaconditionsforvanishingmainterm25} we have $\left(\frac{-2^{\sigma_2+\sigma_3+v_3(m_{02}m_{03}m_{12}m_{13})}}{\cdot}\right) = \left(\frac{-2}{\cdot}\right)$.
    It follows that, in order for $\left(\frac{-2^{\sigma_2+\sigma_3+v_2(m_{02}m_{03}m_{12}m_{13})}}{\cdot}\right)\chi(\cdot)$
    to be non-principal, we must have $\chi(\cdot) = \left(\frac{-2}{\cdot}\right)$,
    in which case $\chi$ is non-principal.    
\end{proof}

Now we may repeat the argument of \S \ref{vanish14}, applying Proposition \ref{fixedconductorlemmaforvanishingmainterms} and summing over the remaining variables to obtain the following.

\begin{proposition}\label{vanishingmainterm25proposition}
    Fix some $\b{b}\in\N^4$. Then
    \[
    \mathcal{V}_{2,5}(B,\b{b}) \ll_{A} \frac{\tau(b_0)\tau(b_1)\tau(b_2)\tau(b_3)B^2\sqrt{\log\log B}}{b_0^2b_1^2b_2^2b_3^2(\log B)}.
    \]
\end{proposition}

\subsection{Bounding of $N_{r,4}(B)$ and $N_{r,5}(B)$} We may now prove the following:

\begin{proposition}\label{Vanish_Main_Term_Contribution}
    Let $B\geq 3$. Then
    \[
    N_{r,4}(B),N_{r,5}(B) \ll_{A} \frac{B^2\sqrt{\log\log B}}{(\log B)}.
    \]
\end{proposition}

\begin{proof}
    By Propositions \ref{Errorterms4},\ref{Errorterms5},\ref{vanishingmainterm14proposition},\ref{vanishingmainterm15proposition},\ref{vanishingmainterm24proposition} and \ref{vanishingmainterm25proposition} we have
    \[
    N_{r,i}(B) \ll_{A} \sum_{\b{b}\in\N^4} \frac{\tau(b_0)\tau(b_1)\tau(b_2)\tau(b_3)B^2\sqrt{\log\log B}}{b_0^2b_1^2b_2^2b_3^2(\log B)} \ll_{A} \frac{B^2\sqrt{\log\log B}}{(\log B)}
    \]
    for $r=1,2$ and $i=4,5$.
\end{proof}

\section{Main Terms}\label{mainterm}
In this section we will finally isolate the true main terms from $\mathcal{M}_{r,i}(B,\b{b})$ where $r\in\{1,2\}$ and $i\in\{2,3\}$. This is achieved by trimming the remaining contributions from summing over oscillating characters in these regions. Recall that the inner sums of $\mathcal{M}_{r,2}(B,\b{b})$ and $\mathcal{M}_{r,3}(B,\b{b})$ are of the form \eqref{generalinnersum} for $i\in\{2,3\}$ and $\{\b{d},\Tilde{\b{d}}\}=\{\b{1},\b{m}_{\textrm{odd}}\}$. We split these inner sums into two. In $H_{r,2}(\b{1},\b{m}_{\textrm{odd}},\b{1},\b{L},B)$ we separate the parts where $k_0k_1k_2k_3=1$ and in $H_{r,3}(\b{m}_{\textrm{odd}},\b{1},\b{K},\b{1},B)$ we separate the parts where $l_0l_1l_2l_3=1$:
\[
H_{r,2}(\b{1},\b{m}_{\textrm{odd}},\b{1},\b{L},B) = M_{r,2}(\b{1},\b{m}_{\textrm{odd}},\b{1},\b{L},B) + EM_{r,2}(\b{1},\b{m}_{\textrm{odd}},\b{1},\b{L},B),
\]
and
\[
H_{r,3}(\b{m}_{\textrm{odd}},\b{1},\b{K},\b{1},B) = M_{r,3}(\b{m}_{\textrm{odd}},\b{1},\b{K},\b{1},B) + EM_{r,3}(\b{m}_{\textrm{odd}},\b{1},\b{K},\b{1},B),
\]
where
\begin{equation}\label{innermaintermr2}
M_{r,2}(\b{1},\b{m}_{\textrm{odd}},\b{1},\b{L},B) = \mathop{\sum\sum\sum\sum}_{\substack{\|l_0,l_1\|,\|l_2,l_3\|>z_1\\ \b{l}\equiv \b{L}\bmod{8}\\\eqref{maingcdconditionsandsquarefree2},\;\eqref{mainheightconditions2}}} \frac{\mu^2(2l_0l_1l_2l_3)}{\tau\left(l_0l_1l_2l_3\right)},
\end{equation}
\begin{equation}\label{innermaintermr3}
M_{r,3}(\b{m}_{\textrm{odd}},\b{1},\b{K},\b{1},B) = \mathop{\sum\sum\sum\sum}_{\substack{\|k_0,k_1\|,\|k_2,k_3\|>z_1\\\b{k}\equiv \b{K}\bmod{8}\\\eqref{maingcdconditionsandsquarefree3},\;\eqref{mainheightconditions3}}} \frac{\mu^2(2k_0k_1k_2k_3)}{\tau\left(k_0k_1k_2k_3\right)},
\end{equation}
\begin{equation}
EM_{r,2}(\b{1},\b{m}_{\textrm{odd}},\b{1},\b{L},B) = \mathop{\sum\sum\sum\sum}_{\substack{(\b{k},\b{l})\in\mathcal{H}_2\\(\b{k},\b{l})\equiv(\b{1},\b{L})\bmod{8}\\k_0k_1k_2k_3\neq 1\\ \eqref{gcdconditionsandsquarefree3},\;\eqref{heightconditions3}}} \frac{\mu^2(2k_0l_0k_1l_1k_2l_2k_3l_3)}{\tau\left(k_0l_0k_1l_1k_2l_2k_3l_3\right)}\Theta_2(\b{1},\b{m}_{\textrm{odd}},\b{k},\b{l}),
\end{equation}
\begin{equation}
EM_{r,3}(\b{m}_{\textrm{odd}},\b{1},\b{K},\b{1},B) = \mathop{\sum\sum\sum\sum}_{\substack{(\b{k},\b{l})\in\mathcal{H}_3\\(\b{k},\b{l})\equiv(\b{K},\b{1})\bmod{8}\\ l_0l_1l_2l_3\neq 1 \\ \eqref{gcdconditionsandsquarefree3},\;\eqref{heightconditions3}}} \frac{\mu^2(2k_0l_0k_1l_1k_2l_2k_3l_3)}{\tau\left(k_0l_0k_1l_1k_2l_2k_3l_3\right)}\Theta_2(\b{m}_{\textrm{odd}},\b{1},\b{k},\b{l}),
\end{equation}
\begin{equation}\label{maingcdconditionsandsquarefree2} \begin{cases}
\gcd(l_0,2^{\sigma_1}m_{02}m_{03}m_{12}m_{13}b_1)=\gcd(l_1,2^{\sigma_0}m_{02}m_{03}m_{12}m_{13}b_0)=1,\\
\gcd(l_2,2^{\sigma_3}m_{02}m_{03}m_{12}m_{13}b_3)=\gcd(l_3,2^{\sigma_2}m_{02}m_{03}m_{12}m_{13}b_2)=1,\\
\end{cases}
\end{equation}

\begin{equation}\label{mainheightconditions2}
\|2^{\sigma_0}l_0m_{02}m_{03}b_0^2,2^{\sigma_1}l_1m_{12}m_{13}b_1^2\|\cdot\|2^{\sigma_2}l_2m_{02}m_{12}b_2^2,2^{\sigma_3}l_3m_{03}m_{13}b_3^2\|\leq B,
\end{equation}

\begin{equation}\label{maingcdconditionsandsquarefree3} \begin{cases}
\gcd(k_0,2^{\sigma_1}m_{02}m_{03}m_{12}m_{13}b_1)=\gcd(k_1,2^{\sigma_0}m_{02}m_{03}m_{12}m_{13}b_0)=1,\\
\gcd(k_2,2^{\sigma_3}m_{02}m_{03}m_{12}m_{13}b_3)=\gcd(k_3,2^{\sigma_2}m_{02}m_{03}m_{12}m_{13}b_2)=1,\\
\end{cases}
\end{equation}

\begin{equation}\label{mainheightconditions3}
\|2^{\sigma_0}k_0m_{02}m_{03}b_0^2,2^{\sigma_1}k_1m_{12}m_{13}b_1^2\|\cdot\|2^{\sigma_2}k_2m_{02}m_{12}b_2^2,2^{\sigma_3}k_3m_{03}m_{13}b_3^2\|\leq B.
\end{equation}

\begin{remark}\label{ignorenonsquareconditions}
    Note that the we have dropped non-square condition \eqref{nonsquareconditions3}. This is because the lower bounds $\|l_0,l_1\|,\|l_2,l_3\|>z_1$ and $\|k_0,k_1\|,\|k_2,k_3\|>z_1$ from $\mathcal{H}_2$ and $\mathcal{H}_3$ automatically ensure that it is satisfied.
\end{remark}

\subsection{The Error Terms $EM_{r,2}(\b{1},\b{m}_{\textrm{odd}},\b{1},\b{L},B)$ and $EM_{r,3}(\b{m}_{\textrm{odd}},\b{1},\b{K},\b{1},B)$}
To deal with $EM_{r,2}(\b{1},\b{m}_{\textrm{odd}},\b{1},\b{L},B)$ we note that
\[
\Theta_2(\b{1},\b{m}_{\textrm{odd}},\b{k},\b{l}) = \left(\frac{(m_{02}m_{03}m_{12}m_{13})_{\textrm{odd}}}{k_0k_1k_2k_3}\right)\left(\frac{l_0l_1}{k_2k_3}\right)\left(\frac{l_2l_3}{k_0k_1}\right).
\]
From this we can see that the conditions $k_0k_1k_2k_3\neq 1$ and $\mu^2(k_0k_1k_2k_3)=1$ dictate that, by summing first over the $l_i$, we are summing over non-trivial characters. Similarly, for the error term $EM_{r,3}(\b{m}_{\textrm{odd}},\b{1},\b{K},\b{1},B)$ we note that
\[
\Theta_2(\b{m}_{\textrm{odd}},\b{1},\b{k},\b{l}) = \left(\frac{l_0l_1l_2l_3}{(m_{02}m_{03}m_{12}m_{13})_{\textrm{odd}}}\right)\left(\frac{l_0l_1}{k_2k_3}\right)\left(\frac{l_2l_3}{k_0k_1}\right),
\]
and so the same observation holds here with the $k_i$ and $l_i$ switched. Thus we may repeat the arguments used in \S \ref{smallconductor23} to obtain the following:

\begin{lemma}
    Fix some $\b{b}\in\N^4$, some $\b{m}\in\N^4$ satisfying \eqref{mconditions}, some $\boldsymbol{\sigma}\in\{0,1\}^4$ satisfying \eqref{sigmaconditions} and some $\b{q}\in\mathcal{A}(\b{m},\boldsymbol{\sigma})$. Then for $\b{L}\in(\Z/8\Z)^{*4}$ satisfying \eqref{twoadicmaintermconditions} we have
    \begin{align*}
    EM_{r,2}(\b{1},\b{m}_{\textrm{odd}},\b{1},\b{L},B) &\ll_{A} \frac{\tau(b_0)\tau(b_1)\tau(b_2)\tau(b_3)B^2}{M_0M_1M_2M_3(\log B)(\log \log B)^{66A}}.
    \end{align*}
    Similarly, for $\b{K}\in(\Z/8\Z)^{*4}$ satisfying \eqref{twoadicmaintermconditions} we have
    \begin{align*}
    EM_{r,3}(\b{m}_{\textrm{odd}},\b{1},\b{K},\b{1},B) &\ll_{A} \frac{\tau(b_0)\tau(b_1)\tau(b_2)\tau(b_3)B^2}{M_0M_1M_2M_3(\log B)(\log \log B)^{66A}}.
    \end{align*}
\end{lemma}
Then by summing over $H_{r,2}(\b{1},\b{m}_{\textrm{odd}},\b{1},\b{L},B)$ and $H_{r,3}(\b{m}_{\textrm{odd}},\b{1},\b{K},\b{1},B)$, using the same methods used to prove Propositions \ref{Errorterms2} and \ref{Errorterms3} to sum over these error parts above, we obtain:
\begin{align*}
\mathcal{M}_{r,2}(B,\b{b}) &= \mathop{\sum}_{\substack{\b{m}\in\N^4\\ m_{ij}\leq z_0\\ \eqref{mconditions}}} \mathop{\sum}_{\substack{\boldsymbol{\sigma}\in\{0,1\}^4\\ \eqref{sigmaconditions}}}\mathop{\sum}_{\substack{\b{q}\in\mathcal{A}(\b{m},\boldsymbol{\sigma})}}\mathop{\sum}_{\substack{\b{L}\in(\Z/8\Z)^{*4}\\ \eqref{twoadicmaintermconditions}}} \frac{M_{r,2}(\b{1},\b{m}_{\textrm{odd}},\b{1},\b{L},B)}{\tau\left((m_{02}m_{03}m_{12}m_{13})_{\mathrm{odd}}\right)} \\ &+ O_{A}\left(\frac{\tau(b_0)\tau(b_1)\tau(b_2)\tau(b_3)B^2}{b_0^2b_1^2b_2^2b_3^2(\log B)(\log\log B)^{66A}}\right),
\end{align*}
and
\begin{align*}
\mathcal{M}_{r,3}(B,\b{b}) &= \mathop{\sum}_{\substack{\b{m}\in\N^4\\ m_{ij}\leq z_0\\ \eqref{mconditions}}} \mathop{\sum}_{\substack{\boldsymbol{\sigma}\in\{0,1\}^4\\ \eqref{sigmaconditions}}}\mathop{\sum}_{\substack{\b{q}\in\mathcal{A}(\b{m},\boldsymbol{\sigma})}}\mathop{\sum}_{\substack{\b{K}\in(\Z/8\Z)^{*4}\\ \eqref{twoadicmaintermconditions}}} \frac{\Theta_{r,1}(\b{m}_{\text{odd}},\b{K},\boldsymbol{\sigma})M_{r,3}(\b{m}_{\textrm{odd}},\b{1},\b{K},\b{1},B)}{\tau\left((m_{02}m_{03}m_{12}m_{13})_{\mathrm{odd}}\right)} \\ &+ O_{A}\left(\frac{\tau(b_0)\tau(b_1)\tau(b_2)\tau(b_3)B^2}{b_0^2b_1^2b_2^2b_3^2(\log B)(\log\log B)^{66A}}\right).
\end{align*}

\subsection{The Inner Main Terms $M_{r,2}(\b{1},\b{m}_{\textrm{odd}},\b{1},\b{L},B)$ and $M_{r,3}(\b{m}_{\textrm{odd}},\b{1},\b{K},\b{1},B)$}
In this subsection we deal with the inner sums. Once this is done we will only be left to compute the constants. Define
\[
\overline{\mathfrak{S}}(\b{b},\b{m},v) = \frac{4f_0^4}{\phi(8)^4f_2^4}\prod_{i=0}^{3}\left(\prod_{\substack{p|v\\p\nmid m_{02}m_{12}m_{03}m_{13}b_i\\p\;\textrm{odd}}}f_p^{-1}\right)\prod_{i=0}^{3}\left(\prod_{\substack{p|m_{02}m_{12}m_{03}m_{13}b_i\\p\;\textrm{odd}}}f_p^{-1}\right).
\]
and
\[
\mathfrak{C}(\b{b},\b{m},\boldsymbol{\sigma}) = \sum_{v\in\N}\mu(v)\mathop{\sum\sum\sum\sum}_{{\substack{a_0,a_1,a_2,a_3\in\N \\ p|a_0a_1a_2a_3\Rightarrow p|v\\ v^2|a_0a_1a_2a_3\\ \eqref{gcdconditionsformaintermconstants}}}}\frac{\overline{\mathfrak{S}}(\b{b},\b{m},v)}{a_0a_1a_2a_3\tau(a_0)\tau(a_1)\tau(a_2)\tau(a_3)}
\]
where
\begin{equation}\label{gcdconditionsformaintermconstants}
\begin{cases}
\gcd(a_0,2^{\sigma_1}m_{02}m_{03}m_{12}m_{13}b_1)=\gcd(a_1,2^{\sigma_0}m_{02}m_{03}m_{12}m_{13}b_0)=1,\\
\gcd(a_2,2^{\sigma_3}m_{02}m_{03}m_{12}m_{13}b_3)=\gcd(a_3,2^{\sigma_2}m_{02}m_{03}m_{12}m_{13}b_2)=1,\\
\gcd(a_0a_1a_2a_3,2)=1.
\end{cases}
\end{equation}

\begin{remark}
We note here that $\overline{\mathfrak{S}}(\b{b},\b{m},v)$ is obtained by re-arranging $\mathfrak{S}_2(\b{r})$ from Proposition \ref{maintermproposition} with $\b{r}$ being equal to
\[
((m_{02}m_{03}m_{12}m_{13})_{\textrm{odd}}vb_1,(m_{02}m_{03}m_{12}m_{13})_{\textrm{odd}}vb_0,(m_{02}m_{03}m_{12}m_{13})_{\textrm{odd}}vb_3,(m_{02}m_{03}m_{12}m_{13})_{\textrm{odd}}vb_2)
\]
using the conditions \eqref{mconditions}, \eqref{sigmaconditions} and \eqref{gcdconditionsformaintermconstants}.
\end{remark}

\begin{lemma}\label{innermaintermlemma}
    Let $B\geq 3$. Fix $\b{b}\in\N^4$, fix $\b{m}\in\N^4$ and $\boldsymbol{\sigma}\in\{0,1\}^{4}$ satisfying \eqref{mconditions} and \eqref{sigmaconditions}. Then
    \begin{align*}
    M_{r,i}(\b{1},\b{m}_{\textrm{odd}},\b{1},\b{L},B) &= \frac{\mathfrak{C}(\b{b},\b{m},\boldsymbol{\sigma})B^2\log\log B}{M_0M_1M_2M_3\log B}\\ &+ O_{A}\left(\frac{\tau(m_{02}m_{03}m_{12}m_{13})^4\tau(b_0)\tau(b_1)\tau(b_2)\tau(b_3)B^2\sqrt{\log\log B}}{M_0M_1M_2M_3\log B}\right)
    \end{align*}
    for $i=2,3$.
\end{lemma}

\begin{proof}
    By switching $l_i$ with $k_i$ the proofs for $i=2$ and $i=3$ may be seen to be identical. We therefore only deal with $M_{r,2}(\b{1},\b{m}_{\textrm{odd}},\b{1},\b{L},B)$. Recall \eqref{innermaintermr2} and Remark \ref{ignorenonsquareconditions}. By applying \eqref{removesquarefree} with $w_0=w_1=z_0$ as we have done previously $M_{r,2}(\b{1},\b{m}_{\textrm{odd}},\b{1},\b{L},B)$ becomes equal to
    \begin{align}\label{Innermaintermcheckpoint1}
     \sum_{v\leq z_0}\mu(v)\mathop{\sum\sum\sum\sum}_{{\substack{l_0',l_1',l_2',l_3'\leq z_0 \\ p|l_0'l_1'l_2'l_3'\Rightarrow p|v\\ v^2|l_0'l_1'l_2'l_3', \eqref{gcdconditionsformaintermconstants}}}}\frac{M'_{r,2}(\b{1},\b{m}_{\textrm{odd}},\b{1},\b{L},\b{l}',B)}{\tau(l_0')\tau(l_1')\tau(l_2')\tau(l_3')} + O\left(\frac{B^2(\log B)}{M_0M_1M_2M_3(\log B)^{A/3}}\right)
    \end{align}
    where
    \[
    M'_{r,2}(\b{1},\b{m}_{\textrm{odd}},\b{1},\b{L},\b{l}',B) = \mathop{\sum\sum\sum\sum}_{\substack{\|l_0'l_0'',l_1'l_1''\|,\|l_2'l_2'',l_3'l_3''\|>z_1\\ l_i''\equiv L_i/l_i'\bmod{8}\\\eqref{maingcdconditionsandsquarefree2''},\;\eqref{mainheightconditions2''}}} \frac{1}{\tau(l_0'')\tau(l_1'')\tau(l_2'')\tau(l_3'')},
    \]
    \begin{equation}\label{maingcdconditionsandsquarefree2''}
        \begin{cases}        \gcd(l_0'',2^{\sigma_1}m_{02}m_{03}m_{12}m_{13}vb_1)=\gcd(l_1'',2^{\sigma_0}m_{02}m_{03}m_{12}m_{13}vb_0)=1,\\ \gcd(l_2'',2^{\sigma_3}m_{02}m_{03}m_{12}m_{13}vb_3)=\gcd(l_3'',2^{\sigma_2}m_{02}m_{03}m_{12}m_{13}vb_2)=1,
        \end{cases}
    \end{equation}
    \begin{equation}\label{mainheightconditions2''} \|2^{\sigma_0}l_0'l_0''m_{02}m_{03}b_0^2,2^{\sigma_1}l_1'l_1''m_{12}m_{13}b_1^2\|\cdot\|2^{\sigma_2}l_2'l_2''m_{02}m_{12}b_2^2,2^{\sigma_3}l_3'l_3''m_{03}m_{13}b_3^2\|\leq B.
    \end{equation}
    Now we apply Proposition \ref{maintermproposition} with the previous remark to these sums yielding:
    \begin{align*}
    M'_{r,2}(\b{1},\b{m}_{\textrm{odd}},\b{1},\b{L},\b{l}',B) &= \frac{\overline{\mathfrak{S}}(\b{b},\b{m},v)B^2\log\log B}{l_0'l_1'l_2'l_3'M_0M_1M_2M_3\log B} \\ &+ O_{A}\left(\frac{\tau(m_{02}m_{03}m_{12}m_{13})^4\tau(v)^4\tau(b_0)\tau(b_1)\tau(b_2)\tau(b_3)B^2\sqrt{\log\log B}}{l_0'l_1'l_2'l_3'M_0M_1M_2M_3\log B}\right).
    \end{align*}
    Now substitute this into \eqref{Innermaintermcheckpoint1}. Then
    \begin{align*}
    M_{r,2}(\b{1},\b{m}_{\textrm{odd}},\b{1},\b{L},B) =& \frac{B^2\log\log B}{M_0M_1M_2M_3\log B}\sum_{v\leq z_0}\mu(v)\mathop{\sum\sum\sum\sum}_{{\substack{l_0',l_1',l_2',l_3'\leq z_0 \\ p|l_0'l_1'l_2'l_3'\Rightarrow p|v\\ v^2|l_0'l_1'l_2'l_3'\\ \eqref{gcdconditionsformaintermconstants}}}}\frac{\overline{\mathfrak{S}}(\b{b},\b{m},v)}{l_0'l_1'l_2'l_3'\tau(l_0')\tau(l_1')\tau(l_2')\tau(l_3')} \\& +O_{A}\left(\frac{\tau(m_{02}m_{03}m_{12}m_{13})^4\tau(b_0)\tau(b_1)\tau(b_2)\tau(b_3)\mathcal{R}'B^2\sqrt{\log\log B}}{M_0M_1M_2M_3\log B}\right)
    \end{align*}
    
    To deal with the error term it is enough to show that 
    \[
    \mathcal{R}'=\sum_{v\leq z_0}\mathop{\sum\sum\sum\sum}_{{\substack{l_0',l_1',l_2',l_3'\leq z_0 \\ p|l_0'l_1'l_2'l_3'\Rightarrow p|v\\ v^2|l_0'l_1'l_2'l_3', \eqref{gcdconditionsformaintermconstants}}}}\frac{\tau(v)^4}{l_0'l_1'l_2'l_3'\tau(l_0')\tau(l_1')\tau(l_2')\tau(l_3')}\ll 1,
    \]
    which is done using an identical method used to deal with \eqref{breakingsquarefreeconstant}. Finally, we need to show that
    \[
    \sum_{v\leq z_0}\mu(v)\mathop{\sum\sum\sum\sum}_{{\substack{l_0',l_1',l_2',l_3'\leq z_0 \\ p|l_0'l_1'l_2'l_3'\Rightarrow p|v\\ v^2|l_0'l_1'l_2'l_3', \eqref{gcdconditionsformaintermconstants}}}}\frac{\overline{\mathfrak{S}}(\b{b},\b{m},v)}{l_0'l_1'l_2'l_3'\tau(l_0')\tau(l_1')\tau(l_2')\tau(l_3')} = \mathfrak{C}(\b{b},\b{m},\boldsymbol{\sigma}) + O\left(\frac{1}{z_0^{1/3}}\right).
    \]
    This may be seen by noting that the following sums are $O(z_0^{-1/3})$:
    \[
    \sum_{v> z_0}\mathop{\sum\sum\sum\sum}_{{\substack{l_0',l_1',l_2',l_3'\in\N \\ p|l_0'l_1'l_2'l_3'\Rightarrow p|v\\ v^2|l_0'l_1'l_2'l_3', \eqref{gcdconditionsformaintermconstants}}}}\frac{1}{l_0'l_1'l_2'l_3'\tau(l_0')\tau(l_1')\tau(l_2')\tau(l_3')},\sum_{v\leq z_0}\mathop{\sum\sum\sum\sum}_{{\substack{l_0',l_1',l_2',l_3'\in\N \\ p|l_0'l_1'l_2'l_3'\Rightarrow p|v\\ v^2|l_0'l_1'l_2'l_3'\\ l_i>z_0\;\text{for some $i$}}}}\frac{1}{l_0'l_1'l_2'l_3'\tau(l_0')\tau(l_1')\tau(l_2')\tau(l_3')}
    \]
    which may be shown by adhering to the bounding of the error terms \eqref{removesquarefree1errorterm1} and \eqref{removesquarefree1errorterm2} in Lemma \ref{removesquarefree}. The error term of the $M_{r,i}$ resulting from the error term $z_0^{-1/3}$ above will be absorbed into the error terms already present.
\end{proof}

\subsection{The Main Terms $\mathcal{M}_{r,2}(B,\b{b})$ and $\mathcal{M}_{r,3}(B,\b{b})$}
Define
\[
\Sigma_{r,2}(\b{m},\boldsymbol{\sigma}) = \mathop{\sum}_{\substack{\b{q}\in\mathcal{A}(\b{m},\boldsymbol{\sigma})}}\mathop{\sum}_{\substack{\b{L}\in(\Z/8\Z)^{*4}\\ \eqref{twoadicmaintermconditions}}} 1,
\]

\[
\Sigma_{r,3}(\b{m},\boldsymbol{\sigma}) = \mathop{\sum}_{\substack{\b{q}\in\mathcal{A}(\b{m},\boldsymbol{\sigma})}}\mathop{\sum}_{\substack{\b{K}\in(\Z/8\Z)^{*4}\\ \eqref{twoadicmaintermconditions}}} \Theta_{r,1}(\b{m}_{\text{odd}},\b{K},\boldsymbol{\sigma}),
\]

\[
\mathfrak{C}_{r,i}(\b{b}) = \mathop{\sum}_{\substack{\b{m}\in\N^4\\ \eqref{mconditions}}} \mathop{\sum}_{\substack{\boldsymbol{\sigma}\in\{0,1\}^4\\ \eqref{sigmaconditions}}} \frac{\mathfrak{C}(\b{b},\b{m},\boldsymbol{\sigma})\Sigma_{r,i}(\b{m},\boldsymbol{\sigma})}{2^{\sigma_0+\sigma_1+\sigma_2+\sigma_3}m_{02}^2m_{03}^2m_{12}^2m_{13}^2\tau((m_{02}m_{03}m_{12}m_{13})_{\textrm{odd}})}
\]
for $i=2,3$.

\begin{proposition}\label{Maintermswithb}
    Let $B\geq 3$ and fix some $\b{b}\in \N^4$. Then,
    \[
    \mathcal{M}_{r,i}(B,\b{b}) = \frac{\mathfrak{C}_{r,i}(\b{b})B^2\log\log B}{b_0^2b_1^2b_2^2b_3^2\log B} + O_{A}\left(\frac{\tau(b_0)\tau(b_1)\tau(b_2)\tau(b_3)B^2\sqrt{\log\log B}}{b_0^2b_1^2b_2^2b_3^2\log B}\right),
    \]
    for $i=2,3$.
\end{proposition}

\begin{proof}
    The error terms above are obtained by summing the error terms of Lemma \ref{innermaintermlemma} over the finitely many $\b{L}$ (or $\b{K}$) in $(\Z/8\Z)^{*4}$, finitely many $\b{q}\in\mathcal{A}(\b{m},\boldsymbol{\sigma})$, finitely many $\boldsymbol{\sigma}\in\{0,1\}^4$ and finally summing the $\frac{\tau(m_{ij})^3}{m_{ij}^2}$ over $\b{m}\in\N^4$ which converges. For the main term we note that $\mathfrak{C}(\b{b},\b{m},\boldsymbol{\sigma})$ is independent of $\b{q}$ and $\b{L}$ (or $\b{K}$) and that $\mathfrak{C}(\b{b},\b{m},\boldsymbol{\sigma}),\Sigma_{r,i}(\b{m},\boldsymbol{\sigma})\ll 1$ independently of $\b{b},\b{m}$ and $\boldsymbol{\sigma}$ so that we may extend the sum over the $m_{ij}$ to all of $\N^4$ at the cost of a sufficient error term. It is then clear that we may re-arrange the sums in the remaining variables to obtain $\mathfrak{C}_{r,2}(\b{b})$ (or $\mathfrak{C}_{r,3}(\b{b})$). This concludes the result.
\end{proof}

Finally, by setting
\[
\mathfrak{C}_{r,i} = \sum_{\substack{\b{b}\in\N^4\\ \eqref{bgcdconditions}}}\frac{\mathfrak{C}_{r,i}(\b{b})}{b_0^2b_1^2b_2^2b_3^2}
\]
for $i=2,3$ and summing the results of Propositions \ref{Errorterms2}, \ref{Errorterms3} and \ref{Maintermswithb} over $b_i\leq z_0$ and then extending to $\N^4$ at the cost of another error term, we obtain the following:

\begin{proposition}\label{Maintermri}
For $B\geq 3$ we have
\[
N_{r,i}(B) = \frac{\mathfrak{C}_{r,i}B^2\log\log B}{\log B} + O_{A}\left(\frac{B^2\sqrt{\log\log B}}{\log B}\right)
\]
when $i=2,3$.
\end{proposition}

\section{The Constants}\label{constant analysis}
In this section we would like to express $\mathfrak{C}_{r,i}$ in a more concise manner. First we will compute the $\Sigma_{r,i}(\b{m},\boldsymbol{\sigma})$:

\subsection{Computation of $\Sigma_{r,2}$}\label{Sigmar2Comps}
Recall that
\[
\Sigma_{r,2}(\b{m},\boldsymbol{\sigma}) = \mathop{\sum}_{\substack{\b{q}\in\mathcal{A}(\b{m},\boldsymbol{\sigma})}}\mathop{\sum}_{\substack{\b{L}\in(\Z/8\Z)^{*4}\\ \eqref{twoadicmaintermconditions}}} 1,
\]
We first deal with the inner sum. We break the conditions \eqref{twoadicmaintermconditions} using the orthogonality relation \eqref{orthogonalityidentity}. This time however, there is no easy relation between the $q_i$ and so we use this relation to replace all four of the congruence relations above. It follows that
\begin{align*}
\Sigma_{r,2}(\b{m},\boldsymbol{\sigma}) = \mathop{\sum}_{\substack{\b{q}\in\mathcal{A}(\b{m},\boldsymbol{\sigma})}} \sum_{\substack{\chi,\chi',\chi'',\chi''' \\ \text{char.}\bmod{8}}}&\overline{\chi}(-\delta_r q_0 (m_{03}m_{12})_{\textrm{odd}}^{-1})\overline{\chi'}(q_1 (m_{03}m_{12})_{\textrm{odd}}^{-1})\\ &\times\overline{\chi''}(\delta_r q_2 (m_{02}m_{13})_{\textrm{odd}}^{-1})\overline{\chi'''}(-q_3 (m_{02}m_{13})_{\textrm{odd}}^{-1})\Sigma_{r,2}(\b{m},\boldsymbol{\sigma},\boldsymbol{\chi})
\end{align*}
where
\begin{align*}
\Sigma_{r,2}(\b{m},\boldsymbol{\sigma},\boldsymbol{\chi}) &= \frac{1}{4^4}\sum_{\b{L}\in(\Z/8\Z)^{*4}}\chi(L_0L_2)\chi'(L_1L_3)\chi''(L_1L_2)\chi'''(L_0L_3), \\
&=\frac{1}{4^4}\sum_{\b{L}\in(\Z/8\Z)^{*4}}\chi\chi'''(L_0)\chi'\chi''(L_1)\chi\chi''(L_2)\chi'\chi'''(L_3),\\
&=\hspace{-5pt}\left(\frac{1}{4}\sum_{L_0\in(\Z/8\Z)^{*}}\hspace{-15pt}\chi\chi'''(L_0)\right)\hspace{-5pt}\left(\frac{1}{4}\sum_{L_1\in(\Z/8\Z)^{*}}\hspace{-15pt}\chi'\chi''(L_1)\right)\hspace{-5pt}\left(\frac{1}{4}\sum_{L_2\in(\Z/8\Z)^{*}}\hspace{-15pt}\chi\chi''(L_2)\right)\hspace{-5pt}\left(\frac{1}{4}\sum_{L_3\in(\Z/8\Z)^{*}}\hspace{-15pt}\chi'\chi'''(L_3)\right),\\
&= \mathds{1}(\chi\chi''' = \chi_0)\mathds{1}(\chi'\chi'' = \chi_0)\mathds{1}(\chi\chi'' = \chi_0)\mathds{1}(\chi'\chi''' = \chi_0),
\end{align*}
and $\chi_{0}$ denotes the principal character modulo $8$.Since the group of principle characters modulo $8$ is isomorphic to $\Z/2\Z\times \Z/2\Z$, every element is equal to its own inverse. This leads us to deduce that
\begin{equation}\label{equalcharacterlemma}
    \mathds{1}(\chi\chi''' = \chi_0)\mathds{1}(\chi'\chi'' = \chi_0)\mathds{1}(\chi\chi'' = \chi_0)\mathds{1}(\chi'\chi''' = \chi_0) = \mathds{1}(\chi=\chi'=\chi''=\chi''')
\end{equation}
Thus,
$\Sigma_{r,2}(\b{m},\boldsymbol{\sigma},\boldsymbol{\chi}) = \mathds{1}(\chi=\chi'=\chi''=\chi''')$. Substituting this into the expression for $\Sigma_{r,2}(\b{m},\boldsymbol{\sigma})$ we obtain
\begin{align*}
    \Sigma_{r,2}(\b{m},\boldsymbol{\sigma}) &= \mathop{\sum}_{\substack{\b{q}\in\mathcal{A}(\b{m},\boldsymbol{\sigma})}} \sum_{\substack{\chi\;\text{char.}\bmod{8}}}\overline{\chi}( q_0q_1q_2q_3 (-\delta_r(m_{02}m_{03}m_{12}m_{13})_{\textrm{odd}})^{-2})\\
    &= \mathop{\sum}_{\substack{\b{q}\in\mathcal{A}(\b{m},\boldsymbol{\sigma})}} \sum_{\substack{\chi\;\text{char.}\bmod{8}}}\overline{\chi}(q_0q_1q_2q_3)\\
    &= 4\mathop{\sum}_{\substack{\b{q}\in\mathcal{A}(\b{m},\boldsymbol{\sigma})\\ q_0q_1q_2q_3\equiv 1\bmod{8}}} 1 = 4\#\{\b{q}\in\mathcal{A}(\b{m},\boldsymbol{\sigma}): q_0q_1q_2q_3\equiv 1\bmod{8}\}.
\end{align*}
Now, recalling the definition of $\mathcal{A}(\b{m},\boldsymbol{\sigma})$, it follows that
\begin{align*}
    \Sigma_{r,2}(\b{m},\boldsymbol{\sigma}) \hspace{-4pt}= \hspace{-4pt}\begin{cases}
        4\#\{\b{q}\in\mathcal{A}_1: q_0q_1q_2q_3\equiv 1\bmod{8}\}\;\text{if}\;2\nmid m_{02}m_{03}m_{12}m_{13}\; \&\;\sigma_{i}=0\;\forall\; i\in\{0,1,2,3\},\\
        4\#\{\b{q}\in\mathcal{A}_2: q_0q_1q_2q_3\equiv 1\bmod{8}\}\;\text{otherwise}.
    \end{cases}
\end{align*}
This is because $\mathcal{A}(\b{m},\boldsymbol{\sigma})$ is equal to $\mathcal{A}_1$ in the first case and is easily seen to be bijective to $\mathcal{A}_2$ by applying a single permutation to the components of each of its elements otherwise, an operation which does not effect the condition $q_0q_1q_2q_3\equiv 1\bmod{8}$.
This extra condition significantly simplifies the counting process. This is done in the following lemma:

\begin{lemma}\label{mod8SOLUTIONSr2}
    We have the following
    \begin{align*}
    \Sigma_{r,2}(\b{m},\boldsymbol{\sigma}) = \begin{cases}
        192\;\text{if}\;2\nmid m_{02}m_{03}m_{12}m_{13}\; \&\;\sigma_{i}=0\;\forall\; i\in\{0,1,2,3\},\\
        128\;\text{otherwise}.
    \end{cases}
\end{align*}
\end{lemma}

\begin{proof}
First assume that $2\nmid m_{02}m_{03}m_{12}m_{13}\; \&\;\sigma_{i}=0\;\forall\; i\in\{0,1,2,3\}$. Then we want compute $
\#\{\b{q}\in\mathcal{A}_1: q_0q_1q_2q_3\equiv 1\bmod{8}\}.$
By running the code given in \S\ref{APPENDIX}, we find that
$$\#\{\b{q}\in\mathcal{A}_1: q_0q_1q_2q_3\equiv 1\bmod{8}\} = 48.$$
We remark that this set corresponds to \verb|A1| in the given code. Multiplying this by $4$ gives the first case of the result. For the second case, if $2|m_{02}m_{03}m_{12}m_{13}$ or $\sigma_i=1$ for some $i\in\{0,1,2,3\}$, then we want to compute $
\#\{\b{q}\in\mathcal{A}_2: q_0q_1q_2q_3\equiv 1\bmod{8}\}$. 
This set corresponds to \verb|A2| in the code given in \S\ref{APPENDIX}. Running this code yields,
$$\#\{\b{q}\in\mathcal{A}_2: q_0q_1q_2q_3\equiv 1\bmod{8}\} = 32.$$
Multiplying this by $4$ gives the second part of the result.
\end{proof}

\subsection{Computation of $\Sigma_{r,3}(\b{m},\boldsymbol{\sigma})$}\label{Sigmar3Comps}
Recall that
\[
\Sigma_{r,3}(\b{m},\boldsymbol{\sigma}) = \mathop{\sum}_{\substack{\b{q}\in\mathcal{A}(\b{m},\boldsymbol{\sigma})}}\mathop{\sum}_{\substack{\b{K}\in(\Z/8\Z)^{*4}\\ \eqref{twoadicmaintermconditionsrepeated}}} \Theta_{r,1}(\b{m}_{\text{odd}},\b{K},\boldsymbol{\sigma}),
\]
where $\Theta_{r,1}(\b{m}_{\text{odd}},\b{K},\boldsymbol{\sigma})$ is defined as
\[ (-1)^{f_r(\b{m}_{\text{odd}},\b{K})}\left(\frac{2^{\sigma_0+\sigma_1+\sigma_2+\sigma_3}}{(m_{02}m_{03}m_{12}m_{13})_{\textrm{odd}}}\right)\left(\frac{2^{\sigma_2+\sigma_3}}{K_0K_1}\right)\left(\frac{2^{\sigma_0+\sigma_1}}{K_2K_3}\right)\left(\frac{2^{v_2(m_{02}m_{03}m_{12}m_{13})}}{K_0K_1K_2K_3}\right).
\]
and
\begin{equation}\label{twoadicmaintermconditionsrepeated}
    \begin{cases}
        K_0K_2(m_{03}m_{12})_{\textrm{odd}}\equiv -\delta_r q_0\bmod{8},\;
        K_1K_3(m_{03}m_{12})_{\textrm{odd}}\equiv q_1\bmod{8},\\
        K_1K_2(m_{02}m_{13})_{\textrm{odd}}\equiv \delta_r q_2\bmod{8},\;
        K_0K_3(m_{02}m_{13})_{\textrm{odd}}\equiv -q_3\bmod{8}.
    \end{cases}
\end{equation}
We may then use \eqref{twoadicmaintermconditionsrepeated} to write $f_r(\b{m}_{\textrm{odd}},\b{K})$ as
\[
\equiv \frac{(K_0K_1+(\delta_r-2)K_2K_3)(q_0q_2+1) + (1-\delta_r)(K_2K_3 - 1)}{4}\bmod{2}
\]
to remove the dependence on $\b{m}_{\textrm{odd}}$. Now notice that
\[
4\mid (K_0K_1+(\delta_r-2)K_2K_3)(q_0q_2+1)\;\;\text{and}\;\;4\mid (1-\delta_r)(K_2K_3 - 1).
\]
Thus we write
\[
(-1)^{f_r(\b{m},\b{K})} = (-1)^{\Tilde{f}_r(\b{q},\b{K})}\left(\frac{-1}{K_2K_3}\right)^{\frac{(1-\delta_r)}{2}}
\]
where
\[
\Tilde{f}_r(\b{q},\b{K}) = \frac{(K_0K_1+(\delta_r-2)K_2K_3)(q_0q_2+1)}{4}.
\]
Once more appealing to \eqref{twoadicmaintermconditionsrepeated} we may see that
\begin{align*}
\Tilde{f}_r(\b{q},\b{K}) &\equiv \frac{(-q_0q_2+(1-2\delta_r)q_0q_3)(m_{02}m_{03}m_{12}m_{13})_{\textrm{odd}}(q_0q_2+1)}{4}\bmod{2}\\
&\equiv \frac{((2\delta_r-1)q_0q_3+q_0q_2)(q_0q_2+1)}{4}\bmod{2}
\end{align*}
since $(m_{02}m_{03}m_{12}m_{13})_{\textrm{odd}}$ is odd. Now by expanding out the numerator and recalling that \eqref{twoadicmaintermconditionsrepeated} asserts that $q_0q_1q_2q_3\equiv 1\bmod{8}$,
\begin{align*}
\Tilde{f}_r(\b{q},\b{K}) &\equiv \frac{((2\delta_r-1)q_0q_3+q_0q_2)(q_0q_2+1)}{4}\bmod{2}\\
&\equiv \frac{q_0(q_0+q_2+(2\delta_r-1)(q_1+q_3))}{4}\bmod{2}\\
&\equiv \frac{(q_0+q_2+(2\delta_r-1)(q_1+q_3))}{4}\bmod{2}.
\end{align*}
Setting the final ratio above to be $\Tilde{f}_r(\b{q})$ we therefore write
\[
\Tilde{\Theta}_{r,1}(\b{m},\b{q},\boldsymbol{\sigma}) = (-1)^{\Tilde{f}_r(\b{q})}\left(\frac{2^{\sigma_0+\sigma_1}}{\delta_rq_0q_3}\right)\left(\frac{2^{\sigma_2 +\sigma_3}}{7q_0q_2}\right)\left(\frac{2^{v_2(m_{02}m_{03}m_{12}m_{13})}}{7\delta_r q_0q_1}\right)
\]
and thus obtain
\[
\Sigma_{r,3}(\b{m},\boldsymbol{\sigma}) = 4\sum_{\b{q}\in\mathcal{A}(\b{m},\boldsymbol{\sigma})}\Tilde{\Theta}_{r,1}(\b{m},\b{q},\boldsymbol{\sigma})\sum_{\substack{\b{K}\in(\Z/8\Z)^{*4}\\\eqref{twoadicmaintermconditionsrepeated}}} \left(\frac{-1}{K_2K_3}\right)^{\frac{(1-\delta_r)}{2}}.
\]
Using the orthogonality relation \eqref{orthogonalityidentity} to break the condition \eqref{twoadicmaintermconditionsrepeated} the inner most sum over $\b{K}$ becomes
\begin{align*}
\sum_{\substack{\chi,\chi',\chi'',\chi'''\\ \text{char. $\bmod{8}$}}} &\overline{\chi}(-\delta_rq_0 (m_{03}m_{12})_{\textrm{odd}}^{-1})\overline{\chi'}(q_1 (m_{03}m_{12})_{\textrm{odd}}^{-1}) \\ &\times\overline{\chi''}(\delta_rq_2 (m_{02}m_{13})_{\textrm{odd}}^{-1})\overline{\chi'''}(-q_3 (m_{02}m_{13})_{\textrm{odd}}^{-1})\Sigma_{r,3}(\b{m},\boldsymbol{\sigma},\boldsymbol{\chi})
\end{align*}
where
\begin{align*}
\Sigma_{r,3}(\b{m},\boldsymbol{\sigma},\boldsymbol{\chi}) &= \frac{1}{4^4}\sum_{\b{K}\in(\Z/8\Z)^{*4}}\chi(K_0K_2)\chi'(K_1K_3)\chi''(K_1K_2)\chi'''(K_0K_3)\left(\frac{-1}{K_2K_3}\right)^{\frac{(1-\delta_r)}{2}}\\
 &=\mathds{1}\left(\chi=\chi'\left(\frac{-1}{\cdot}\right)^{\frac{(1-\delta_r)}{2}}=\chi''\left(\frac{-1}{\cdot}\right)^{\frac{(1-\delta_r)}{2}}=\chi'''\right)
\end{align*}
using the same method as for $\Sigma_{r,2}(\b{m},\boldsymbol{\sigma},\boldsymbol{\chi})$ previously, noting that the reasoning behind \eqref{equalcharacterlemma} also applies here. Then
\begin{align*}
\sum_{\substack{\b{K}\in(\Z/8\Z)^{*4}\\\eqref{twoadicmaintermconditionsrepeated}}} \left(\frac{-1}{K_2K_3}\right)^{\frac{(1-\delta_r)}{2}} &= \sum_{\substack{\chi\;\text{char.}\\ \bmod{8}}}\overline{\chi}(q_0q_1q_2q_3)\left(\frac{-1}{\delta_rq_1q_2(m_{02}m_{03}m_{12}m_{13})_{\textrm{odd}}^{-1}}\right)^{\frac{(1-\delta_r)}{2}}\\
& =\mathds{1}(q_0q_1q_2q_2\equiv 1\bmod{8})\left(\frac{-1}{\delta_rq_1q_2(m_{02}m_{03}m_{12}m_{13})_{\textrm{odd}}}\right)^{\frac{(1-\delta_r)}{2}}.
\end{align*}
Collecting this information gives
\[
\Sigma_{r,3}(\b{m},\boldsymbol{\sigma}) = 4\left(\frac{-1}{(m_{02}m_{03}m_{12}m_{13})_{\textrm{odd}}}\right)^{\frac{(1-\delta_r)}{2}}\sum_{\substack{\b{q}\in\mathcal{A}\\ q_0q_1q_2q_3\equiv 1\bmod{8}}}\Theta_{r,1}'(\b{m},\boldsymbol{\sigma},\b{q})
\]
where
\begin{align*}
    \Theta_{r,1}'(\b{m},\boldsymbol{\sigma},\b{q}) &=(-1)^{\Tilde{f}_r(\b{q})}\left(\frac{2^{\sigma_0+\sigma_1}}{\delta_rq_0q_3}\right)\left(\frac{2^{\sigma_2 +\sigma_3}}{7q_0q_2}\right)\left(\frac{2^{v_2(m_{02}m_{03}m_{12}m_{13})}}{7\delta_r q_0q_1}\right)\left(\frac{-1}{\delta_rq_1q_2}\right)^{\frac{(1-\delta_r)}{2}}\\
    &= (-1)^{\Tilde{f}_r(\b{q})}\left(\frac{2^{\sigma_0+\sigma_1}}{q_0q_3}\right)\left(\frac{2^{\sigma_2 +\sigma_3}}{q_0q_2}\right)\left(\frac{2^{v_2(m_{02}m_{03}m_{12}m_{13})}}{ q_0q_1}\right)\left(\frac{-1}{\delta_rq_1q_2}\right)^{\frac{(1-\delta_r)}{2}}
\end{align*}
since $\left(\frac{2}{7}\right) = \left(\frac{2}{\delta_r}\right) =1$. We now split into cases $r=1$ and $r=2$.

\subsubsection{The Case $r=1$} In this case,
\[
\Theta'_{1,1}(\b{m},\b{q},\boldsymbol{\sigma}) = (-1)^{\Tilde{f}_1(\b{q})}\left(\frac{2^{\sigma_0+\sigma_1}}{q_0q_3}\right)\left(\frac{2^{\sigma_2 +\sigma_3}}{q_0q_2}\right)\left(\frac{2^{v_2(m_{02}m_{03}m_{12}m_{13})}}{q_0q_1}\right)
\]
and
\[
\Tilde{f}_1(\b{q}) = \frac{(q_1+q_2+q_3+q_0)}{4}.
\]
We note that, since $\b{q}\in\mathcal{A}(\b{m},\boldsymbol{\sigma})$ and $q_0q_1q_2q_3\equiv 1\bmod{8}$, this exponent is always an integer. This can be seen by noting that such $\b{q}\in(\Z/8\Z)^{*4}$ are component wise permutations of points in $\{\b{q}\in\mathcal{A}_1: q_0q_1q_2q_3\equiv 1\bmod{8}\}$ or $\{\b{q}\in\mathcal{A}_2: q_0q_1q_2q_3\equiv 1\bmod{8}\}$ and so by considering the outputs of \verb|print A1| and \verb|print A2| from the code in \S\ref{APPENDIX} respectively, it is easy to check by hand that all elements of these sets have a component sum which is $0$ or $4$ modulo $8$. Now we split into cases determined by the values of $v_2(m_{02}m_{03}m_{12}m_{13})$ and $\boldsymbol{\sigma}$. Here the precise definition of $\mathcal{A}(\b{m},\boldsymbol{\sigma})$ is required since the relative positions of the $q_i$ will effect the value of the Jacobi symbol. We have the following cases:
\begin{itemize}
    \item[$(a)$] if $2\nmid m_{02}m_{03}m_{12}m_{13}$, and $\sigma_i=0\;\forall\; i\in\{0,1,2,3\}$,
    \begin{equation*}
    \Sigma_{1,3}(\b{m},\boldsymbol{\sigma}) = 4\sum_{\substack{\b{q}\in\mathcal{A}_1\\ q_0q_1q_2q_3\equiv 1\bmod{8}}}(-1)^{\frac{q_0+q_1+q_2+q_3}{4}},
    \end{equation*}

    \item[$(b)$] if $2\mid m_{03}m_{12},\;2\nmid m_{02}m_{13}$ and $\sigma_i=0\;\forall\; i\in\{0,1,2,3\}$,
    \begin{equation*}
    \Sigma_{1,3}(\b{m},\boldsymbol{\sigma}) = 4\sum_{\substack{\b{q}\in\mathcal{A}_{0,1,2,3}\\ q_0q_1q_2q_3\equiv 1\bmod{8}}}(-1)^{\frac{q_0+q_1+q_2+q_3}{4}}\left(\frac{2}{q_0q_1}\right),
    \end{equation*}

    \item[$(c)$] if $2\mid m_{02}m_{13},\;2\nmid m_{02}m_{13}$ and $\sigma_i=0\;\forall\; i\in\{0,1,2,3\}$,
    \begin{equation*}
    \Sigma_{1,3}(\b{m},\boldsymbol{\sigma}) = 4\sum_{\substack{\b{q}\in\mathcal{A}_{2,3,0,1}\\ q_0q_1q_2q_3\equiv 1\bmod{8}}}(-1)^{\frac{q_0+q_1+q_2+q_3}{4}}\left(\frac{2}{q_0q_1}\right),
    \end{equation*}

    \item[$(d)$] if $2\nmid m_{03}m_{12}m_{02}m_{13},\;\sigma_0=1$ and $\sigma_i=0\;\forall\; i\in\{0,1,2,3\}\setminus\{0\}$,
    \begin{equation*}
    \Sigma_{1,3}(\b{m},\boldsymbol{\sigma}) = 4\sum_{\substack{\b{q}\in\mathcal{A}_{0,3,1,2}\\ q_0q_1q_2q_3\equiv 1\bmod{8}}}(-1)^{\frac{q_0+q_1+q_2+q_3}{4}}\left(\frac{2}{q_0q_3}\right),
    \end{equation*}
    
    \item[$(e)$] if $2\nmid m_{03}m_{12}m_{02}m_{13},\sigma_1=1$ and $\sigma_i=0\;\forall\; i\in\{0,1,2,3\}\setminus\{1\}$,
    \begin{equation*}
    \Sigma_{1,3}(\b{m},\boldsymbol{\sigma}) = 4\sum_{\substack{\b{q}\in\mathcal{A}_{1,2,0,3}\\ q_0q_1q_2q_3\equiv 1\bmod{8}}}(-1)^{\frac{q_0+q_1+q_2+q_3}{4}}\left(\frac{2}{q_0q_3}\right),
    \end{equation*}

    \item[$(f)$] if $2\nmid m_{03}m_{12}m_{02}m_{13},\sigma_2=1$ and $\sigma_i=0\;\forall\; i\in\{0,1,2,3\}\setminus\{2\}$,
    \begin{equation*}
    \Sigma_{1,3}(\b{m},\boldsymbol{\sigma}) = 4\sum_{\substack{\b{q}\in\mathcal{A}_{0,2,1,3}\\ q_0q_1q_2q_3\equiv 1\bmod{8}}}(-1)^{\frac{q_0+q_1+q_2+q_3}{4}}\left(\frac{2}{q_0q_2}\right),
    \end{equation*}

    \item[$(g)$] if $2\nmid m_{03}m_{12}m_{02}m_{13},\sigma_3=1$ and $\sigma_i=0\;\forall\; i\in\{0,1,2,3\}\setminus\{3\}$,
    \begin{equation*}
    \Sigma_{1,3}(\b{m},\boldsymbol{\sigma}) = 4\sum_{\substack{\b{q}\in\mathcal{A}_{1,3,0,2}\\ q_0q_1q_2q_3\equiv 1\bmod{8}}}(-1)^{\frac{q_0+q_1+q_2+q_3}{4}}\left(\frac{2}{q_0q_2}\right).
    \end{equation*}
\end{itemize}

These sums are computed using the code in \S\ref{APPENDIX}. In this code the sum in case $(a)$ is \verb|Sigma13a| in the code and the sums in cases $(b)$ through $(g)$ correspond to the \verb|Sigma13Perm| inside the "for" loop, the $i$th cycle of the loop corresponding to $i$th element of $\{b,c,d,e,f,g\}$. By running this code we obtain the following:

\begin{lemma}\label{mod8SOLUTIONS13}
    We have the following
    \begin{align*}
    \Sigma_{1,3}(\b{m},\boldsymbol{\sigma}) = \begin{cases}
        192\;\text{if}\;2\nmid m_{02}m_{03}m_{12}m_{13}\; \&\;\sigma_{i}=0\;\forall\; i\in\{0,1,2,3\},\\
        128\;\text{otherwise}.
    \end{cases}
\end{align*}
\end{lemma}

\subsubsection{The Case $r=2$} In this case
\[
\Theta'_{2,1}(\b{m},\b{q},\boldsymbol{\sigma}) = (-1)^{\Tilde{f}_2(\b{q})}\left(\frac{2^{\sigma_0+\sigma_1}}{q_0q_3}\right)\left(\frac{2^{\sigma_2 +\sigma_3}}{q_0q_2}\right)\left(\frac{2^{v_2(m_{02}m_{03}m_{12}m_{13})}}{q_0q_1}\right)\left(\frac{-1}{7q_1q_2}\right)
\]
and
\[
\Tilde{f}_2(\b{m}_{\textrm{odd}},\b{q}) = \frac{(q_0+5q_1+q_2+5q_3)}{4}.
\]
We split up into similar cases as before, again noting that the precise definition of $\mathcal{A}(\b{m},\boldsymbol{\sigma})$ is once more required since the relative positions of the $q_i$ will effect the value of the Jacobi symbol and the reciprocity factor. We have the following cases:
\begin{itemize}
    \item[$(a)$] if $2\nmid m_{02}m_{03}m_{12}m_{13}$, and $\sigma_i=0\;\forall\; i\in\{0,1,2,3\}$,
    \begin{equation*}
    \Sigma_{2,3}(\b{m},\boldsymbol{\sigma}) = 4\left(\frac{-1}{(m_{02}m_{03}m_{12}m_{13})_{\textrm{odd}}}\right)\sum_{\substack{\b{q}\in\mathcal{A}_1\\ q_0q_1q_2q_3\equiv 1\bmod{8}}}(-1)^{\frac{(q_0+5q_1+q_2+5q_3)}{4}}\left(\frac{-1}{7q_1q_2}\right),
    \end{equation*}

    \item[$(b)$] if $2\mid m_{03}m_{12},\;2\nmid m_{02}m_{13}$ and $\sigma_i=0\;\forall\; i\in\{0,1,2,3\}$,
    \begin{equation*}
    \Sigma_{2,3}(\b{m},\boldsymbol{\sigma}) = 4\left(\frac{-1}{(m_{02}m_{03}m_{12}m_{13})_{\textrm{odd}}}\right)\sum_{\substack{\b{q}\in\mathcal{A}_{0,1,2,3}\\ q_0q_1q_2q_3\equiv 1\bmod{8}}}(-1)^{\frac{(q_0+5q_1+q_2+5q_3)}{4}}\left(\frac{2}{q_0q_1}\right)\left(\frac{-1}{7q_1q_2}\right),
    \end{equation*}

    \item[$(c)$] if $2\mid m_{02}m_{13},\;2\nmid m_{02}m_{13}$ and $\sigma_i=0\;\forall\; i\in\{0,1,2,3\}$,
    \begin{equation*}
    \Sigma_{2,3}(\b{m},\boldsymbol{\sigma}) = 4\left(\frac{-1}{(m_{02}m_{03}m_{12}m_{13})_{\textrm{odd}}}\right)\sum_{\substack{\b{q}\in\mathcal{A}_{2,3,0,1}\\ q_0q_1q_2q_3\equiv 1\bmod{8}}}(-1)^{\frac{(q_0+5q_1+q_2+5q_3)}{4}}\left(\frac{2}{q_0q_1}\right)\left(\frac{-1}{7q_1q_2}\right),
    \end{equation*}

    \item[$(d)$] if $2\nmid m_{03}m_{12}m_{02}m_{13},\;\sigma_0=1$ and $\sigma_i=0\;\forall\; i\in\{0,1,2,3\}\setminus\{0\}$,
    \begin{equation*}
    \Sigma_{2,3}(\b{m},\boldsymbol{\sigma}) = 4\left(\frac{-1}{(m_{02}m_{03}m_{12}m_{13})_{\textrm{odd}}}\right)\sum_{\substack{\b{q}\in\mathcal{A}_{0,3,1,2}\\ q_0q_1q_2q_3\equiv 1\bmod{8}}}(-1)^{\frac{(q_0+5q_1+q_2+5q_3)}{4}}\left(\frac{2}{q_0q_3}\right)\left(\frac{-1}{7q_1q_2}\right),
    \end{equation*}
    
    \item[$(e)$] if $2\nmid m_{03}m_{12}m_{02}m_{13},\sigma_1=1$ and $\sigma_i=0\;\forall\; i\in\{0,1,2,3\}\setminus\{1\}$,
    \begin{equation*}
    \Sigma_{2,3}(\b{m},\boldsymbol{\sigma}) = 4\left(\frac{-1}{(m_{02}m_{03}m_{12}m_{13})_{\textrm{odd}}}\right)\sum_{\substack{\b{q}\in\mathcal{A}_{1,2,0,3}\\ q_0q_1q_2q_3\equiv 1\bmod{8}}}(-1)^{\frac{(q_0+5q_1+q_2+5q_3)}{4}}\left(\frac{2}{q_0q_3}\right)\left(\frac{-1}{7q_1q_2}\right),
    \end{equation*}

    \item[$(f)$] if $2\nmid m_{03}m_{12}m_{02}m_{13},\sigma_2=1$ and $\sigma_i=0\;\forall\; i\in\{0,1,2,3\}\setminus\{2\}$,
    \begin{equation*}
    \Sigma_{2,3}(\b{m},\boldsymbol{\sigma}) = 4\left(\frac{-1}{(m_{02}m_{03}m_{12}m_{13})_{\textrm{odd}}}\right)\sum_{\substack{\b{q}\in\mathcal{A}_{0,2,1,3}\\ q_0q_1q_2q_3\equiv 1\bmod{8}}}(-1)^{\frac{(q_0+5q_1+q_2+5q_3)}{4}}\left(\frac{2}{q_0q_2}\right)\left(\frac{-1}{7q_1q_2}\right),
    \end{equation*}

    \item[$(g)$] if $2\nmid m_{03}m_{12}m_{02}m_{13},\sigma_3=1$ and $\sigma_i=0\;\forall\; i\in\{0,1,2,3\}\setminus\{3\}$,
    \begin{equation*}
    \Sigma_{2,3}(\b{m},\boldsymbol{\sigma}) = 4\left(\frac{-1}{(m_{02}m_{03}m_{12}m_{13})_{\textrm{odd}}}\right)\sum_{\substack{\b{q}\in\mathcal{A}_{1,3,0,2}\\ q_0q_1q_2q_3\equiv 1\bmod{8}}}(-1)^{\frac{(q_0+5q_1+q_2+5q_3)}{4}}\left(\frac{2}{q_0q_2}\right)\left(\frac{-1}{7q_1q_2}\right).
    \end{equation*}
\end{itemize}

Again, these sums are given by the output of the code in \S\ref{APPENDIX}. The sum in case $(a)$ corresponds to \verb|Sigma23a|, while the cases $(b)$ through $(g)$ are given by \verb|Sigma23Perm| in the corresponding "for" loop. The output of this code gives the following:

\begin{lemma}\label{mod8SOLUTIONS23}
    We have the following
    \begin{align*}
    \Sigma_{2,3}(\b{m},\boldsymbol{\sigma}) = \begin{cases}
        64\left(\frac{-1}{(m_{02}m_{03}m_{12}m_{13})_{\textrm{odd}}}\right)\;\text{if}\;2\nmid m_{02}m_{03}m_{12}m_{13}\; \&\;\sigma_{i}=0\;\forall\; i\in\{0,1,2,3\},\\
        0\;\text{otherwise}.
    \end{cases}
\end{align*}
\end{lemma}

\subsection{Removing Dependency on $\boldsymbol{\sigma}$}\label{evenpartsofconstant}
In this section we deal with the condition \eqref{sigmaconditions}. In doing so we will simplify our expression for $\mathfrak{C}$ into sums over $\b{b}$ and $\b{m}$ whose components are all odd. To start we note that a close examination of $\overline{\mathfrak{S}}(\b{b},\b{m},v)$ and \eqref{gcdconditionsformaintermconstants} tells us that $\mathfrak{C}(\b{b},\b{m},\boldsymbol{\sigma})$ is independent of both $\boldsymbol{\sigma}$ and $v_2(m_{02}m_{03}m_{12}m_{13})$. Write
\[
\mathfrak{C}_{r,i}(\b{b}) = \frac{4f_0^4}{\phi(8)^4f_2^4}\mathop{\sum}_{\substack{\b{m}\in\N^4\\ \eqref{mconditions}}} \mathop{\sum}_{\substack{\boldsymbol{\sigma}\in\{0,1\}^4\\ \eqref{sigmaconditions}}} \frac{\overline{\mathfrak{S}}(\b{b},\b{m})\mathfrak{C}(\b{b},\b{m})\Sigma_{r,i}(\b{m},\boldsymbol{\sigma})}{2^{\sigma_0+\sigma_1+\sigma_2+\sigma_3}m_{02}^2m_{03}^2m_{12}^2m_{13}^2\tau((m_{02}m_{03}m_{12}m_{13})_{\textrm{odd}})}
\]
for $i=2,3$ where
\[
\mathfrak{C}(\b{b},\b{m}) = \sum_{v\in\N}\mu(v)\mathop{\sum\sum\sum\sum}_{{\substack{a_0,a_1,a_2,a_3\in\N \\ p|a_0a_1a_2a_3\Rightarrow p|v\\ v^2|a_0a_1a_2a_3, \eqref{gcdconditionsformaintermconstants}}}}\frac{\overline{\mathfrak{S}'}(v)}{a_0a_1a_2a_3\tau(a_0)\tau(a_1)\tau(a_2)\tau(a_3)}
\]
with
\[
\overline{\mathfrak{S}}(\b{b},\b{m}) = \prod_{i=0}^{3}\left(\prod_{\substack{p|m_{02}m_{12}m_{03}m_{13}b_i\\p\;\textrm{odd}}}f_p^{-1}\right)
\;\;\text{and}\;\;
\overline{\mathfrak{S}'}(v) = \prod_{i=0}^{3}\left(\prod_{\substack{p|v\\p\nmid m_{02}m_{12}m_{03}m_{13}b_i\\ p\;\textrm{odd}}}f_p^{-1}\right).
\]
noting that the dependency of $\mathfrak{C}(\b{b},\b{m})$ on $\b{b}$ and $\b{m}$ is contained in the condition \eqref{gcdconditionsformaintermconstants}. Now we observe that only the $\Sigma_{r,i}(\b{m},\boldsymbol{\sigma})$ depend
on $\boldsymbol{\sigma}$ and $v_2(m_{02}m_{03}m_{12}m_{13}) \in\{0,1\}$, allowing us to write $\mathfrak{C}_{r,i}(\b{b})$ as
\[
\frac{4f_0^4}{\phi(8)^4f_2^4}\mathop{\sum}_{\substack{\b{m}\in\N_{\textrm{odd}}^4\\ \eqref{mconditions}}} \frac{\overline{\mathfrak{S}}(\b{b},\b{m})\mathfrak{C}(\b{b},\b{m})}{m_{02}^2m_{03}^2m_{12}^2m_{13}^2\tau(m_{02}m_{03}m_{12}m_{13})} \mathop{\sum}_{\substack{\boldsymbol{\sigma},\boldsymbol{\Tilde{\sigma}}\in\{0,1\}^4\\ \eqref{NEWsigmaconditions} }} \frac{\Sigma_{r,i}(\b{m},\boldsymbol{\sigma},\boldsymbol{\Tilde{\sigma}})}{2^{\sigma_0+\sigma_1+\sigma_2+\sigma_3}4^{\Tilde{\sigma}_{02}+\Tilde{\sigma}_{03}+\Tilde{\sigma}_{12}+\Tilde{\sigma}_{13}}}
\]
where
\begin{equation}\label{NEWsigmaconditions}
    \begin{cases}\sigma_0+\sigma_1+\sigma_2+\sigma_3+\Tilde{\sigma}_{02}+\Tilde{\sigma}_{03}+\Tilde{\sigma}_{12}+\Tilde{\sigma}_{13}\leq 1,\\
     \mathrm{gcd}(2^{\sigma_0},b_1)=\mathrm{gcd}(2^{\sigma_1},b_0)=\mathrm{gcd}(2^{\sigma_2},b_3)=\mathrm{gcd}(2^{\sigma_3},b_2)=1,\\
     \mathrm{gcd}(2^{\Tilde{\sigma}_{02}},b_1b_3)=\mathrm{gcd}(2^{\Tilde{\sigma}_{03}},b_1b_2)=\mathrm{gcd}(2^{\Tilde{\sigma}_{12}},b_0b_3)=\mathrm{gcd}(2^{\Tilde{\sigma}_{13}},b_0b_2)=1.
 \end{cases}
\end{equation}
and
\[
\Sigma_{r,i}(\b{m},\boldsymbol{\sigma},\boldsymbol{\Tilde{\sigma}}) = \Sigma_{r,i}(\Tilde{\b{m}},\boldsymbol{\sigma})
\]
where here $\Tilde{\b{m}} = (2^{\Tilde{\sigma}_{02}}m_{02},2^{\Tilde{\sigma}_{03}}m_{03},2^{\Tilde{\sigma}_{12}}m_{12},2^{\Tilde{\sigma}_{13}}m_{13})$. Now using the previous subsection, we may compute the sum over $\boldsymbol{\sigma}$ and $\boldsymbol{\Tilde{\sigma}}$, which we will call $\Delta_{r,i}(\b{b},\b{m})$. We have the following:
\[
\Delta_{1,i}(\b{b},\b{m}) = 
192 + \#\{0\leq j\leq 3:2\nmid b_j\}\frac{128}{2} + \#\{0\leq j\leq 1:2\nmid b_j\}\cdot\#\{2\leq j\leq 3: 2\nmid b_j\}\frac{128}{4}
\]
for $i=2,3$ and
\[
\Delta_{2,2}(\b{b},\b{m}) = 
192 + \#\{0\leq j\leq 3:2\nmid b_j\}\frac{128}{2} + \#\{0\leq j\leq 1:2\nmid b_j\}\cdot\#\{2\leq j\leq 3: 2\nmid b_j\}\frac{128}{4}
\]
\begin{align*}
\Delta_{2,3}(\b{b},\b{m}) =& 
64\left(\frac{-1}{(m_{02}m_{03}m_{12}m_{13})_{\textrm{odd}}}\right).
\end{align*}
It follows that we may write
\[
\mathfrak{C}_{r,i}(\b{b}) = \frac{4f_0^4}{\phi(8)^4f_2^4}\mathop{\sum}_{\substack{\b{m}\in\N_{\textrm{odd}}^4\\ \eqref{mconditions}}} \frac{\overline{\mathfrak{S}}(\b{b},\b{m})\mathfrak{C}(\b{b},\b{m})\Delta_{r,i}(\b{b},\b{m})}{m_{02}^2m_{03}^2m_{12}^2m_{13}^2\tau(m_{02}m_{03}m_{12}m_{13})}.
\]
Next, we remove the even parts in the sum over $b_i$. Noting that only the $\Delta(\b{b},\b{m})$ depend on the even part of the $b_i$, we write $b_i=(b_i)_{\text{odd}}2^{\mu_i}$ for $\mu_i=v_2(b_i)$. Then
\[
\mathfrak{C}_{r,i} = \frac{4f_0^4}{\phi(8)^4f_2^4}\sum_{\substack{\b{b},\b{m}\in\N_{\textrm{odd}}^4\\ \eqref{bgcdconditions},\eqref{mconditions}}}\frac{\overline{\mathfrak{S}}(\b{b},\b{m})\mathfrak{C}(\b{b},\b{m})}{b_0^2b_1^2b_2^2b_3^2m_{02}^2m_{03}^2m_{12}^2m_{13}^2\tau(m_{02}m_{03}m_{12}m_{13})}\sum_{\substack{\boldsymbol{\mu}\in(\N\cup\{0\})^4\\\eqref{evenbgcdcondition}}} \frac{\Delta_{r,i}(\b{b},\b{m})}{4^{\mu_0+\mu_1+\mu_2+\mu_3}}
\]
where
\begin{equation}\label{evenbgcdcondition}
        \mathrm{gcd}(2^{\mu_0},2^{\mu_1}) = \mathrm{gcd}(2^{\mu_2},2^{\mu_3}) = 1.
\end{equation}
Now,
\[
\sum_{\substack{\boldsymbol{\mu}\in(\N\cup\{0\})^4\\\eqref{evenbgcdcondition}}} \frac{1}{4^{\mu_0+\mu_1+\mu_2+\mu_3}} = \frac{25}{9},
\]
\[
\sum_{\substack{\boldsymbol{\mu}\in(\N\cup\{0\})^4\\\eqref{evenbgcdcondition}}} \frac{\#\{0\leq j\leq 3:\mu_j=0\}}{4^{\mu_0+\mu_1+\mu_2+\mu_3}} = 4 + 12\sum_{\substack{\Tilde{\mu}\in\N\\ \Tilde{\mu}>0}}\frac{1}{4^{\Tilde{\mu}}} + 8\mathop{\sum\sum}_{\substack{\Tilde{\mu}_0,\Tilde{\mu}_1\in\N\\\Tilde{\mu}_0,\Tilde{\mu}_1>0}}\frac{1}{4^{\Tilde{\mu}_0+\Tilde{\mu}_1}} = \frac{80}{9}
\]
and
\[
\sum_{\substack{\boldsymbol{\mu}\in(\N\cup\{0\})^4\\\eqref{evenbgcdcondition}}} \frac{\#\{0\leq j\leq 1:\mu_j=0\}\cdot\#\{2\leq j\leq 3: \mu_j=0\}}{4^{\mu_0+\mu_1+\mu_2+\mu_3}} = 4 + 8\sum_{\substack{\Tilde{\mu}\in\N\\ \Tilde{\mu}>0}}\frac{1}{4^{\Tilde{\mu}}} + 4\mathop{\sum\sum}_{\substack{\Tilde{\mu}_0,\Tilde{\mu}_1\in\N\\\Tilde{\mu}_0,\Tilde{\mu}_1>0}}\frac{1}{4^{\Tilde{\mu}_0+\Tilde{\mu}_1}} = \frac{64}{9}.
\]
It follows that
\[
\sum_{\substack{\boldsymbol{\mu}\in(\N\cup\{0\})^4\\\eqref{evenbgcdcondition}}} \frac{\Delta_{r,i}(\b{b},\b{m})}{4^{\mu_0+\mu_1+\mu_2+\mu_3}} = \frac{4800}{9} + \frac{5120}{9} + \frac{2048}{9} = \frac{11968}{9}
\]
for $(r,i)=(1,2),(1,3),(2,2)$, and
\[
\sum_{\substack{\boldsymbol{\mu}\in(\N\cup\{0\})^4\\\eqref{evenbgcdcondition}}} \frac{\Delta_{2,3}(\b{b},\b{m})}{4^{\mu_0+\mu_1+\mu_2+\mu_3}} = \frac{1600}{9}\left(\frac{-1}{(m_{02}m_{03}m_{12}m_{13})_{\textrm{odd}}}\right).
\]
Now for $m\in\N$ and $(r,i)\in\{(1,2),(1,3),(2,2),(2,3)\}$ define
\begin{align*}
\rho_{(r,i)} = \begin{cases}
    \frac{11968}{9}\;\;&\text{if}\;\;(r,i)\in\{(1,2),(1,3),(2,2)\},\\
    \frac{1600}{9}\;\;&\text{if}\;\;(r,i)=(2,3),
\end{cases}
\end{align*}
and
\begin{align*}
\rho'_{(r,i)}(m) = \begin{cases}
    1\;\;&\text{if}\;\;(r,i)\in\{(1,2),(1,3),(2,2)\},\\
    \left(\frac{-1}{m}\right)\;\;&\text{if}\;\;(r,i)=(2,3).
\end{cases}
\end{align*}
Then
\[
\mathfrak{C}_{r,i} = \frac{f_0^4\rho_{(r,i)}}{64f_2^4}\sum_{\substack{\b{b},\b{m}\in\N_{\textrm{odd}}^4\\ \eqref{bgcdconditions},\eqref{mconditions}}}\frac{\overline{\mathfrak{S}}(\b{b},\b{m})\mathfrak{C}(\b{b},\b{m})\rho'_{(r,i)}(m_{02}m_{03}m_{12}m_{13})}{b_0^2b_1^2b_2^2b_3^2m_{02}^2m_{03}^2m_{12}^2m_{13}^2\tau(m_{02}m_{03}m_{12}m_{13})}
\]
for all $(r,i)\in\{(1,2),(1,3),(2,2),(2,3)\}$.

\subsection{Simplification of $\mathfrak{C}(\b{b},\b{m})$} 
Let $\b{x}\in\N^4$ and define
\[
\mathfrak{C}(\b{x}) = \sum_{v\in\N}\mu(v)\mathop{\sum\sum\sum\sum}_{{\substack{a_0,a_1,a_2,a_3\in\N \\ p|a_0a_1a_2a_3\Rightarrow p|v\\ v^2|a_0a_1a_2a_3\\ \mathrm{gcd}(a_i,2x_i)=1\forall i}}}\frac{\overline{\mathfrak{S}'}(v,\b{x})}{a_0a_1a_2a_3\tau(a_0)\tau(a_1)\tau(a_2)\tau(a_3)}.
\]
where
\[
\mathfrak{S}'(v,\b{x}) = \prod_{i=0}^{3}\left(\frac{1}{\prod_{\substack{p|v\\p\nmid 2x_i}}f_p}\right).
\]
Then we have the following,
\begin{lemma}\label{simplifybreaksquarefreeconstant}
    For any $\b{x}\in\N^4$,
    \[
    \mathfrak{C}(\b{x}) = \prod_{p\neq 2}\frac{1}{f_p^{\#\{0\leq i\leq 3:p\nmid x_i\}}}\left(1+\frac{\#\{0\leq i\leq 3:p\nmid x_i\}}{2p}\right).
    \]
\end{lemma}

\begin{proof}
    We write,
    \[
    \mathfrak{C}(\b{x}) = \sum_{v\in\N_{\textrm{odd}}}\mu(v)\overline{\mathfrak{S}'}(v,\b{x})\sum_{\substack{w\in\N_{\textrm{odd}} \\ v^2|w \\ p|w\Rightarrow p|v}}\frac{1}{w}\mathop{\sum}_{{\substack{\b{a}\in\N^4 \\a_0a_1a_2a_3=w\\ \mathrm{gcd}(a_i,2x_i)=1\forall i}}}\frac{1}{\tau(a_0)\tau(a_1)\tau(a_2)\tau(a_3)}.
    \]
    Then the sum over $w$ is
    \[
    \prod_{p|v}\left(\sum_{j=2}^{\infty}\frac{1}{p^j}\sum_{\substack{\b{e}\in\Z_{\geq 0}^4\\ e_0+e_1+e_2+e_3=j\\ \mathrm{gcd}(p^{e_i},x_i)=1\forall i}}\prod_{i=0}^{3}\left(\frac{1}{e_i+1}\right)\right) = \prod_{p|v}\left(f_p^{\#\{0\leq i\leq 3:p\nmid x_i\}}-1-\frac{\#\{0\leq i\leq 3:p\nmid x_i\}}{2p}\right)
    \]
    the equality coming from adding in the terms for which $j=0$ and $j=1$. Call the term inside the product $c_p$ for each prime $p$ then by summing over $v$ we conclude that
    \[
    \mathfrak{C}(\b{x}) = \prod_{p\neq 2}\left(1-\frac{c_p}{f_p^{\#\{0\leq i\leq 3:p\nmid x_i\}}}\right) = \prod_{p\neq 2}\frac{1}{f_p^{\#\{0\leq i\leq 3:p\nmid x_i\}}}\left(1+\frac{\#\{0\leq i\leq 3:p\nmid x_i\}}{2p}\right).
    \]
\end{proof}

From this, we may prove the following:

\begin{lemma}
    For $\b{b},\b{m}\in\N_{\textrm{odd}}^{4}$ satisfying \eqref{bgcdconditions} and \eqref{mconditions},
    \[
    \frac{f_0^4}{f_2^4}\overline{\mathfrak{S}}(\b{b},\b{m})\mathfrak{C}(\b{b},\b{m}) = \frac{1}{(2\pi)^2}\prod_{p\neq 2}\left(1-\frac{1}{p}\right)^2\left(1+\frac{\#\{0\leq i\leq 3:p\nmid m_{02}m_{03}m_{12}m_{13}b_i\}}{2p}\right)
    \]
\end{lemma}

\begin{proof}
    From Lemma \ref{simplifybreaksquarefreeconstant} we have
    \[
    \mathfrak{C}(\b{b},\b{m}) = \prod_{p\neq 2}\frac{1}{f_p^{\#\{0\leq i\leq 3:p\nmid m_{02}m_{03}m_{12}m_{13}b_i\}}}\left(1+\frac{\#\{0\leq i\leq 3:p\nmid m_{02}m_{03}m_{12}m_{13}b_i\}}{2p}\right)
    \]
    and by re-arranging we have
    \[
    \overline{\mathfrak{S}}(\b{b},\b{m}) = \prod_{i=0}^{3}\left(\frac{1}{\prod_{\substack{p|m_{02}m_{12}m_{03}m_{13}b_i\\p\;\textrm{odd}}}f_p}\right) = \prod_{p\neq 2}\left(\frac{1}{f_p^{\#\{0\leq i\leq 3:p| m_{02}m_{03}m_{12}m_{13}b_i\}}}\right).
    \]
    Thus
    \[
    \mathfrak{C}(\b{b},\b{m})\overline{\mathfrak{S}}(\b{b},\b{m}) = \prod_{p\neq 2}\frac{1}{f_p^{4}}\left(1+\frac{\#\{0\leq i\leq 3:p\nmid m_{02}m_{03}m_{12}m_{13}b_i\}}{2p}\right),
    \]
    and by recalling the definition of $f_0$, the result follows.
\end{proof}

Define the function $\gamma:\N_{\textrm{odd}}^4\rightarrow \R$ by
\[
\gamma(\b{x}) = \prod_{p\neq 2}\left(1-\frac{1}{p}\right)^2\left(1+\frac{\#\{0\leq i\leq 3:p\nmid x_i\}}{2p}\right).
\]
Then what we have now shown is that
\[
\mathfrak{C}_{r,i} = \frac{\rho_{(r,i)}}{64(2\pi)^2}\sum_{\substack{\b{b},\b{m}\in\N_{\textrm{odd}}^4\\ \eqref{bgcdconditions},\eqref{mconditions}}}\frac{\gamma(m_{02}m_{03}m_{12}m_{13}\b{b})\rho'_{(r,i)}(m_{02}m_{03}m_{12}m_{13})}{b_0^2b_1^2b_2^2b_3^2m_{02}^2m_{03}^2m_{12}^2m_{13}^2\tau(m_{02}m_{03}m_{12}m_{13})}.
\]

\subsection{Sum over $\b{m}$}
Noting that the summand only depends on the product of the components of $\b{m}$ we collect terms to write
\[
\mathfrak{C}_{r,i} = \frac{\rho_{(r,i)}}{256\pi^2}\sum_{\substack{\b{b}\in\N_{\textrm{odd}}^4\\ \eqref{bgcdconditions}}}\sum_{m\in\N_{\textrm{odd}}}\frac{\mu^2(m)\gamma(m\b{b})\rho'_{(r,i)}(m)}{b_0^2b_1^2b_2^2b_3^2m^2\tau(m)}\left(\sum_{\substack{m_{02}m_{03}m_{12}m_{13}=m\\\eqref{mconditions}}}1\right).
\]
This inner sum is a four-fold Dirichlet convolution of (completely) multiplicative functions (indicator functions of the gcd conditions in \eqref{mconditions}) applied to a square-free integer $m$. By considering its behaviour for $m$ prime, we may therefore deduce that this inner sum can be written as
\[
\mu^2(m)\beta(m,\b{b}) = \mu^2(m)\prod_{\substack{p|m\\ p\neq 2}}\left(\#\{0\leq i\leq 1:p\nmid b_i\}\cdot\#\{2\leq i\leq 3:p\nmid b_i\}\right),
\]
which is multiplicative in $m$. Note also that we may untangle the dependence of $m$ from $\gamma(m\b{b})$ since
\[
\gamma(m\b{b}) = \gamma(\b{b})\gamma_{0}(m,\b{b})
\]
where
\[
\gamma_{0}(m,\b{b}) = \prod_{\substack{p|m\\p\neq 2}}\left(1+\frac{\#\{0\leq i\leq 3:p\nmid b_i\}}{2p}\right)^{-1}
\]
which is also multiplicative in $m$. Therefore the sum over $m$ in $\mathfrak{C}_{r,i}$ becomes
\[
    \sum_{m\in\N_{\textrm{odd}}}\frac{\mu^2(m)\gamma_{0}(m,\b{b})\beta(m,\b{b})\rho'_{(r,i)}(m)}{m^2\tau(m)} = \prod_{p\neq 2}\left(1+\frac{\rho'_{(r,i)}(p)\gamma_{0}(p,\b{b})\beta(p,\b{b})}{2p^2}\right).
\]
Writing $\gamma_0(p)=\gamma_0(p,\b{1})$ and $\beta(p)=\beta(p,\b{1})$ we define
\[
\kappa^{(1)}_p = \left(1-\frac{1}{p}\right)^2\gamma_0(p)^{-1}
\;\;\text{and}\;\;
\kappa^{(2)}_p = \gamma_0(p)^{-1}\left(1+\frac{\rho'_{(r,i)}(p)\gamma_0(p)\beta(p)}{2p^2}\right) = \left(1+\frac{2}{p}+\frac{2\rho'_{(r,i)}(p)}{p^2}\right).
\]
Then
\[
\gamma(\b{b}) = \left(\prod_{p\neq 2}\kappa^{(1)}_p\right)g^{(1)}(\b{b})\;\;\text{and}\;\;
\prod_{p\neq 2}\left(1+\frac{\rho'_{(r,i)}(p)\gamma_{0}(p,\b{b})\beta(p,\b{b})}{2p^2}\right) = \left(\prod_{p\neq 2}\gamma_0(p)\kappa^{(2)}_p\right)g^{(2)}(\b{b})
\]
where
\[
g^{(1)}(\b{b}) = \hspace{-5pt}\prod_{\substack{p|b_0b_1b_2b_3\\p\neq 2}}\hspace{-5pt}
\gamma_0(p)\gamma_0(p,\b{b})^{-1}
\;\;\textrm{and}\;\;
g^{(2)}(\b{b}) = \hspace{-5pt}\prod_{\substack{p|b_0b_1b_2b_3\\p\neq 2}}\hspace{-5pt}(\gamma_0(p)\kappa^{(2)}_p)^{-1}\left(1+\frac{\rho'_{(r,i)}(p)\gamma_{0}(p,\b{b})\beta(p,\b{b})}{2p^2}\right)
\]
We are left with
\begin{align}\label{constantcheckpoint-onlybleft}
\mathfrak{C}_{r,i} 
&= \frac{\rho_{(r,i)}}{256\pi^2}\prod_{p\neq 2}\left(\left(1-\frac{1}{p}\right)^2\kappa^{(2)}_{p}\right)\sum_{\substack{\b{b}\in\N_{\textrm{odd}}^4\\ \eqref{bgcdconditions}}} \frac{g^{(1)}(\b{b})g^{(2)}(\b{b})}{b_0^2b_1^2b_2^2b_3^2}.
\end{align}

\subsection{Sum over $\b{b}$}
Let $g(\b{b})=g^{(1)}(\b{b})g^{(2)}(\b{b})$. Then $g(\b{b})$ may be seen to be
\begin{align*}
\prod_{\substack{p|b_0b_1b_2b_3\\p\neq 2}}(\kappa^{(2)}_p)^{-1}\left(1+\frac{\#\{0\leq i\leq 3:p\nmid b_i\}}{2p}+\frac{\rho'_{(r,i)}(p)\#\{0\leq i\leq 1:p\nmid b_i\}\cdot\#\{2\leq i\leq 3:p\nmid b_i\}}{2p^2}\right)
\end{align*}
which is clearly multiplicative in the sense that, if the products $b_0b_1b_2b_3$ and $\Tilde{b}_0\Tilde{b}_1\Tilde{b}_2\Tilde{b}_3$ are coprime, then $g(b_0\Tilde{b}_0,b_1\Tilde{b}_1,b_2\Tilde{b}_2,b_3\Tilde{b}_3)=g(\b{b})g(\Tilde{\b{b}})$. Therefore,
\begin{equation}\label{sumoverbforconstant}
\sum_{\substack{\b{b}\in\N_{\textrm{odd}}^4\\ \eqref{bgcdconditions}}} \frac{g(\b{b})}{b_0^2b_1^2b_2^2b_3^2} = \prod_{p\neq 2}\left(\sum_{\substack{\b{e}\in(\N\cup\{0\})^4\\ \min(e_0,e_1)=0\\\min(e_2,e_3)=0}}\frac{g(p^{e_0},p^{e_1},p^{e_2},p^{e_3})}{p^{2e_0+2e_1+2e_2+2e_3}}\right)
\end{equation}

We now consider the sums inside the product. There is a single term for which $e_i=0$ for all $i\in\{0,1,2,3\}$ given by
\begin{equation}\label{finalsumcontribution1}
    g(1,1,1,1)=1=(\kappa^{(2)}_p)^{-1}(\kappa^{(2)}_p) = (\kappa^{(2)}_p)^{-1}\left(\frac{p^2+2p+2\rho'_{(r,i)}(p)}{p^2}\right).
\end{equation}
When $e_i\geq 1$ for a single $i\in\{0,1,2,3\}$,
\begin{equation*}
    g(p^{e_0},p^{e_1},p^{e_2},p^{e_3}) = (\kappa^{(2)}_p)^{-1}\left(1+\frac{3}{2p}+\frac{\rho'_{(r,i)}(p)}{p^2}\right)\;\;\text{and}\;\;\sum_{e\geq 1}\frac{1}{p^{2e}} = \frac{1/p^2}{(1-1/p^2)}.
\end{equation*}
There are four such terms, together giving a contribution of
\begin{equation}\label{finalsumcontribution2}
    (\kappa^{(2)}_p)^{-1}\left(\frac{4p^2+6p+4\rho'_{(r,i)}(p)}{p^2(p^2-1)}\right).
\end{equation}
When $e_i\geq 1$ for exactly two $i\in\{0,1,2,3\}$, the minimum conditions on the $e_i$ dictate that
\begin{equation*}
    g(p^{e_0},p^{e_1},p^{e_2},p^{e_3}) = (\kappa^{(2)}_p)^{-1}\left(1+\frac{1}{p}+\frac{\rho'_{(r,i)}(p)}{2p^2}\right)\;\;\text{and}\;\;\sum_{e_i,e_j\geq 1}\frac{1}{p^{2e_i+2e_j}} = \left(\frac{1/p^2}{(1-1/p^2)}\right)^2.
\end{equation*}
There are four possible pairs $(i,j)\in\{0,1,2,3\}^2$ ($i< j$) in which this can occur, together giving a contribution of
\begin{equation}\label{finalsumcontribution3}
    (\kappa^{(2)}_p)^{-1}\left(\frac{4p^2+4p+2\rho'_{(r,i)}(p)}{p^2(p^2-1)^2}\right).
\end{equation}
The condition $\min(e_0,e_1)=\min(e_2,e_3)=0$ does not allow any contribution from $\b{e}\in(\N\cup\{0\})^4$ where three or four $e_i\geq 1$. Therefore, inputting \eqref{finalsumcontribution1},\eqref{finalsumcontribution2} and \eqref{finalsumcontribution3} into the right hand side of \eqref{sumoverbforconstant} for each prime $p\neq 2$ gives
\begin{align*}
    \sum_{\substack{\b{b}\in\N_{\textrm{odd}}^4\\ \eqref{bgcdconditions}}} &\hspace{-5pt}\frac{g(\b{b})}{b_0^2b_1^2b_2^2b_3^2} = \prod_{p\neq 2}(\kappa^{(2)}_p)^{-1}\hspace{-4pt}\left(\frac{p^2+2p+2\rho'_{(r,i)}(p)}{p^2}+\frac{4p^2+6p+4\rho'_{(r,i)}(p)}{p^2(p^2-1)}+\frac{4p^2+4p+2\rho'_{(r,i)}(p)}{p^2(p^2-1)^2}\right)\\
    &= \prod_{p\neq 2}(\kappa^{(2)}_p)^{-1}\hspace{-4pt}\left(\frac{p^6+2p^5+2(\rho'_{(r,i)}(p)+1)p^4+2p^3+p^2}{p^2(p^2-1)^2}\right).
\end{align*}
Inputting this into \eqref{constantcheckpoint-onlybleft} and using $(p^2-1)=p^2(1-1/p)(1+1/p)$, we have now proved the following:
\begin{proposition}\label{CONSTANTS_AS_EULER_PRODUCTS}
    For each $(r,i)\in\{(1,2),(1,3),(2,2),(2,3)\}$, the constant $\mathfrak{C}_{r,i}$ is equal to
    \[
    \frac{\rho_{(r,i)}}{256\pi^2}\prod_{p\neq 2}\left(1+\frac{1}{p}\right)^{-2}\left(1+\frac{2}{p}+\frac{2(\rho'_{(r,i)}(p)+1)}{p^2}+\frac{2}{p^3}+\frac{1}{p^4}\right).
    \]
\end{proposition}

\subsection{Conclusion of the proof of Theorem \ref{MAINTHEOREM}}\label{Conclusion}
We combine Propositions \ref{Simplificationsummary}, \ref{Vanish_Main_Term_Contribution} and \ref{Maintermri} to obtain
\begin{equation*}
    N_r(B) = \frac{(\mathfrak{C}_{r,2}+\mathfrak{C}_{r,3})B^2\log\log B}{\log B} + O_{A}\left(\frac{B^2\sqrt{\log\log B}}{\log B}\right)
\end{equation*}
for sufficiently large $A>0$. Then, by Proposition \ref{positivereal},
\begin{equation*}
    N(B) = \frac{(2\mathfrak{C}_{1,2}+2\mathfrak{C}_{1,3}+\mathfrak{C}_{2,2}+\mathfrak{C}_{2,3})B^2\log\log B}{\log B} + O_{A}\left(\frac{B^2\sqrt{\log\log B}}{\log B}\right).
\end{equation*}
Finally, using Proposition \ref{CONSTANTS_AS_EULER_PRODUCTS}, $(2\mathfrak{C}_{1,2}+2\mathfrak{C}_{1,3}+\mathfrak{C}_{2,2}+\mathfrak{C}_{2,3})$ is then
\begin{align*}
&\frac{935}{36\pi^2}\prod_{p\neq 2}\left(1+\frac{1}{p}\right)^{-2}\left(1+\frac{2}{p}+\frac{4}{p^2}+\frac{2}{p^3}+\frac{1}{p^4}\right)\\&+\frac{25}{36\pi^2}\prod_{p\neq 2}\left(1+\frac{1}{p}\right)^{-2}\left(1+\frac{2}{p}+\frac{2\left(1+\left(\frac{-1}{p}\right)\right)}{p^2}+\frac{2}{p^3}+\frac{1}{p^4}\right).
\end{align*}

\section{Appendix A: Notation summary}\label{Notation}
In this section, we give a summary of the primary notation used throughout the paper. We remark that much of the notation used throughout this paper has limited use and that such auxiliary notation has not been listed.
\begin{itemize}
    \item $N(B)$ is the counting problem of interest. It is first defined in \eqref{THEcountingproblem}. After \S\ref{GeoInput}, it is written as a sum over integers \eqref{countingproblem}. In Proposition \ref{positivereal}, we write $N(B)$ in terms of two sums over positive integers $N_1(B)$ and $N_2(B)$. The index $r\in\{1,2\}$ denotes whether we are counting in $N_1(B)$ or $N_2(B)$.
    \item $\mathcal{A}_1$ and $\mathcal{A}_2$ are subsets of $(\Z/8\Z)^{*4}$ satisfying certain congruence conditions (see \eqref{Q2condition1} and \eqref{Q2condition2}). They encode the $2$-adic solubility conditions of certain diagonal quadrics. Later, we define the sets $\mathcal{A}(\b{m},\boldsymbol{\sigma})$, which is either $\mathcal{A}_1$ or a permutation of $\mathcal{A}_2$ - see Lemma \ref{2adiclemma}. Generic elements in $\mathcal{A}(\b{m},\boldsymbol{\sigma})$ are denoted $\b{q}$.
    \item We have $\b{t},\b{b},\b{k},\b{l},\b{m},\b{d},\widetilde{\b{d}}\in\N^4$, $\b{K},\b{L}\in(\Z/8\Z)^{*4}$, $\boldsymbol{\sigma}\in\{0,1\}^{4}$. Here, $\b{t}$ are the initial variables appearing in the conics of the counting problems $N_r(B)$; the components of $\b{b}$ encode the square parts of the components of $\b{t}$, while products of the components of $\b{m}=(m_{02},m_{03},m_{12},m_{13})$, $\b{k}$, $\b{l}$ and $(2^{\sigma_0},2^{\sigma_1},2^{\sigma_2},2^{\sigma_3})$ comprise the square-free parts of $\b{t}$; we also have $(m_{ij})_{\textrm{odd}}=d_{ij}\widetilde{d}_{ij}$. Finally, $\b{K}$ and $\b{L}$ denote the fixed congruence classes of $\b{k}$ and $\b{l}$. We also use $M_0,M_1,M_2,M_3$ to compactify some of the notation. These are defined as
\begin{align*}
    M_0 = 2^{\sigma_0}m_{02}m_{03}b_0^2,\;
    M_1 = 2^{\sigma_1}m_{12}m_{13}b_1^2,\;
    M_2 = 2^{\sigma_2}m_{02}m_{12}b_2^2,\;
    M_3 = 2^{\sigma_3}m_{03}m_{13}b_3^2.
\end{align*}
    \item $\Theta_{r,1}(\b{d},\b{K},\boldsymbol{\sigma})$ denotes a product of characters of modulus $8$ which varies with $\b{d},\b{K}$, and $\boldsymbol{\sigma}$ - see \eqref{oddJacobisymbols}. In this expression, $f_r(\b{d},\b{K})$ is defined as in \eqref{reciprocityfactor}, and is obtained via quadratic reciprocity and the identity $\left(\frac{-1}{n}\right)=(-1)^{\frac{n-1}{2}}$. This reciprocity factor is the only part of $\Theta_{r,1}$ which is dependent on $r$.
    \item $\Theta_{2}(\b{d},\widetilde{\b{d}},\b{k},\b{l})$ is a product of Jacobi symbols in the indicated variables - see \eqref{oddJacobisymbols}. Throughout \S\ref{largeconductors} and \S\ref{smallconductors} we break up this Jacobi symbol in a variety of ways depending on the respective size of the variables in $\b{k}$ and $\b{l}$.
    \item $\mathcal{H}_i$ for $i\in\{1,\ldots,6\}$ indicate sub-regions of $\N^8$ which control the magnitude of the components of $\b{k}$ and $\b{l}$ - see \eqref{H1} - \eqref{H6}.
    \item $H_{r,i}(\b{d},\widetilde{\b{d}},\b{K},\b{L},B)$ are sums over the variables $(\b{k},\b{l})\in\mathcal{H}_i$ for fixed $\b{d},\widetilde{\b{d}},\b{K},\b{L}$ defined in \eqref{generalinnersum}. The summand is a product of the $\frac{\mu(2k_il_i)}{\tau(k_il_i)}$ and the product of Jacobi symbols $\Theta_2(\b{d},\widetilde{\b{d}},\b{k},\b{l})$. We also remark that these sums also depend on the choices of $\b{b},\b{m},\boldsymbol{\sigma}$ and $\b{q}$. The reason of highlighting the former variables and submerging the latter is that the behaviour of the $H_{r,i}$ sums (for example, whether they give an error term or a main term) are primarily determined by the region $\mathcal{H}_i$ and the choices of $\b{d},\widetilde{\b{d}},\b{K},\b{L}$. Although these are dissected in various ways throughout \S\ref{largeconductors}, \S\ref{smallconductors}, \S\ref{vanishingmainterm} and \S\ref{mainterm}, these $H_{r,i}$ should be considered the primary sums of interest.
    \item $\mathcal{M}_{r,i}(B,\b{b})$ for $i\in\{2,3\}$ and $\b{b}$ fixed, denotes a sum in the variables $\b{d},\widetilde{\b{d}},\b{K},\b{L},\b{m},\boldsymbol{\sigma}$ and $\b{q}$, restricted to those values for which $H_{r,i}$ will give a main term contribution. The summand over these variables is a product of $\frac{1}{\tau((m_{ij})_{\textrm{odd}})}=\frac{1}{\tau(d_{ij}\widetilde{d}_{ij})}$ and the sum $H_{r,i}(\b{d},\widetilde{\b{d}},\b{K},\b{L},B)$. In the case $i=3$, there is also a factor of $\Theta_{r,1}(\b{m}_{\textrm{odd}},\b{K},\boldsymbol{\sigma})$. These sums are defined in \eqref{MAINTERMR2} and \eqref{MAINTERMR3}.
    \item $\mathcal{V}_{r,i}(B,\b{b})$ for $i\in\{4,5\}$ and $\b{b}$ fixed, are sums over those $\b{m},\boldsymbol{\sigma},\b{q},\b{K}$ and $\b{L}$ restricted to those values for which the $H_{r,i}$ sums are roughly of order $B^2$. The summands over these variables are the corresponding $H_{r,i}$ sums multiplied by the even modulus characters $\Theta_{r,i}$. In \S\ref{vanishingmainterm}, it is shown that the presence of the non-square conditions \eqref{nonsquareconditions3} and the even characters $\Theta_{r,1}$ push these $\mathcal{V}_{r,i}$ sums into the error term of our counting function. These sums are defined in \eqref{VANISHMAINTERMR4} and \eqref{VANISHMAINTERMR5}.
    \item $\mathcal{E}_{r,i}(B,\b{b})$ for $i\in\mathcal\{1,\ldots,6\}$ and $\b{b}$ fixed, are sums over $\b{m},\boldsymbol{\sigma},\b{q},\b{K},\b{L},\b{d}$ and $\widetilde{\b{d}}$, where the $\b{K},\b{L},\b{d}$ and $\widetilde{\b{d}}$ are configured so that the inner $H_{r,i}$ sums are guaranteed to produce an error term. These sums are defined in \eqref{Errorterms1and2}, \eqref{Errorterms3and4} and \eqref{Errorterms5and6}; their summand is defined in \eqref{innersumforerrortermstwistedbyJacobi}.
    \item Finally, we denote by $N_{r,i}(B)$ as the sum over $\b{b}$ over the sum of those $\mathcal{M}_{r,i}$, $\mathcal{V}_{r,i}$ and $\mathcal{E}_{r,i}$ which share the same $r,i$ indexing. They are defined in \eqref{regionalmaintermcontributions}, \eqref{regionalvanishingmaintermcontributions}, and \eqref{regionalerrortermcontributions}. How these sums add together to produce asymptotics for is expressed in Proposition \ref{Simplificationsummary}.
\end{itemize}

\section{Appendix B: Code for computing $\Sigma_{r,i}$}\label{APPENDIX}
Below is the \verb|MAGMA| code used to compute the $\Sigma_{r,i}$ quantities from \S\ref{Sigmar2Comps} and \S\ref{Sigmar3Comps}. In this code, \verb|permutation[i]| corresponds to the $i$th element of $\{b,c,d,e,f,g\}$ from \S\ref{Sigmar3Comps}. 

\begin{verbatim}
Z := Integers();
Z8 := Integers(8);
Z8u := [Z8|1,3,5,7];
Z8u4 := CartesianPower(Z8u,4);
pairs := CartesianProduct([1,2],[3,4]);
evens := [<0,0>,<0,2>,<2,0>,<0,6>,<6,0>,<2,6>,<6,2>];

A1 := {q: q in Z8u4| (&or[q[x[1]] + q[x[2]] in [0,4]:x in pairs] or
<q[1]+q[2], q[3]+q[4]> in evens) and &*q eq 1};
Sigma13a := 4*&+[(-1)^(Z!(q[1]+q[2]+q[3]+q[4]) div 4): q in A1];
Sigma23a := 4*&+[(-1)^(Z!(q[1]+5*q[2]+q[3]+5*q[4]) div
4)*JacobiSymbol(-1,7*Z!q[2]*Z!q[3]): q in A1];


partition := [<i,j,k,l>: i, j, k, l in [1..4]|{i,j} in [{1,2},{3,4}] and
{k,l} in [{1,2},{3,4}] and {i,j,k,l} eq {1..4}];
A2 := {q: q in Z8u4| &or[(q[p[1]]+q[p[2]] eq 0 and (q[p[3]]+v)*(q[p[4]]+v) eq 0)
       or (q[p[1]]+q[p[2]] eq 2*v and (q[p[3]]+v)*(q[p[4]]+v) eq 0)
       :p in partition, v in Z8u] and &*q eq 1};

print "The set A1 is", A1;
print "The cardinality of A1 is", #A1;
print "The set A2 is", A2;
print "The cardinality of A2 is", #A2;

print "The Sigma13a sum is equal to", Sigma13a;
print "The Sigma23a sum is equal to", Sigma23a;

permutations := [<1,2,3,4>,<3,4,1,2>,<1,4,2,3>,<2,3,1,4>,<1,3,2,4>,
<2,4,1,3>];
denominator := [<1,2>,<1,2>,<1,4>,<1,4>,<1,3>,<1,3>];

for i in [1..6] do
perm := permutations[i];
print "Computations corresponding to the permutation of A2 defined by",
perm, ":";
denom := denominator[i];
Aperm := {q: q in Z8u4|<q[perm[1]],q[perm[2]],q[perm[3]],q[perm[4]]> in A2};
Sigma13Perm := 4*&+[(-1)^(Z!(q[1]+q[2]+q[3]+q[4]) div 4)*JacobiSymbol(2,
(Z!q[denom[1]]+8)*(Z!q[denom[2]]+8)): q in Aperm];
print "The Sigma13 sum corresponding to", perm, "is equal to", Sigma13Perm;
Sigma23Perm := 4*&+[(-1)^(Z!(q[1]+5*q[2]+q[3]+5*q[4]) div 4)*JacobiSymbol(2,
(Z!q[denom[1]]+8)*(Z!q[denom[2]]+8))*JacobiSymbol(-1,7*(Z!q[2])*(Z!q[3])):
q in Aperm];
print "The Sigma23 sum corresponding to", perm, "is equal to", Sigma23Perm;
end for;
\end{verbatim}

\end{document}